\documentclass[12pt]{amsart}
\usepackage{geometry}   
\usepackage[colorlinks,citecolor = red, linkcolor=blue,hyperindex]{hyperref}
\usepackage{euscript,eufrak,verbatim, mathrsfs}
\usepackage[psamsfonts]{amssymb}
\usepackage{bbm}
\usepackage{graphicx}
 \usepackage{float}
\usepackage{float, tikz}

\usepackage{extpfeil}

 \usepackage[all, cmtip]{xy}

\usepackage{upref, xcolor, dsfont}
\usepackage{amsfonts,amsmath,amstext,amsbsy, amsopn,amsthm}
\usepackage{enumerate}

\usepackage{url}

\usepackage{mathtools}
\usepackage{bookmark}

 \usepackage{euscript}
 \usepackage{times}
\usepackage{type1cm}        
\usepackage{multicol}        
\usepackage[bottom]{footmisc}
\newcommand{\normmm}[1]{{\left\vert\kern-0.25ex\left\vert\kern-0.25ex\left\vert #1 
    \right\vert\kern-0.25ex\right\vert\kern-0.25ex\right\vert}}

\newtheorem{theorem}{Theorem}[section]
\newtheorem*{theorem*}{Theorem B}
\newtheorem{lemma}[theorem]{Lemma}

\newtheorem{proposition}[theorem]{Proposition}
\newtheorem{corollary}[theorem]{Corollary}

\newtheorem*{observation*}{Observation}

\newtheorem*{assumption*}{Assumption}
\newtheorem*{question*}{Question}

\theoremstyle{definition}

\newtheorem*{definition*}{Definition}

\theoremstyle{remark}
\newtheorem{remark}{Remark}[section]
\newtheorem*{remark*}{Remark}
\newtheorem{example}[remark]{Example}

\newcommand{\R}{\mathbb{R}}
\newcommand{\N}{\mathbb{N}}
\newcommand{\Z}{\mathbb{Z}}
\newcommand{\Q}{\mathbb{Q}}

\newcommand{\C}{\mathbb{C}}
\newcommand{\E}{\mathbb{E}}
\newcommand{\T}{\mathbb{T}}
\newcommand{\PP}{\mathbb{P}}

\newcommand{\fX}{\mathbb{X}}
\newcommand{\MAA}{\mathscr{M}_{n-1}(W^{(2)})}
\newcommand{\MAB}{\mathscr{M}_{n}(W^{(2)})}

\newcommand{\Var}{\mathrm{Var}}
\newcommand{\Cov}{\mathrm{Cov}}
\newcommand{\supp}{\mathrm{supp}}
\newcommand{\sgn}{\mathrm{sgn}}

\newcommand{\Ab}{\mathcal{A}}

\newcommand{\ldim}{\underline{\dim}}

\newcommand{\an}{\text{\, and \,}}
\newcommand{\as}{\text{\, as \,}}

\newcommand{\indi}{\mathds{1}}

\newcommand{\NQ}{\mathcal{N}^{(\alpha, p, q)}}

\newcommand{\VU}{Y(u)}

\numberwithin{equation}{section}

\begin{document}

\title[Harmonic analysis of Mandelbrot cascades]{Harmonic analysis of Mandelbrot cascades\\ {\MakeLowercase{--- in the context of vector-valued martingales}}}

\author
{Xinxin Chen}
\address
{Xinxin CHEN: School of Mathematical Sciences
Beijing Normal (Teachers) University
Beijing 100875, China}
\email{xinxin.chen@bnu.edu.cn}

\author
{Yong Han}
\address
{Yong HAN: School of Mathematical Sciences, Shenzhen University, Shenzhen 518060, Guangdong, China}
\email{hanyong@szu.edu.cn}

\author
{Yanqi Qiu}
\address
{Yanqi Qiu: School of Fundamental Physics and Mathematical Sciences, HIAS, University of Chinese Academy of Sciences, Hangzhou 310024, China}
\email{yanqi.qiu@hotmail.com, yanqiqiu@ucas.ac.cn}

\author{Zipeng Wang}
\address{Zipeng WANG: College of Mathematics and Statistics, Chongqing University, Chongqing
401331, China}
\email{zipengwang2012@gmail.com, zipengwang@cqu.edu.cn}


\begin{abstract}

We solve a long-standing open problem of determining the Fourier dimension of the Mandelbrot canonical cascade measure (MCCM). This problem of significant interest was raised by Mandelbrot in 1976 and reiterated by Kahane in 1993. Specifically,
we derive the exact formula for the Fourier dimension of the MCCM for  random weights $W$ satisfying the condition $\E[W^t]<\infty$ for all $t>0$. As a corollary, we prove that  the MCCM is Salem if and only if the random weight has a specific two-point distribution.  In addition, we show that  the  MCCM is Rajchman with polynomial Fourier decay whenever the random weight satisfies $\E[W^{1+\delta}]<\infty$ for some $\delta>0$.  As a consequence, we discover that,   in the Biggins-Kyprianou's boundary case,  the Fourier dimension of the MCCM exhibits a second order phase transition at the inverse temperature $\beta = 1/2$; we establish the upper Frostman regularity for MCCM;  and we obtain a Fourier restriction estimate for MCCM.

The major novelty of this paper is the discovery of putting the fine analysis  of Fourier decay for multiplicative chaos measures into the theory of vector-valued martingales.  This new viewpoint is of fundamental importance in  the study of Fourier decay of multiplicative chaos measures. Indeed, in the sequel to this paper, combining the vector-valued martingale methods and  ideas from Littlewood-Paley theory, the precise Fourier dimensions will be established for various classical models of multiplicative chaos measures including GMC of all dimensions, microcanonical Mandelbrot cascades, Mandelbrot random coverings, as well as Fourier-Walsh analysis of these models.  
\end{abstract}


\maketitle

\setcounter{tocdepth}{2}
\tableofcontents

\setcounter{tocdepth}{0}
\setcounter{equation}{0}



\section{Introduction}\label{sec-intro}

\subsection{Main results}
Aiming to provide a critical mathematical discussion of  Kolmogorov's claims on the statistics of energy dissipation in intermittent turbulence (see Kolmogorov \cite{K61}, Landau-Lifshitz \cite{LK-59}, Obukhov \cite{O-62} and Yaglom  \cite{Yaglom}),     Mandelbrot  \cite{M72} introduced the log-normal multiplicative martingales to build random measures. Later in  \cite{M74,M74-2}, Mandelbrot  introduced the more tractable model of the  canonical cascade measures (MCCM),  which originally focuses on  the construction of related random measures on the unit interval  and becomes nowadays an important aspect of the active theory of multiplicative chaos.    
Mandelbrot in 1974 \cite{M74-2}  proved or conjectured various fundamental fractal properties of the MCCM,   many of which were rigorously proved by  Kahane-Peyri\`ere  \cite{Kahane-Peyriere-advance}  in 1976.

Following Mandelbrot,   we fix an integer $b\ge 2$ and an initial random weight $W$: 
\begin{align}\label{in-W}
 W\ge 0, \quad  \E[W]=1 \an \text{$W$ is non-constant}.
 \end{align}
Let  $\mu_\infty$ denote the MCCM, i.e., the random cascade measure on $[0,1]$ corresponding to the random weight $W$ and the $b$-adic structure of $[0,1]$. The formal definitions  of MCCM will be recalled in  \S \ref{sec-Mandelbrot-cascade}.  By  Kahane-Peyri\`ere  \cite{Kahane-Peyriere-advance}, $\mu_\infty$ is  non-degenerate  (that is, $\PP(\mu_\infty\ne 0) >0$) 
if and only if 
\begin{align}\label{H-dim}
D_H  := D_H(W, b) = 1- \E[W\log_b W]>0.
\end{align}
And,  under a mild assumption,  almost surely on $\{\mu_\infty\ne 0\}$,  the measure $\mu_\infty$ is unidimensional and the exact formula of its Hausdorff dimension is   $\dim_H(\mu_\infty)= D_H$.

{\flushleft \bf Mandelbrot-Kahane problem.}  
Denote  the Fourier transform of  $\mu_\infty$ by 
\[
\widehat{\mu}_\infty (\zeta)=\int_{[0,1]} e^{-2\pi i\zeta x}d\mu_\infty(x),\quad \zeta\in \mathbb{R}.
\]
In 1976, Mandelbrot \cite{M76} (and in his selected works \cite[pp. 402]{MA-book}) asked {\it  whether the Fourier coefficients of the cascade measure $\mu_\infty$ satisfy 
\text{$|\widehat{\mu}_{\infty}(k)|^2\sim k^{-D}$ as $k\to\infty$
}} for a suitable exponent $D$ and what is the relationship between the optimal $D$ and its Hausdorff dimension $D_H$.  This question was  reiterated  by Kahane \cite{Ka-93} in 1993, where he started the general open program to study the {\it decay behavior of Fourier transforms of natural random measures}. 

Within Kahane's framework, natural random measures include local times, time set images, and multiplications by random weights (in particular, the  MCCM). Kahane \cite{Ka-93} provided detailed analysis of MCCM and noted that,  {\it except for a few cases, the behavior of $\widehat{\mu}_\infty(\zeta)$ was not known.}  Kahane's program involves the study of : 
\begin{itemize}
	\item  {\it Rajchman property} :  Does $\widehat{\mu}_\infty (\zeta) \to 0$ as $\zeta\to\infty$ ? 
	\item  {\it Polynomial Rajchman property} : Does   $|\widehat{\mu}_\infty (\zeta)|^2 =  O ( |\zeta|^{-D})$ for a certain $D>0$ ?  
	\item {\it Exact Fourier dimension}:  What is the exact value of the Fourier dimension $\dim_F(\mu_\infty)$ ? Here $\dim_F(\mu_\infty)$ is  defined by 
\[
\dim_F(\mu_\infty): =\sup\{D \in[0,1]:  |\widehat{\mu}_\infty(\zeta)|^2 = O(|\zeta|^{-D})\}.
\]
	\item {\it Salem property} : 	When does  the equality $\dim_{F}(\mu_\infty)=\dim_H(\mu_\infty)$  hold ? 
\end{itemize} 

The Mandelbrot-Kahane problem remains open.   Mandelbrot \cite{M76}   wrote 
\begin{quote} 
The nature of the relationship
between $D$ and Fourier analysis has long been central to the fine mathematical
aspects of trigonometric series (see Kahane \& Salem 1963), but the
resulting theory is little known and little used beyond its original context. ... If, as I hope, the importance of fractal shapes in turbulence is recognized, the spectral analysis of the motion of fluids may at long last benefit from a number of pure mathematical results in harmonic analysis.
\end{quote}

Influenced by ideas from  the modern harmonic analysis (see in particular \S \ref{sec-outline}),   we  are now able to  solve this long-standing  problem (the result has been announced in \cite{CHQWa}).       The crucial use of the modern theory of vector-valued martingales distinguishes our current work with many works in the litterature on the Fourier decay of random measures.   To give  the reader a quick glance of our main result, consider  the case of  the log-normal weights ($
W=e^{\sigma N-\sigma^2/2}$ with $N\sim \mathcal{N}(0,1)$ and $\sigma \ge 0$) which were treated  in the origin of the cascade theory  (see \cite{M72}  and \cite[pp. 374]{KP-en}).  An  illustration of the Fourier and Hausdorff dimensions of  the cascade measure $\mu_\infty$ in this crucial case is given in Figure \ref{fig-log-normal}. 
\begin{figure}[H]
\centering
\begin{tikzpicture}[scale=2.3]
		\draw[->](0,0)--(0,0)node[below]{$0$}--(1.9,0)node[below]{$\sigma$};
				\draw(0.7411477585,0.75)node[right]{$\mathrm{dim}_H(\mu_\infty) = 1 - \frac{\sigma^2}{ 2 \log b}$};
				
		\draw[->](0,0)--(0,1.2) node[left]{$\mathrm{dim}_F(\mu_\infty)$};
		\draw[domain=0:sqrt(ln(3)/2),blue]plot(\x,{1-\x^2/ln(3)});
		\draw[domain=sqrt(ln(3)/2):sqrt(2*ln(3)),red]plot(\x,{1-\x^2/ln(3)+(1-sqrt(2*\x^2/ln(3)))^2});

		\draw[densely dashed](0,0.5)--((0.7411477585,0.5);
		\draw[densely dashed](0.7411477585,0)--(0.7411477585,0.5);
\draw [fill] (0.7411477585,0) circle [radius=0.4pt];
\draw [fill] (0.7411477585,0.5) circle [radius=0.4pt];
\draw [fill] (0,0.5) circle [radius=0.4pt];
\draw [fill] (0,1) circle [radius=0.4pt];
\draw (0.7411477585,0) node[below] {$\frac{1}{2}\sqrt{2 \log b}$};
\draw (2*0.7411477585,0) node[below] {$\sqrt{2 \log b}$};
\draw [fill] (2*0.7411477585,0) circle [radius=0.4pt];
\draw (0,0.5) node[left] {$\frac{1}{2}$};
\draw (0,1) node[left] {$1$};
\draw[domain=0:sqrt(2*ln(3)), densely dashed]plot(\x,{1-\x^2/(2 *ln(3))});
\draw[->](0,0)--(0,0)node[below]{$0$}; 
	\end{tikzpicture} \caption{$W=e^{\sigma N-\sigma^2/2}$ with $\E[W\log W]<\log b \Longleftrightarrow \sigma^2< 2 \log b$.}\label{fig-log-normal}
\end{figure}
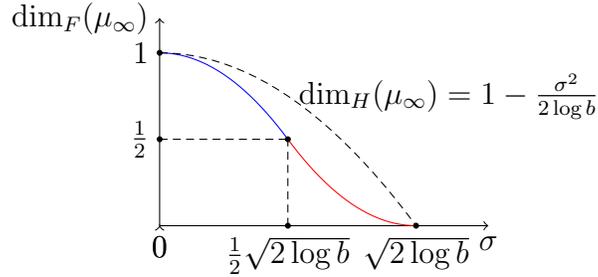

Now we proceed to state our main results.  In our first result, we are going to assume that $\E[W^t]<\infty$ for all $t>0$.  Hence, we may set 
\[
W^{(2)} =\frac{ W^2}{\E[W^2]}
\]
and 
define 
\begin{align}\label{def-DF}
D_F: = D_F(W, b) = \left\{
\begin{array}{lc}
\displaystyle 1 - \frac{\log \E[W^2]}{\log b} & \quad \text{if $\E [ W^{(2)}\log W^{(2)}] \le \log b$}
\\
\displaystyle 1 -   \inf_{1/2\le t\le 1} \frac{\log  \E[b^{1-t}W^{2t}]}{ t \log b} & \quad \text{if $\E [ W^{(2)}\log W^{(2)}] >\log b$}
\end{array}
\right. .
\end{align}
For further reference, we say that the random weight $W$ is in: 
\begin{itemize}
\item {\it squared sub-critical regime}  if $\E[W^{(2)}\log W^{(2)}]< \log b$; 
\item {\it squared critical regime} if $\E[W^{(2)}\log W^{(2)}] =  \log b$;  
\item {\it squared super-critical regime} if $\E[W^{(2)}\log W^{(2)}]>\log b$. 
\end{itemize}
 The  squared super-critical regime  involves the extremal point process of  {\it branching random walks (BRW)}. Hence, in this situation, as usual, we always make  the standard
 
{\bf \flushleft Assumption:}
$\log W$ is {\it non-lattice}: that is, it is not supported on an arithmetic progression.

\begin{theorem}[Exact Fourier dimension]\label{thm-fourier}
Assume that $\E[W\log W]<\log b$ and $\E[W^t]<\infty$ for all $t>0$.   Then $0< D_F<1$ and  almost surely on $\{\mu_\infty\ne 0\}$,  we have 
\begin{align}\label{FD-gen}
\dim_{F}(\mu_\infty)  = D_F.
\end{align}
\end{theorem}

\begin{remark}
Under the assumption that $W$  has sub-Gaussian tail $
\PP(W>t)\le 2 \exp (-ct^2)$,  the Fourier dimension \eqref{FD-gen} is established simultaneously and indepedently by Changhao Chen, Bing Li and Ville Suomala \cite{CLS24}. They later  replace this tail by the sub-exponential tail $\PP(W>t)\le 2 \exp (-ct)$. Note that, neither sub-Gaussian nor sub-exponential tail is satisfied by the log-normal weights.  
\end{remark}

\begin{remark}\label{rem-non-lattice}
The non-lattice assumption  on $\log W$  is needed only in the upper bound of $\dim_F(\mu_\infty)$ when $W$ is  in the squared super-critical regime.  In the sequel to this paper,   we will  on the one hand relax the condition $\E[W^t]<\infty$ for all $t>0$ to the much simpler condition $\E[W^2]<\infty$, and on the other hand  obtain the same upper bound of $\dim_F(\mu_\infty)$ without the non-lattice assumption on $\log W$   by establishing explicitly  the optimal  exponent  $\alpha$ of the  H\"older regularity: 
\begin{align}\label{Hol-optimal}
\mu_\infty([x,y]) = O (|x-y|^{\alpha})  \quad \text{for all  $0\le x<y \le 1$.} 
\end{align}
 See  Corollary \ref{cor-holder} below and its proof in \S \ref{sec-cor-holder} for more details. 
\end{remark}

\begin{corollary}[Salem property]\label{cor-salem}
Assume that $\E[W\log W]<\log b$ and $\E[W^t]<\infty$ for all $t>0$.  Then almost surely on $\{\mu_\infty\ne 0\}$,   $\mu_\infty$ is Salem if and only if $W$ has a two-point distribution : 
\begin{align}\label{2-pt-w}
\PP(W= x^{-1})= 1 - \PP(W=0) = x \quad \text{\,with\,} \quad  b^{-1}<x\le 1.
\end{align}
\end{corollary}

\begin{remark*}
By a rather different method,  the sufficient part of Corollary \ref{cor-salem} was  obtained by Shmerkin and Suomala \cite{SS18} following a construction of  {\L}aba and Pramanik \cite{LP09}. 
\end{remark*}

\begin{remark*}
When the random weights are given by \eqref{2-pt-w}, the corresponding multiplicative cascades  are called   {\it $\beta$-models}, which are also known as {\it absolute curdling} in intermittent turbulence \cite{M76}.   Note that in $\beta$-model case, since $W^{(2)} = W$, the random weight is automatically in the squared sub-critical regime once it satisfies Mandelbrot-Kahane's non-degeneracy condition. 
\end{remark*}

 \begin{corollary}[Log-normal weights]\label{cor-log-normal}
For the   log-normal random weights $W=e^{\sigma N-\sigma^2/2}$ with $0< \sigma < \sqrt{2 \log b}$,  we have 
 \begin{align}\label{log-normal-F}
 \dim_F(\mu_\infty)  =  \left\{
 \begin{array}{ll}
  \displaystyle 1- \frac{ \sigma^2}{\log b}  & \text{if $ \displaystyle \frac{\sigma^2}{\log b} \le  \frac{1}{2}$}
  \\  
    \displaystyle    2 \Big(1- \frac{\sigma}{\sqrt{2\log b}} \Big)^2=  1-  \frac{\sigma^2}{\log b}  +     \Big( 1 -   \sqrt{\frac{2\sigma^2}{\log b}} \Big)^2 & \text{if $ \displaystyle \frac{1}{2} <   \frac{\sigma^2}{\log b} <  2$}
 \end{array}
 \right.. 
 \end{align}
 \end{corollary}

\begin{theorem}[Polynomial Rajchman property]\label{thm-pRaj}
	Assume that $\E[W\log W]<\log b$ and $\E[W^{1+\delta}]<\infty$ for some $\delta>0$. Then almost surely on $\{\mu_\infty\ne 0\}$, we have $\dim_F(\mu_\infty)>0$. In other words,  $\mu_\infty$ has  polynomial Rajchman property. 
\end{theorem}

Theorem \ref{thm-pRaj}  reveals a  general common feature (proved in Theorem \ref{thm-preFdim} below) of the multiplicative  cascades (under the assumption that $\E[W^{1+\delta}]<\infty$ for some $\delta>0$) :
\[
\text{\it the action by the multiplicative cascade preserves the polynomial Rajchman property}. 
\]
Indeed,   the action by  multiplicative cascades on any Borel probability measure $\nu$ on $[0,1]$ gives a random cascade measure  $Q\nu$ (see \S \ref{sec-g-m} for the precise definition of $Q\nu$).    Suppose that the weight $W$ satisfies $\E[W^{1+\delta}]<\infty$ for some $\delta>0$ and  let $\ldim_{p}(\nu)$ denote the $L^p$-dimension introduced by R\'enyi \cite{Renyi} (see \S \ref{sec-def-Dp}).  Then  under a natural Mandelbrot-Kahane type non-degeneracy condition  
$
 \E[W\log_b W]< \lim_{p\to 1^{+}} \ldim_{p}(\nu), 
$ 
  we prove  in Theorem \ref{thm-preFdim} that, 
$$
\dim_F(\nu)>0   \Longrightarrow  \text{$\dim_F(Q\nu)> 0$ almost surely on $\{Q\nu\ne 0\}$}. 
$$
 In particular,  under the assumption  that  $p\mapsto \ldim_{p}(\nu)$ is a constant function  on $p\in (1, 2)$ with \footnote{The condition \eqref{typical-D} holds for typical  (in the sence of Baire's category) measures on Euclidean spaces, see \cite{Ol05}. }
\begin{align}\label{typical-D}
\ldim_{p}(\nu) =  D \quad  \text{for all $p \in (1,2)$,}
\end{align}   for the   log-normal random weights $W=e^{\sigma N-\sigma^2/2}$ with $0< \sigma < \sqrt{2  D \log b}$,  the multiplicative chaos measure $Q\nu$ is non-degenerate and almost surely on $\{Q\nu\ne 0\}$, we have  (see Corollary \ref{cor-nu})
 \[
 \dim_F(Q\nu)  \ge   \left\{
 \begin{array}{ll}
  \displaystyle  \Big(1-   \frac{\sigma^2}{D\log b}\Big)  \dim_F(\nu)  & \text{if $ \displaystyle \frac{\sigma^2}{\log b} \le  \frac{D}{2}$}
  \vspace{2mm}
  \\  
    \displaystyle    2 \Big(  1  - \frac{\sigma}{\sqrt{2D \log b}} \Big)^2 \dim_F(\nu) & \text{if $ \displaystyle \frac{D}{2} <   \frac{\sigma^2}{\log b} <  2D$}
 \end{array}
 \right.. 
 \]

{\flushleft \bf Failure of the Kahane's moment method in predicting the Fourier dimension.}
   In Kahane's open program, to predict and obtain a lower estimate of the Fourier dimension of a random measure, a useful method is to compute all the higher moments of the Fourier coefficients.   However,  in general,  this moment method fails in   predicting the Fourier dimension of MCCMs. We shall see that,   the second moment prediction may even differ from  the fourth moment prediction.  For instance,  the second moment method works  for predicting the Fourier dimension  works  when the random weights are in  the squared sub-critical or squared critical regimes but fails  in  the squared super-critical  regime.    Moreover,  in general, the fourth moment method fails in predicting the Fourier dimension even in the squared sub-critical regime.  More precisely, we have 
 \begin{itemize}
 \item {\it Second moment prediction in squared sub-critical or squared critical regimes}.   One shall see from  Lemma~\ref{lem-DF} below that,  under the assumption of Theorem~\ref{thm-fourier},   the following  implication holds: 
\[
\E [ W^{(2)}\log W^{(2)}] \le \log b \Longrightarrow \E[W^2]<b. 
\]
Moreover, if  the random weight $W$ satisfies $\E[W^2]<b$, then (see \S \ref{S-second moment} in the Appendix) 
	\begin{align}\label{2-m-intuition}
	\E[|\widehat{\mu}_\infty(k)|^2] \asymp k^{-(1-\frac{\log \E[W^2]}{\log b})} \quad \text{as $\N\ni k\to\infty$}.
	\end{align}
Therefore,  here the asymptotic in the average sense \eqref{2-m-intuition}  gives a correct prediction of the almost sure asymptotics of $|\widehat{\mu}_\infty(k)|$ and thus the Fourier dimension of $\mu_\infty$: 
\[
\dim_F(\mu_\infty)= D_F = 1-\frac{\log \E[W^2]}{\log b}. 
\]
 \item {\it Second moment prediction fails in  squared super-critical regime}.  In this case, we could have $\E[|\widehat{\mu}_\infty(k)|^2] = \infty$ since the condition $\E[W^2]<b$ is in general not guaranteed. Moreover,  even under the additional assumption $\E[W^2]<b$ (such random weights do exist, for instance, the log-normal weights $W = e^{\sigma N- \sigma^2/2}$ with $\frac{\log b}{2} < \sigma^2 <\log b$),  the   asymptotic  \eqref{2-m-intuition} does not provide a correct prediction of  the almost sure asymptotics of $|\widehat{\mu}_\infty(k)|$. Indeed,  in this case, we always have 
 \[
\dim_F(\mu_\infty)   = D_F >1-\frac{\log \E[W^2]}{\log b}.
 \]
 That is, the almost sure asymptotics  for  $|\widehat{\mu}_\infty(k)|$ has faster decay than the asymptotics in the average sense \eqref{2-m-intuition}. Note that when this happens, we must have,   for all  small enough $\varepsilon>0$, 
 \[
 \sup_{k\in \N}  k^{D_F - \varepsilon} |\widehat{\mu}_\infty(k)|^2 <\infty   \quad a.s.  \an \E \big[ \sup_{k\in \N}  k^{D_F - \varepsilon} |\widehat{\mu}_\infty(k)|^2\big] = \infty. 
 \]
And, in fact,  from our proof, we shall see that for some small enough $\delta>0$ depending on the random weight $W$ (see the stronger inequality \eqref{max-more} below),  
\[
 \E\big [ \big( \sup_{k\in \N}  k^{D_F - \varepsilon} |\widehat{\mu}_\infty(k)|^2\big)^{\frac{1+\delta}{2}}\big] < \infty.
\]
\item {\it Fourth moment prediction may fail even in squared sub-critical regime}.  Using the method developped in \cite{HQW-23},  for the log-normal weights $W= e^{\sigma N - \sigma^2/2}$ with  $0<\sigma^2< \frac{ \log b}{2}$ (thus $W$ are in the squared sub-critical regime), we shall show that 
\begin{align}\label{4-mom}
\E[|\widehat{\mu}_\infty(k)|^4] \asymp    \left\{  \begin{array}{ll}    
k^{- 2 ( 1 -  \frac{\log \E[W^2]}{\log b})} & \text{if $0<\sigma^2< \frac{ \log b}{4}$}
\\
 k^{- 2 ( 1 -  \frac{\log \E[W^2]}{\log b})}  \cdot  \log k   & \text{if $\sigma^2 =\frac{\log b}{4}$}
\\
 k^{- (3 - \frac{\log \E[W^4]}{\log b})}  &   \text{if $\frac{\log b}{4}<\sigma^2< \frac{ \log b}{2}$}
\end{array}
\right. . 
\end{align}
Thus the fourth moment method fails in  predicting the correct Fourier dimension.  Since the calculation \eqref{4-mom} is not needed in this paper,  we shall give its proof in a different paper. 
	\end{itemize}

{\flushleft \bf  Highlights in the proofs.} 
The strategy and the main novelty of this work is  to put the study of the fine analysis of  Fourier decay  of  random cascade measures in the framework of vector-valued martingale theory. This allows us to overcome the main difficulty in obtaining the sharp lower bound of  $\dim_F(\mu_\infty)$.        Below are the main ingredients of our proofs:
\begin{itemize}
	\item 
	The proofs of Theorems \ref{thm-fourier} and \ref{thm-pRaj}   are  based on an elementary observation of the natural connections between Fourier coefficients of the  random cascade measures and the vector-valued martingale theory.   

	\item The classical vector-valued martingale inequalities (including Pisier's martingale type inequalities,  vector-valued Burkholder inequalities, the recent $\ell^q$-vector-valued Burkholder-Rosenthdal inequalities of Dirksen and Yaroslavtsev \cite{PLMS-2019}, Bourgain-Stein inequalities, Kahane-Khintchine inequalities, etc)  will play natural roles in the lower bound of the Fourier dimension of $\mu_\infty$. 
	
	\item To obtain the upper bound of the Fourier dimension of $\mu_\infty$ is relatively simpler and we use the  fluctuations of martingales (in particular a freezing phenomenon \cite{D-S}) in BRW (which seem to be  of their own interests), as well as the stochastic self-similar structure of the Fourier coefficients   $\widehat{\mu}_\infty(b^n)$ on the subset of $b$-adic integers.     
\end{itemize}

In a previous work \cite{HQW-23}, we used  the scalar martingale inequalities due to Burkholder and Burkholder-Rosenthal  to  obtain the precise asymptotic growth rate of the $L^p$-moment  of MCCM  at the critical exponent \cite[Theorem 1.2]{HQW-23}, and also give an alternative proof of Kahane-Peyri\`ere’s $L^p$-boundedness condition \cite[Proposition 1.3]{HQW-23} for $p\ge 2$.

{\flushleft \bf Further applications of the vector-valued martingale method.}
Our  viewpoint of putting the fine analysis of Fourier decay of multiplicative chaos in the context of the vector-valued martingales, is of fundamental importance in  the study of Fourier decay of classical models of multiplicative chaos. Indeed, in the sequel to this paper, combining the vector-valued martingale methods and  ideas from Littlewood-Paley theory,  we are  able to  deal with Fourier decay of multiplicative chaos in the following situations: 
\begin{itemize}
\item the resolution of the Garban-Vargas conjecture for the Gaussian multiplicative chaos measures of all dimensions. See \cite{GV23} for the  backgrounds. 
	\item the cascade measures on higher dimensional cubes $[0,1]^d$,  including Marstrand-type projection problem.
	See \cite{RS14, SV14, FFJ15, FJ17, SS18, Mat19} for more  backgrounds on this topics. 
	\item the  precise decay of Fourier-Walsh coefficients of  the cascade measures on $[0,1]$. See \cite{PQ09} for its difference with the decay of usual Fourier coefficients. 
	\item the cascade measures on $[0,1]$ associated to a random Galton-Watson tree (see \cite{Way05}).
	\item the cascade measures on classical  self-similar sets such as Sierpinski's carpet or sponge (see \cite{Bar13, LBM15}); 
	\item the analysis of   oscillatory integrals of the cascade measures, for instance, 
	\[
	I(\lambda):=\int e^{i \lambda\Psi(x)} d\mu_\infty(x) \quad  \text{with phase function $\Psi: [0,1]\rightarrow \R^d$ and $\lambda>0$}.
	\] 
	See \cite[Chapters VIII and IX]{S93} for oscillatory integrals in classical harmonic analysis.  
\end{itemize}

In particular, in the sequel to this paper, the {\it exact values of Fourier dimensions} will be established (see \cite{CHQWb, LQT24, LQT25}) for various classical models of multiplicative chaos measures including GMC of all dimensions, microcanonical Mandelbrot cascades, Mandelbrot random coverings,  as well as the Fourier-Walsh analysis of these models.  

\subsection{Consequences and discussions}

{\flushleft\it A.  Common second order phase transition of the Fourier dimension.}

The Hausdorff dimension of MCCM is given in \eqref{H-dim} by $\dim_H(\mu_\infty) = D_H(W,b)= 1 - \E[W\log_b W]$ and  depends analytically on the random weight.  In sharp contrast, by Corollary \ref{cor-log-normal}, we see that, in the log-normal case,  the Fourier dimension of $\mu_\infty$ has a second order phase transition. This phase transition turns out to be a common feature for the MCCM and it would be interesting to have a physical interpretation of this phase transition. 

To see this, we consider the  random weights in  the Biggins-Kyprianou's boundary case (the precise definitions are recalled in \S \ref{sec-BK-transformation}): 
\[
W= \frac{e^{-\beta \xi}}{\E[e^{-\beta \xi}]}  \text{\,\, with \,\,}
\E[\xi e^{-\xi}] = 0, \quad \E[e^{-\xi}] = \frac{1}{b} \an 
\PP(\xi\in (-\infty, 0))> 0.
\]  
Assume that $\E[e^{-t \xi}]<\infty$ for all $t> 0$. Then the expression \eqref{def-DF} for the Fourier dimension $D_F$ has a simpler form and exhibits a second order phase transition. Indeed, 
\[
D_F =D_F(\beta, \xi)=  \left\{
\begin{array}{lc}
\displaystyle \frac{2 \psi(\beta)-  \psi(2\beta) }{\log b} & \quad \text{if $0<\beta\le 1/2$}
\\
\displaystyle  \frac{2\psi(\beta)}{\log b} & \quad \text{if $1/2<\beta<1$}
\end{array}
\right. ,
\]
where $\psi$ is the strict convex function defined by 
$
\psi(t) = \psi_\xi(t)=\log \E[be^{-t\xi}]. 
$
See Lemma \ref{lem-psi} below for the basic properties of the function $\psi$.  In particular, we have $\psi(1)= \psi'(1)=0$ and $\psi''(1)>0$.  Therefore, for fixed $\xi$,  the map $\beta \mapsto D_F(\beta, \xi)$ exhibits a second order phase transition at $\beta = 1/2$.

\medskip

{\flushleft \it  B. MCCM and GMC.} 

Based on  Mandelbrot’s groundbreaking work   \cite{M74,M74-2} on MCCM, Kahane introduced the influential theory of scalar $T$-martingales. Kahane's theory  confirmed  the three fundamental conjectures in MCCM  of Mandelbrot.  Furthermore, inspired by Mandelbrot's model of multiplicative chaos, Kahane introduced  the powerful theory of  Gaussian multiplicative chaos (GMC). A comparison between MCCM and GMC is investigated by Kahane in \cite{Ka85b}.  See Rhodes-Vargas' survey paper \cite{RV14} for more developments on GMC and its relations to MCCM. 

For the GMC setting, Falconer and Jin \cite{FJ19} obtained a non-trivial lower bound for the Fourier dimension of 2D GMC with small parameters $\gamma<\frac{1}{33}\sqrt{858-132\sqrt{34}}$.  In a recent pionner work \cite{GV23}, Garban and Vargas studied the Fourier decay of the standard GMC measure on the unit circle and  demonstrated \cite[Theorem~1.1]{GV23} that in the  sub-critical  GMC measure on the unit circle (denoted there by $M_\gamma$ with $\gamma<\sqrt{2}$) almost surely
is Rajchman. Moreover, they  proved  \cite[Theorem~1.2]{GV23} that $M_\gamma$
has a positive Fourier dimension for small parameters $\gamma<1/\sqrt{2}$: 
\[
\frac{1}{2} - \gamma^2 \le \dim_F(M_\gamma) \le 1 - \gamma^2<  \dim_H(M_\gamma) = 1 - \frac{\gamma^2}{2}. 
\]
 In addition,  from   \cite[Theorem~1.3]{GV23}, one naturally conjectures that  the Fourier dimension of $M_\gamma$ is given by $1- \gamma^2$ for small parameters $\gamma<1/\sqrt{3}$.  We believe that, similar to the MCCM setting in our work, in the GMC setting, there is also a second order phase transition for the Fourier dimension. 
 
  Other works on Fourier coefficients of GMC are included in the construction of the Virasoro algebra of Liouville conformal field theory and in the number theory or random matrix theory (see \cite{BGK,CN19} and their references).

\medskip
{\flushleft \it C. Upper Frostman regularity and KPZ relation. }

A non-negative Borel measure $\nu$ on $\R$ is said to be  $\gamma$-upper Frostman regular if 
\begin{align}\label{def-FA}
\sup  \big\{    \frac{\nu(I)}{|I|^\gamma}:   \text{$I$ are finite  intervals of $\R$}  \big\} <\infty. 
\end{align}

The upper-Frostman regularity for the classical sub-critical  GMC  measures on the unit circle was first established in  \cite{AJKS}. 
The upper Frostman regularity for MCCM  seems closely related to its KPZ relation established by Benjamini and Schramm in \cite{BS}. As pointed by Benjamini and Schramm, the following average H\"older estimate serves to {\em motivate KPZ relation for MCCM}: if $b =2$, then for any $x,y\in [0,1]$ and $s\in (0,1)$,
\[
\mathbb{E}[\mu_\infty([x,y])^s]\leq 8 |x-y|^{\phi(s)}, \quad \text{where $\phi(s)=s-\log_2 \E[W^s]$.}
\]
One can also find continuity estimates for other cascade models in \cite{MRV15}.

Theorem \ref{thm-fourier} has the following consequence on the stronger almost sure H\"older estimate of MCCM. 
 
\begin{corollary}\label{cor-holder}
Assume that $\E[W\log W]<\log b$ and $\E[W^t]<\infty$ for all $t>0$.   Then   almost surely on $\{\mu_\infty\ne 0\}$,  the measure $\mu_\infty$ is $\gamma$-upper Frostman regular for any 
\[
0\le \gamma< D_F/2. 
\]
\end{corollary}

\begin{remark*}
The  upper Frostman regularity  or equivalently the global H\"older regularity of Mandelbrot cascade  measures has been investigated in \cite[Theorem 3 and Theorem 4]{BKNSW14}.  Note also that it is known that the exponent of the  upper Frostman regularity is always at least half of the Fourier dimension, see \cite{LL24}. Here we include a proof for completeness.   
\end{remark*}

\medskip
{\flushleft \it D. The Fourier restriction estimate for MCCM. }

Theorem \ref{thm-fourier} and Corollary \ref{cor-holder} combined with the celebrated Fourier restriction estimate obtained in \cite[Theorem~4.1]{M00}   imply 

\begin{corollary}\label{cor-FR}
Assume that $\E[W\log W]<\log b$ and $\E[W^t]<\infty$ for all $t>0$.   Then   almost surely on $\{\mu_\infty\ne 0\}$,  the measure $\mu_\infty$ satisfies the following Fourier restriction estimate: for any $1\le r< \frac{4}{4-D_F}$,  there exists $C(r, \mu_\infty)>0$ such that for all $f\in L^r(\R)$, 
\begin{align}\label{F-R-ineq}
\| \widehat{f}\|_{L^2(\mu_\infty)} \le C(r, \mu_\infty) \| f\|_{L^r(\R)}. 
\end{align}
\end{corollary}

\begin{remark*}
In obtaining Corollary \ref{cor-FR}, we only use the Fourier dimension and \cite[Theorem~4.1]{M00}.  A natural question is: what is the optimal exponent $r\in (1,2)$ in the estimate \eqref{F-R-ineq}?  
\end{remark*}

\medskip
{\flushleft  \it E. Sharp estimate of certain maximal functions.}

 Theorem \ref{thm-fourier} can be reformulated as the following sharp  estimate of certain  maximal functions: it asserts that, almost surely on $\{\mu_\infty\ne 0\}$,  for any $\varepsilon >0$,  
\[
\sup_{n} |n|^{(D_F-\varepsilon)/2}  \cdot   |\widehat{\mu}_\infty(n)|<\infty \an 
 \sup_{n} |n|^{(D_F+\varepsilon)/2}  \cdot   |\widehat{\mu}_\infty(n)| = \infty. 
\]
These assertions follow from  the sharp estimate of certain  maximal function: under the assumption of Theorem~\ref{thm-fourier},   there exists $\delta= \delta(b, W)> 0$ such that 
for any $\varepsilon >0$, 
\[
\E\Big[\Big(\sup_{n} |n|^{(D_F-\varepsilon)/2}  \cdot   |\widehat{\mu}_\infty(n)|\Big)^{1+ \delta}\Big] <\infty. 
\]
Indeed, we do not know a direct proof of the above inequality.   Instead we shall prove more: there exists a sufficiently large $q>2$, 
\begin{align}\label{max-more}
\E\Big[\Big\{ \sum_{n}  \Big( |n|^{(D_F-\varepsilon)/2}  \cdot   |\widehat{\mu}_\infty(n)|\Big)^q\Big\}^{(1+ \delta)/q}\Big] <\infty. 
\end{align}

{\flushleft \it F.  Sobolev regularity.}

The inequality \eqref{max-more}  reflects certain Sobolev regularity  of the random measure $\mu_\infty$.  Indeed,  identify  $[0,1)$ with the unit circle $\T$, then the inequality \eqref{max-more} means that,  for any $0\le \beta< D_F/2$ and for suitable  parameters $p, q$ with $1< p< 2< q<\infty$,  we have $ \big\| \mathcal{F}\big[ (-\Delta)^{\beta} \mu_\infty \big] \big\|_{L^p(\PP; \, \ell^q)}<\infty$, where $(-\Delta)^{\beta}$ is the fractional Laplacian on  $\T$ and $\mathcal{F}$ denotes the Fourier transform on $\T$. 

For the parameters $p,q$ in the range $1<p<2<q<\infty$,  if we formally apply the standard Hausdorff-Young inequality for Fourier series, then \[ \big\| \mathcal{F}\big[ (-\Delta)^{\beta} \mu_\infty \big] \big\|_{L^p(\PP; \, \ell^q)} \lesssim_{p,q}    \big\| (-\Delta)^{\beta} \mu_\infty \big\|_{L^p(\PP; \, L^{q'} (\T))} \quad \text{with $\frac{1}{q} + \frac{1}{q'} = 1$}.\]However, almost surely on $\{\mu_\infty \ne 0\}$,  the measure $\mu_\infty$ is  not absolutely continuous with respect to the Lebesgue measure  and hence
$
\| (-\Delta)^{\beta} \mu_\infty \|_{L^p(\PP; \, L^{q'} (\T))} = \infty.
$
In a certain sense,  one may regard $\| \mathcal{F} [ (-\Delta)^{\beta} \mu_\infty  ] \|_{L^p(\PP; \, \ell^q)} $ as  a substitute of  $ \| (-\Delta)^{\beta} \mu_\infty \|_{L^p(\PP; \, L^{q'} (\T))}$.


\subsection{Main ingredients:  Vector-valued martingales meet BRW}

The lower and upper  bounds of the Fourier dimension $\dim_F(\mu_\infty)$ are given by the {\it vector-valued martingale} method and fluctuations of martingales in BRW respectively.   The proofs are outlined in the   schematic graph in   Figure \ref{fig-graph}.

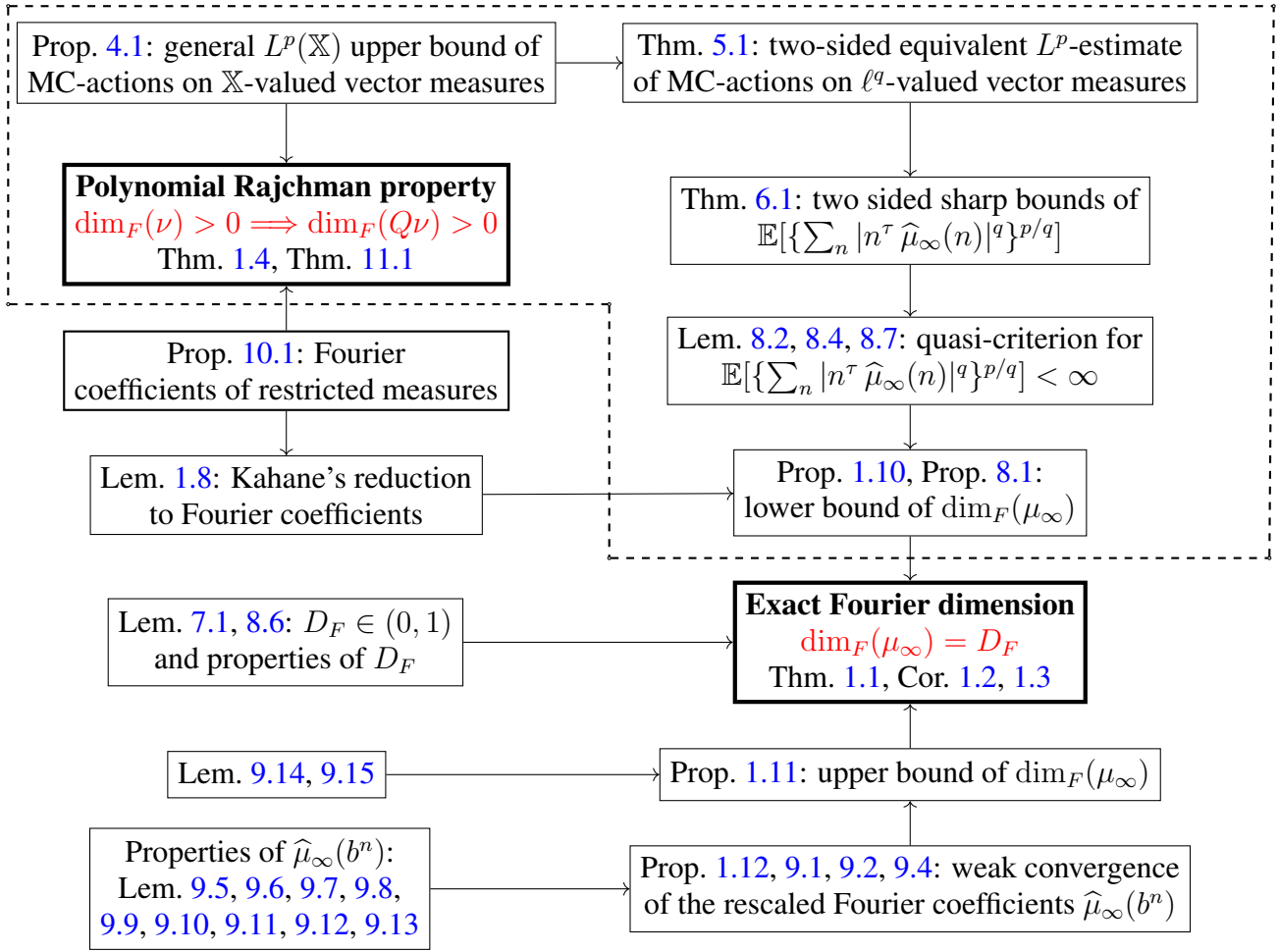
\begin{figure}[H]
\usetikzlibrary{arrows.meta, positioning}
\begin{tikzpicture}[node distance=2.5cm]
\node  (lemma82) [draw, ultra  thick, align=center,rectangle] {\textbf{Exact Fourier  dimension} \\${\color{red}\dim_F(\mu_\infty)=D_F}$ \\ Thm.~\ref{thm-fourier}, Cor.~\ref{cor-salem}, \ref{cor-log-normal} };
\node  (lemma71) [draw,align=center,rectangle, left=3.65cm of lemma82,node distance=3cm] {Lem.~\ref{lem-DF}, \ref{lem-psi-beta}: $D_F\in (0,1)$ \\and properties of $D_F$};
\node  (prop19) [draw,align=center,rectangle, below=0.6cm of lemma82,node distance=3cm] {Prop.~\ref{prop-to-infty}: upper bound of  $\dim_F(\mu_\infty)$};
\node  (lemma914) [draw,align=center,rectangle, left=3.7cm of prop19,node distance=3cm] { Lem.~\ref{lemma-cov-law-con-two},  \ref{lemma-cov-law-con-one}};         
 \node  (prop110) [draw,align=center,rectangle, below=0.6cm of prop19,node distance=3cm] { Prop.~\ref{prop-clt}, \ref{thm-sq-sub-critical},  \ref{thm-sq-critical}, \ref{thm-non-cri-super-critical}:   weak  convergence  \\ of 
 the  rescaled Fourier coefficients $\widehat{\mu}_\infty(b^n)$};   
\node  (lemma98) [draw,align=center,rectangle, left=2.7cm of prop110,node distance=3cm] {
Properties of $\widehat{\mu}_\infty(b^n)$:
\\
   Lem.~\ref{lemma-eqn-function-equality}, \ref{lemma-jihi}, \ref{coro-97},  
    \ref{lemma-mean-zero-s},  \\ \ref{lem-b2},  \ref{lemma-Yu-zero-s}, \ref{lemma-CLT-critical}, \ref{lemma-tightness}, \ref{lem-partition}
     };
\node  (prop81) [draw,align=center,rectangle, above=0.6cm of lemma82,node distance=3cm] {Prop.~\ref{prop-a-c}, Prop.~\ref{prop-sq-3case}: \\ lower bound of  $\dim_F(\mu_\infty)$};
\node  (lemma18) [draw,align=center,rectangle, above=0.6cm of prop81,node distance=3cm] {Lem.~\ref{prop-sub}, \ref{prop-sq-crit}, \ref{prop-super}: quasi-criterion for \\
 $\E[\{ \sum_{n} | n^{\tau}  \,\widehat{\mu}_\infty(n)|^q\}^{p/q} ]<\infty$
};
 \node (thm51) [draw, align=center, rectangle, above=2.9cm of lemma18, node distance=3cm]
 {Thm.~\ref{thm-precise}: two-sided  equivalent $L^p$-estimate 
 \\
 of   MC-actions on $\ell^q$-valued vector measures};
  \node (thm61) [draw, align=center, rectangle, below=1 cm of thm51, node distance=3cm] {Thm.~\ref{thm-CCM}: 
 two  sided sharp  bounds of  \\ 
 $\E[\{ \sum_{n} | n^{\tau}  \,\widehat{\mu}_\infty(n)|^q\}^{p/q} ]$ };
  \node (prop41) [draw, align=center, rectangle, left=0.9cm of thm51, node distance=3cm] {Prop.~\ref{prop-mart-type}: general  $L^p(\fX)$ upper bound of \\
  MC-actions  on $\fX$-valued vector measures};
  \node(A)[draw, circle, scale  = 0.1, above left = 0.2 cm and 0.1 cm of prop41]{};
  \node(B)[draw, circle, scale  = 0.1, right = 17.1 cm of A]{};
  \node(C)[draw, circle, scale  = 0.1, below  = 4.0cm of A]{};
    \node(D)[draw, circle, scale  = 0.1, right  = 8.1cm of C]{};
        \node(E)[draw, circle, scale  = 0.1, below  = 3.4cm of D]{};
                \node(F)[draw, circle, scale  = 0.1, right  = 8.9cm of E]{};
    \node (thm111) [draw, ultra thick, align=center, rectangle, below=0.8cm of prop41, node distance=3cm] { \textbf{Polynomial Rajchman property}\\ ${\color{red}\dim_F(\nu)>0   \Longrightarrow  \dim_F(Q\nu)> 0}$ \\Thm.~\ref{thm-pRaj}, Thm.~\ref{thm-preFdim}};
     \node (prop1001) [draw, thick, align=center, rectangle, below=0.6cm of thm111, node distance=3cm] {Prop.~\ref{lem-Fdecay}:
 Fourier\\  coefficients of restricted measures };  
\node  (prop101) [draw,align=center,rectangle, below =0.6cm of prop1001,node distance=3cm] {  Lem.~\ref{lem-trans-coef}:  Kahane's reduction\\  to Fourier coefficients};
\draw[->] (lemma71) -- (lemma82);
\draw[->] (prop19) -- (lemma82);
\draw[->] (lemma914) -- (prop19);
\draw[->] (prop110) -- (prop19);
\draw[->] (lemma98) -- (prop110);
\draw[->] (prop81) -- (lemma82);
\draw[->] (prop101) -- (prop81);
\draw[->] (lemma18) -- (prop81);
\draw[->] (thm61)--(lemma18);
\draw[->] (prop41)--(thm51);
\draw[->] (thm51)--(thm61);
 \draw[->] (prop41)--(thm111);
\draw[->] (prop1001)--(thm111);
\draw[->] (prop1001)--(prop101);
\draw[dashed, thick](A)--(B);
\draw[dashed, thick](A)--(C);
\draw[dashed, thick](C)--(D);
\draw[dashed, thick](D)--(E);
\draw[dashed, thick](E)--(F);
\draw[dashed, thick](F)--(B);
\end{tikzpicture}\caption{The schematic graph of the proofs of the main results:  the vector-valued martingale theory is mainly applied in the part enclosed by the dashed lines.}\label{fig-graph}
\end{figure}

{\flushleft \bf Part 0. Kahane's reduction to Fourier coefficients.}

Throughout the whole paper, when dealing with the Fourier dimension of a measure, we shall always assume that the measure is not the zero measure.

Following Kahane, we may reduce the study of the decay behavior of the Fourier transform of $\widehat{\mu}_\infty(\zeta)$ as $\zeta \to\infty$ to that of its Fourier coefficients $\widehat{\mu}_\infty(n/2)$ or $\widehat{\mu}_\infty(n)$ on the half-integers or on the integers as $n\to\infty$.   Indeed,  there are  two alternative ways to realize such reduction, both of which are essentially due to  Kahane  \cite[Chapter 17, Lemma 1]{Kahane-book} (for notational ease, we shall follow the second one, see Part 1 below of this subsection for the details): 
\begin{itemize}
\item  The first one is:  we use Lemma~\ref{lem-trans-coef-bis} to reduce the study of decay behavior of  $\widehat{\mu}_\infty(\zeta)$  to that of $\widehat{\mu}_\infty(n/2)$ and then  replace all  the estimate in Part 1  below on the random sequence  $(\widehat{\mu}_\infty(n))_{n\ge 1}$ by that on the random sequence  $(\widehat{\mu}_\infty(n/2))_{n\ge 1}$. 
\item The second one is: we use Lemma~\ref{lem-trans-coef} to reduce the study of decay behavior of  $\widehat{\mu}_\infty(\zeta)$  to that of $\widehat{\mu}_\infty(n)$ and then in Part 1 below, we estimate  the random sequence  $(\widehat{\mu}_\infty(n))_{n\ge 1}$. 
\end{itemize}

\begin{lemma}\label{lem-trans-coef-bis}
Let $\mu$ be a finite positive Borel  measure on $[0,1]$.  Then 
\[
\dim_F(\mu)= \sup \Big\{ D \in [0,1):    \text{$|\widehat{\mu}(n/2)|^2  = O(n^{-D})$ as $\N \ni n \to \infty$}    \Big\}. 
\]
\end{lemma}

\begin{lemma}\label{lem-trans-coef}
Let $\mu$ be a  finite positive  Borel  measure on $[0,1]$.  Then 
\[
\dim_F(\mu)= \sup \{ D \in [0,1):    \text{$|\widehat{\mu}(n)|^2  = O(n^{-D})$ as $\N \ni n \to \infty$}    \}. 
\]
\end{lemma}

 It is worthwhile to note that Lemma~\ref{lem-trans-coef}  seems to  have been overseen in the litterature where the measures were usually assumed to be supported in $[\delta, 1- \delta]$ for some $0< \delta<1/2$.   In Remark~\ref{rem-ka} below, we briefly explain how this assumption on the support is dropped here. 
 \begin{remark}\label{rem-ka}
  Kahane \cite[Chapter 17, Lemma 1]{Kahane-book} actually proved  more. Here we present a version that is suitable for our purpose. Let $\mu$ be a finite positive Borel measure supported on $[\delta, 1 - \delta]$ for some $0< \delta<1/2$ and let  $\Psi(t)$ be  a positive function on $(0, \infty)$ which is decreasing on $[t_0, \infty)$ for some $t_0>0$ and  such that $\Psi (t/2)=O(\Psi(t))$ as $t\to\infty$. Then 
 \begin{align}\label{Kah-rel}
  \widehat{\mu}(n)=O(\Psi(|n|))  \text{\,\,as\,\,} n \rightarrow \infty  \Longrightarrow \widehat{\mu}(\zeta)=O(\Psi(|\zeta|))   \text{\,\,as\,\,} \zeta \rightarrow \infty. 
  \end{align}
 Clearly, the support assumption of Kahane can be replaced by the assumption that $\mu$ is supported on a sub-interval of length strictly less than $1$.  Lemma \ref{lem-trans-coef-bis}  is obtained as follows: we may assume that $\mu$ is a probability measure and $X \stackrel{d}{=}\PP_X = \mu$.  Then  $Z = X/2$ is supported on $[0, 1/2]$  and $\widehat{\mu}(\zeta)  = \widehat{\PP_X}(\zeta)= \E[\exp(-i 2 \pi X \zeta) ]= \E[\exp(-i 2 \pi Z 2 \zeta) ] = \widehat{\PP_Z}(2\zeta)$.  Since $\PP_Z$ is supported on the interval $[0,1/2]$ of length strictly less than $1$, we may apply Kahane's result \eqref{Kah-rel} and obtain 
 \[
 \widehat{\mu}(n/2) = \widehat{\PP_Z}(n) = O(n^{-D/2})  \text{\,\,as\,\,} n \rightarrow \infty   \Longrightarrow  \widehat{\mu}(\zeta/2) = \widehat{\PP_Z}(\zeta) = O(\zeta^{-D/2})  \text{\,\,as\,\,} \zeta \rightarrow \infty. 
 \]
Lemma~\ref{lem-trans-coef} is  proved in a slightly different way.  Indeed, we shall prove that, for $0<D<1$, 
\[
 \widehat{\mu}(n) = O(n^{-D/2})  \text{\,\,as\,\,} n \rightarrow \infty   \Longrightarrow  \widehat{\mu}(\zeta) = O(|\zeta|^{-D/2} \log |\zeta|)  \text{\,\,as\,\,} \zeta \rightarrow \infty. 
\]
More precisely,   in  Proposition~\ref{lem-Fdecay},  we prove that  if $|\widehat{\mu}(n)|  = O(n^{-\kappa})$ with $\kappa \in (0, 1]$, then 
\[
|\widehat{\mu}_0(n)|  = O(n^{-\kappa} \log n) \an  |\widehat{\mu}_1(n)| = O(n^{-\kappa} \log n),
\]
where $\mu_0 = \mu|_{[0,1/2)}$ and $\mu_1= \mu|_{[1/2, 1]}$ are the restrictions of $\mu$ on two sub-intervals of length $1/2$. We can then apply Kahane's result  \eqref{Kah-rel}  to $\mu_0$ and $\mu_1$ respectively and obtain 
\[
|\widehat{\mu}(\zeta)| \le    |\widehat{\mu}_0(\zeta)|  + |\widehat{\mu}_1(\zeta)|   = O  \big( |\zeta|^{-\kappa} \log (|\zeta|)\big) \as \zeta \to \infty.
\]
\end{remark}

As an immediate consequence of Lemma \ref{lem-trans-coef}, we have 

\begin{corollary}\label{cor-series-dim}
Let $\mu$ be a finite positive Borel  measure on $[0,1]$.  Then 
\[
\dim_F(\mu)= \sup \Big\{D \in [0,1):    \text{$  \sum_{n\ge 1 } | n^{\frac{D}{2}}  \cdot \widehat{\mu}(n)|^q <\infty$ for some $q>2$}    \Big\}. 
\]
\end{corollary}

Therefore,  the central idea   in our proof of Theorem \ref{thm-fourier} is to find the optimal exponent $\tau>0$ such that 
\[
\sum_{n\ge 1} | n^{\tau}  \cdot \widehat{\mu}_\infty(n)|^q<\infty
\]
holds for large enough $q>2$ (here $q$ can not be taken in $(1, 2]$ since otherwise $\mu_\infty$ would be absolutely continuous, which violates the standard known fact).  Indeed,  set
\begin{align}\label{def-tau-c}
\tau_c(\mu_\infty): =  \sup \Big\{\tau \in \R: \sum_{n\ge 1} | n^{\tau}  \cdot \widehat{\mu}_\infty(n)|^q<\infty \text{\, for some $q>2$} \Big\}. 
\end{align}
Then by Corollary \ref{cor-series-dim},  for  $\mu_\infty\ne 0$, 
\[
\dim_F(\mu_\infty) = \min( 2 \tau_c(\mu_\infty), 1).
\]

{\flushleft \bf    Part 1.  Vector-valued martingales.}

  At first glance, one may think that the optimal exponent $\tau_c(\mu_\infty)$ can be obtained by finding a  sharp sufficient condition or a  criterion for 
\begin{align}\label{q-moment}
\E\Big[ \sum_{n\ge 1} | n^{\tau}  \cdot \widehat{\mu}_\infty(n)|^q\Big]<\infty. 
\end{align}
And  this is indeed the common strategy for computing the Fourier dimension for  many random measures in the litterature.  However, the estimate \eqref{q-moment} does not work in our setting. Instead,  we shall find optimal $\tau$ such that for large enough $q>2$ and small enough $\varepsilon> 0$, 
\begin{align}\label{vec-moment}
\E\Big[\Big\{ \sum_{n\ge 1} | n^{\tau}  \cdot \widehat{\mu}_\infty(n)|^q\Big\}^{\frac{1+\varepsilon}{q}} \Big]<\infty.
\end{align}

The estimate \eqref{q-moment} seems to be  much simpler than that of \eqref{vec-moment}, it can be  reduced to the estimate of the asymptotic behavior $\E[|\widehat{\mu}_\infty(n)|^q]$.  Namely, we compute separately all the $q$-moments of $\widehat{\mu}_\infty(n)$ for each $n$.  In a separate work,  based on a method developed in our previous work \cite{HQW-23},  for fixed $q>2$, we shall give  the precise asymptotics of $\E[|\widehat{\mu}_\infty(n)|^q]$  as $n\to\infty$ and thus give a simple criterion for \eqref{q-moment}.  However,  to use the estimate \eqref{q-moment} for finding the optimal exponent $\tau_c(\mu_\infty)$ would require  very restrictive condition on $W$ (roughly speaking, it would require that  $W$ is bounded and a strong condition on the $L^\infty$-bound of $W$).

Therefore, instead of considering coordinatewisely all the random  Fourier coefficients $\widehat{\mu}_\infty(n)$, we shall consider $\widehat{\mu}_\infty$ as a whole random object. Namely, $\widehat{\mu}_\infty$  is identified with the random  vector : 
\[
\widehat{\mu}_\infty  = \big(\widehat{\mu}_\infty (n) \big)_{ n\in \N}. 
\]
Note that,  since $\mu_\infty$ is a non-negative measure, $\widehat{\mu}_\infty(-n) = \overline{\widehat{\mu}_\infty(n)}$,  we may only consider its Fourier coefficients on non-negative integers. Moreover, it is convenient to work with mean zero random vector, hence we shall only consider the Fourier coefficients $\widehat{\mu}_\infty(n)$ for $n\ge 1$ (indeed,  $\E[\widehat{\mu}_\infty(n)] =  \widehat{\mathds{1}_{[0,1]}}(n) = 0$ for all integers $n\ne 0$).  

For this random vector $\widehat{\mu}_\infty$, we can define a series of norms as follows. For any $\alpha \ge  0$ and $p, q\ge 1$, define the $(\alpha,p,q)$-norm of $\mu_\infty$ by
\begin{align}\label{eq-npq}
\mathcal{N}^{(\alpha, p, q)}(\mu_\infty): =   \Big(\E\Big[  \Big\{ \sum_{n \ge 1}    \big| n^{\alpha} \cdot \widehat{\mu}_\infty(n)\big|^q  \Big\}^{\frac{p}{q}}\Big] \Big)^{\frac{1}{p}}   = \Big\|  \big( n^{\alpha} \cdot \widehat{\mu}_\infty(n) \big)_{n \ge 1}  \Big\|_{L^p(\PP; \, \ell^q)}\in [0, +\infty].
\end{align}
Then for appropriate parameters $\alpha, p, q$,  we shall give in Theorem \ref{thm-CCM} below a sharp two-sided estimate (where $\mathring{W} = W- \E[W])$
\begin{align}\label{2-side-es}
 \frac{\|\mathring{W}\|_1}{\|W\|_2}  \cdot  \big\|\mathcal{R}\big\|_{\ell^{\frac{p}{q}}}^{\frac{1}{p}}
 \lesssim_{b, \alpha, p, q}   \NQ(\mu_\infty) 
\lesssim_{b, \alpha, p, q}  \|\mathring{W}\|_q     \cdot \big\|\mathcal{R}\big\|_{\ell^1}^{\frac{1}{p}},
\end{align}
where $\mathcal{R}$ (see  \eqref{def-Rn} for its precise definition) is an explicitly given  scalar sequence  related to the MC-action  with the initial random weight $W^{(2)} = W^2/\E[W^2]$.  The two-sided estimate \eqref{2-side-es} is sharp in the sense that, except in a certain critical case (which can be further removed by varying the parameters $p, q$), the scalar sequence   $\mathcal{R}$ has exponential decay. This  allows the two-sided estimate \eqref{2-side-es} to capture the  quasi-criterion (namely, almost necessary and sufficient condition) for the finiteness of the norm $\mathcal{N}^{(\alpha, p, q)}(\mu_\infty)$, see Lemma ~\ref{prop-sub},  Lemma \ref{prop-sq-crit} and Lemma \ref{prop-super}.  To give the reader a quick insight,  consider  the log-normal weight $
W=e^{\sigma N-\sigma^2/2}$ with $N\sim \mathcal{N}(0,1)$ and $0<\sigma^2<\frac{1}{2}\log b$ (i.e.,  $W$ is in the squared sub-critical regime). Then the two-sided estimate \eqref{2-side-es} reads as (with $c_1, c_2$ two finite constants depending on $\alpha, p, q, b$)
\begin{align*}
c_1\left(\sum_{n=1}^\infty e^{\frac{np^2\,\log b}{q}\left[\alpha-\frac{1}{2}\left(1-\frac{\sigma^2}{\log b}-\frac{2}{q}\right)\right] }\right)^{\frac{q}{p^2}}
  \leq \mathcal{N}^{(\alpha, p, q)}\left(\mu_{\infty}\right)\leq c_2\left(\sum_{n=1}^\infty e^{np\,\log b\left[\alpha-\frac{1}{2}\left(1-\frac{\sigma^2}{\log b}-\frac{2}{q}\right)\right] }\right)^{\frac{1}{p}}.
\end{align*}
Hence  in this case, for any $1<p\le 2\le q<\infty$ and any $\sigma \in (0, \sqrt{\log b/2})$, we can obtain  the following equivalence
\[
\mathcal{N}^{(\alpha, p, q)}<\infty \Longleftrightarrow  \alpha-\frac{1}{2}\left(1-\frac{\sigma^2}{\log b}-\frac{2}{q}\right)<0.
\]

In the next step, we define a critical exponent $\alpha_c$ by 
\begin{align}\label{def-a-c}
\alpha_c: =  \sup \Big\{\alpha \in \R:  \mathcal{N}^{(\alpha, p, q)}(\mu_\infty) <\infty \text{\, for some $1<p<2<q<\infty$} \Big\}. 
\end{align}
Almost surely, $\tau_c(\mu_\infty) \ge \alpha_c$. The sharp two-sided estimate \eqref{2-side-es} then  gives

\begin{proposition}\label{prop-a-c}
If $\E[W\log W]<\log b$ and $\E[W^t]<\infty$ for all $t>0$,   then  almost surely, 
\[
\tau_c(\mu_\infty) \ge \alpha_c \ge D_F/2, 
\]
and hence, almost surely on $\{\mu_\infty \ne 0\}$, 
\[
\dim_F(\mu_\infty) \ge D_F.
\]
\end{proposition}

{\flushleft \bf  Part 2. Fluctuations of martingales in BRW.}

The sharp upper bound of $\tau_c(\mu_\infty)$ relies on the  weak convergence type theorems from BRW. 
\begin{proposition}\label{prop-to-infty}
Assume that $\E[W\log W]<\log b$ and $\E[W^t]<\infty$ for all $t>0$.   Then  almost surely on $\{\mu_\infty\ne 0\}$,   for any $\varepsilon>0$, there exists some subsequence $(n_k)$  depending on $\varepsilon$,  such that 
\[
\lim_{k\to\infty} b^{\frac{n_k (D_F+\varepsilon)}{2}} |\widehat{\mu}_\infty(b^{n_k})| =  \infty. 
\]
\end{proposition}

 Proposition~\ref{prop-to-infty} will be derived as a consequence of  the following Proposition~\ref{prop-clt} on the precise weak convergence limit of the rescaled  Fourier coefficients $\widehat{\mu}_\infty(b^n)$.     In the proof of  Proposition~\ref{prop-clt}, the non-lattice assumption on $\log W$  is used when $W$ is in the squared super-critical regime. 
 
 \begin{remark*}
As we explained in Remark \ref{rem-non-lattice},  in the sequel to this paper,  by using the optimal estimate of the H\"older continuity exponent \eqref{Hol-optimal}, we will prove that Proposition~\ref{prop-to-infty} holds   without the non-lattice assumption on $\log W$. 
 \end{remark*}

\begin{remark*}
As suggested to us by a referee, the quantity $D_F$ in \eqref{def-DF} in fact coincides with the standard correlation dimension $\dim_2(\mu_\infty)$ (see \eqref{def-D-p} below for its precise definition), which is essentially known to the experts and follows from  Molchan \cite{Mol96} and Barral-Gon\c{c}alves \cite{BG11}.  Therefore, the upper bound of $\dim_F(\mu_\infty)\le D_F$ is straightforward by using the standard inequality $\dim_F(\mu_\infty)\le \dim_2(\mu_\infty)$.    We noticed that the latest version of Chen-Li-Suomala \cite{CLS24}, the authors applied the aforementioned method to give the details of  $\dim_2(\mu_\infty) = D_F$. Therefore, to avoid duplication, we decide to keep our original method using the fluctuation  of $\widehat{\mu}_\infty(b^n)$, which we believe may have its own interests.
\end{remark*}

 There are three different rescalings of the Fourier coefficients $\widehat{\mu}_\infty(b^n)$ which are briefly explained as follows.  Recall that we denote $W^{(2)}  = W^2/\E[W^2]$.
This random weight  $W^{(2)}$ can be used to   generate a new MC-action, which is referred to as the {\it squared MC-action}. 
The quantity $\E[ W^{(2)} \log W^{(2)}]$ is important in our work. Indeed,  in the proof of Proposition~\ref{prop-clt} below,  we need to use the asymptotic behavior of the {\it small moments}, namely the moments of order $p/2$ with $1<p<2$,  of certain {\it squared-martingale} (see  \eqref{eqn-def-mnab} below for its precise definition)  related to the squared MC-action. In particular,  the precise asymptotics of small moments of {\it  additive martingales} of critical or supercritical parameters in BRW will play a key role. Therefore,    from Mandelbrot-Kahane's non-degeneracy condition, it is not surprising that we need to deal with three cases:  squared sub-critical regime, squared critical regime and squared super-critical regime. 

\vspace{2mm}

We now proceed to the precise statement. 

Recall that  the standard complex Gaussian random variable is given by 
\begin{align}\label{def-cg}
 \mathcal{N}_{\C}(0,1)  =   \frac{X+iY}{\sqrt{2}} 
\end{align}
with $X$ and $Y$  two independent standard real Gaussian $ \mathcal{N}(0,1)$ random variables.

Let $\mathscr{M}_n(W^{(2)})$ denote the total mass of the random measure  with respect to the random weight $W^{(2)}$  after $n$-step multiplicative cascade operation and let $\mathscr{M}_\infty(W^{(2)})$ denote the total mass of the corresponding limiting random cascade measure, see \S \ref{sec-sq-mc}  and \eqref{eqn-def-mnab} for their precise definitions.   Note that   (see Remark \ref{rk-w-w2}),  in the squared sub-critical regime,
\[
\PP(\mathscr{M}_\infty(W^{(2)}) > 0  | \mu_\infty\ne 0)=1.
\] 

\begin{proposition}\label{prop-clt}
Assume that $\E[W\log W]<\log b$ and $\E[W^t]<\infty$ for all $t>0$.  We have the following convergences in distribution  under $\PP^* = \PP(\cdot | \mu_\infty\ne 0)$:
\begin{itemize}
\item Squared sub-critical regime: if $\E[ W^{(2)} \log W^{(2)}] < \log b$ and $b\ge 3$, then 
\[
b^{\frac{nD_F}{2}} \cdot  \widehat{\mu}_\infty(b^n)  \xrightarrow[n\to\infty]{\quad d \quad }      \sqrt{ \E[|\widehat{\mu}_\infty(1)|^2] \cdot  \mathscr{M}_\infty(W^{(2)})} \cdot \mathcal{N}_{\C}(0,1),
\]
where $\mathcal{N}_{\C}(0,1)$ is  independent of $\mathscr{M}_\infty(W^{(2)})$ and $\E[|\widehat{\mu}_\infty(1)|^2]\in (0,\infty) $ is given in \eqref{eqn-mu-s-square-s=1}. 
\item  Squared critical regime:  if $\E[ W^{(2)} \log W^{(2)}]  = \log b$ and $b \ge 3$, then 
\[
n^{\frac{1}{4}} \cdot b^{\frac{nD_F}{2}} \cdot  \widehat{\mu}_\infty(b^n)  \xrightarrow[n\to\infty]{\quad d \quad }      \sqrt{ \E[|\widehat{\mu}_\infty(1)|^2] \cdot  \sqrt{\frac{2}{\pi \sigma^2}}  \mathscr{D}_\infty(W^{(2)})} \cdot \mathcal{N}_{\C}(0,1),
\]
where $\mathscr{D}_\infty(W^{(2)})$ is the a.s. limit of the derivative martingale of BRW,  coming from the Seneta-Heyde norming limit of the $\mathscr{M}_n(W^{(2)})$  proved in \cite{Aideon-Shi} as follows: 
\[
 \sqrt{n} \cdot  \mathscr{M}_n(W^{(2)}) \xrightarrow[n\to\infty]{\text{in probability}} \mathscr{D}_\infty(W^{(2)}) \sqrt{\frac{2}{\pi \sigma^2}} 
\]
   with $\sigma^2\in(0,\infty)$ given in \eqref{def-BRW-var} and  $\mathcal{N}_\C(0,1)$  independent of $\mathscr{D}_\infty(W^{(2)})$.
   \item Squared super-critical regime:  if $\E[ W^{(2)} \log W^{(2)}] > \log b$ and $\log W$ is non-lattice, then there exists  a unique constant $\beta\in (1/2,1)$ depending only on $W$ (see \eqref{def-beta}), such that  
 \begin{align*}
b^{\frac{n D_F}{2}}n^{\frac{3}{2}\beta} \widehat{\mu}_{\infty}\left(b^{n}\right)\xrightarrow[n\rightarrow\infty]{d}  \mathscr{D}_\infty(W^{(\frac1\beta)})^\beta \cdot \widetilde{\Upsilon} ,
 \end{align*}
where $\mathscr{D}_\infty(W^{(\frac1\beta)})$ is given in \eqref{eqn-dmartingale-beta} and is strictly positive $\PP^*$-a.s., and  $\widetilde{\Upsilon}$ is independent of $\mathscr{D}_\infty(W^{(\frac1\beta)})$ and has a real part following a $\frac{1}{\beta}$-stable law.
\end{itemize}
\end{proposition}

\begin{remark*}
In Proposition~\ref{prop-clt}, in the squared sub-critical and critical regimes,  our statements work for $b \ge 3$. The case $b = 2$ are slightly different in these two regimes: the standard complex normal distribution $\mathcal{N}_\C(0,1)$ should be replaced by 
\[
    \Big(\frac{\E[|\widehat{\mu}_\infty(1)|^2] + \E[(\widehat{\mu}_\infty(1))^2]}{\E[|\widehat{\mu}_\infty(1)|^2]}\Big)^{1/2} \cdot \frac{X}{\sqrt{2}} + i    \Big(\frac{\E[|\widehat{\mu}_\infty(1)|^2] - \E[(\widehat{\mu}_\infty(1))^2]}{\E[|\widehat{\mu}_\infty(1)|^2]}\Big)^{1/2} \cdot   \frac{Y}{\sqrt{2}},
\]
with $\E[(\widehat{\mu}_\infty(1))^2] = - 2 \E[\mathring{W}^2]/\pi^2 \in \R$    and $X, Y$ being two independent standard real normal  random variables.  See    Propositions \ref{thm-sq-sub-critical} and \ref{thm-sq-critical} in \S \ref{sec-clt} for the details. 
\end{remark*}

\subsection{Litterature review}

At the early stages of his cascade theory, Mandelbrot had already recognized the importance of the Fourier transform of MCCM in turbulence theory. Kahane’s open program on random measures was formally proposed in \cite{Ka-93}. It is closely related to the Fourier transforms of random fractal measures which are active problems in fractal geometry \cite{Mattila,BBS}. One can consult  \cite{Kahane-1987,FOS,FS}. Furthermore, the majority of these studies are valid for the Rajchman property and the polynomial Rajchman property.

Shmerkin and Suomala, in 2018, comprehensively elucidated the Fourier dimension of an MCCM of the $\beta$-model by the construction in \cite{LP09} (see \cite[Chapter 14]{SS18} for more details). In particular, the MCCM of the $\beta$-model has Salem property. By our Corollary \ref{cor-salem}, it is the sole Salem measure within MCCMs.

The study of the behavior at infinity of the Fourier transform of the Borel measure dates back to the works of Riemann and Lebesgue. According to the Riemann-Lebesgue Lemma,  all absolutely continuous measures on the unit circle satisfy the Rajchman property $\widehat{\mu}(n)=o(1)$.  It is remarkable that not all Rajchman measures are absolutely continuous. In 1916, Menchoff provided an example of singular Rajchman measure. Yet, Kershner \cite{K36} 
showed in 1936 that Menchoff's example does not possess a positive Fourier dimension. In the same year, Littlewood \cite{L36} demonstrated the existence of a singular measure with a positive Fourier dimension. Littlewood's proof was purely existential. Later, in 1938, Wiener and Wintner \cite{WW38} constructed explicit examples where $\widehat{\mu}(n)=O(|n|^{-1/2+\epsilon})$ for any $\epsilon>0$.

Rajchman property is intimately related to the set of uniqueness in harmonic analysis (See \cite{Ly95} for the early history and \cite{Li18,LS20,LS22,VY22} for recent developments on fractal sets).
The groundbreaking work of Lyons in the 1980s \cite{Ly085} gives theoretical characterizations of the Rajchman measures. On the other hand, given a specific measure, it remains a notoriously challenging problem to determine the Rajchman property.
Started by Bourgain and Dyatlov \cite{BD}, the Fourier decay bounds for the Patterson-Sullivan and Furstenberg measures are investigated in \cite{LNP,Li22, DKW}. The relationship between the positive Fourier dimension of a measure and the Fourier restriction theorem was first established by Mockenhaupt in \cite{M00}. This connection was further explored by Mitsis \cite{Mit02} and Bak-Seeger \cite{BS11}. The sharpness of the Mockenhaupt-Mitsis-Bak-Seeger restriction theorem was subsequently investigated by Hambrook and Łaba in \cite{HL13,HL16}.  One can find further applications to 3-term arithmetic progressions on fractal set in \cite{LP09}. For more details, see also \cite{Laba} and \cite{Pra23}.

In quantum chaos, by the Fractal Uncertainty Principles \cite{BD, DZ}, the asymptotic behavior of Fourier transform of stationary measures is a powerful tool for studying the relationship between classical chaotic dynamics and their quantization. 
Details of connections between iterated function systems and the Fourier  decay can be found in the survey by Sahlsten \cite{Sah}. One can consult the remarkable recent work by Baker-Khalil-Sahlsten \cite{BKS} for the  Fourier decay of various self-similar measures in dynamical systems.

Rajchman and polynomial Rajchman properties are also related to the equidistribution problem and  Diophantine approximation. By Davenport-Erdos-LeVeque \cite{DEL} and Rauzy \cite{Ra76}, let $\mu$ be a  measure on $[0,1]$ with $\dim_F(\mu)>0$, then for any  strictly increasing integer sequence $\{u_n\}_{n\in\mathbb{N}}$ the sequence $\{u_nx \,(\text{mod}1)\}_{n\in\mathbb{N}}$ is  uniformly distributed for $\mu$-$a.e. \, x\in [0,1]$.  One can consult \cite{Bug12} for classical equidistribution theory and \cite{HS15} for fractal measures. For instance, the Rajchman property of self-conformal measures and self-similar measures are widely studied in \cite{ARW,Bre21,Sol22,Rap22}. See also \cite[Section 6.1]{V23} for the recent developments.

In contrast to the Salem property, a continuously parametrized family of dimensions that vary between the Fourier and Hausdorff dimensions is defined as the Fourier spectrum of a measure. This concept was introduced in \cite{Fra24} and further explored in \cite{CF24}. One can find various aspects of the Fourier dimension
and its variants in \cite{ES17}.

The polynomial Rajchman property, more precisely the Fourier dimension, is naturally associated with the celebrated Salem set. 
By using random constructions, Salem in \cite{Sal51} firstly constructed a Salem set. More random examples can be  found in \cite{Kahane-book, Blu99, SX06, LP09, Eks16, FS18}. An explicit example of Salem set in $\mathbb{R}$ was provided by Kaufman \cite{Kau81}, in $\mathbb{R}^2$ by Hambrook \cite{Ham17}, and in the general $\mathbb{R}^n$ by Fraser and Hambrook \cite{FH23}.

In \cite{BBS, Fal86, Mattila, Wolff}, one can find  more information on Hausdorff and Fourier dimensions, as well as their applications.

\subsection{Organization of the paper}

The paper is organized as follows.   In the introduction  \S \ref{sec-intro}, we give the statement of the main results and the  history of Mandelbrot-Kahane problem, as well as other related topics.  In \S \ref{sec-pre}, we give preliminary  materials on Mandelbrot cascades, the  classical definitions of  various dimensions of measures on the real line $\R$.  The classical tools in vector-valued  martingale theory are recalled in \S \ref{sec-pre-bis}.   The sections \S \ref{sec-CAV} and \S\ref{sec-Lplq} are devoted to the  formalization of  the framework of the actions by multiplicative cascades on finitely additive vector measures and general properties of such actions when the Banach space has martingale type $p$.  The sharp lower bound of the Fourier dimension $\dim_F(\mu_\infty)$ is given in  \S \ref{sec-mccm}, \S\ref{sec-W} and \S \ref{sec-squared}.    The sharp upper bound of the Fourier dimension $\dim_F(\mu_\infty)$ is based on   Proposition \ref{prop-clt} about the  weak convergence results on BRW, whose proof is  given  in \S \ref{sec-clt}.     The concluding parts which finalize the whole proofs of the main results Theorem~\ref{thm-fourier}, Corollary~\ref{cor-salem}, Theorem~\ref{thm-pRaj} and Corollary \ref{cor-holder} are included in \S \ref{sec-con-1} and \S \ref{sec-con-2} respectively.  Finally, in the Appendix  of this paper , we include the routine proofs of some classical results such as  the uniqueness of the Biggins-Kyprianou transform, the conditional Linderberg-Feller central limit theorem, etc.

\subsection*{Acknowledgements.} 
During the preparation of our manuscript, we discovered that Changhao Chen, Bing Li, and Ville Suomala \cite{CLS24} were also investigating from multifractal analysis perspective the Fourier dimension of Mandelbrot cascades, a topic closely related to Theorem 1 in our paper. The two groups exchanged preliminary manuscripts, though without full details. Through these communications, we noted that their approach and ideas differ significantly from ours. We appreciate their openness to academic dialogue and value the gracious exchange of progress in their work.

The authors are grateful to  Michel Pain  who pointed out to us the appearance of the stable law in the squared super-critical regime in  Proposition \ref{prop-clt}.  This work is supported by the National Natural Science Foundation of China (No.12288201). XC is supported by Nation Key R\&D Program of China 2022YFA1006500. YH is supported by NSFC 12131016 and 12201419,  ZW is supported by NSF of Chongqing (CSTB2022BSXM-JCX0088,
CSTB2022NSCQ-MSX0321).

\section{Preliminaries}\label{sec-pre}
\subsection{Notation}

Throughout the paper, by writing $A\lesssim_{x,y} B$, we mean there exists a finite constant $C_{x,y}>0$ depending only on $x,y$ such that $A \le C_{x,y}B$. And, by writing  $A \asymp_{x,y} B$, we mean  $A\lesssim_{x,y}B$ and $B\lesssim_{x,y} A$.

Let $\N = \{1, 2, 3, \cdots\}$ denote the set of positive integers.  

Given any integrable random variable $X$, we shall write $\mathring{X}$ the centerization of $X$: 
\[
\mathring{X}: = X-\E[X]. 
\]

\subsection{Mandelbrot  canonical  cascade measures}\label{sec-Mandelbrot-cascade}

Fix an integer $b \ge 2$ and the set    of alphabets $\Ab= \{0, 1, \cdots, b-1\}$.  The rooted $b$-homogeneous tree (with the root denoted by $\emptyset$ and the convention $\Ab^0 = \{\emptyset\}$) can be canonically identified with  
\[
\Ab^* = \bigsqcup_{n\ge 0}\Ab^n.
\] 
 Elements of $\Ab^*$ are written as words: If $u= x_1x_2\cdots x_n$ with $x_j\in \Ab$, then we set $|u| = n, u_j = x_j$ and $u|_k= x_1\cdots x_k$ (with $u|_0=\emptyset$).

Recall that the $b$-adic intervals  on the unit interval  $[0, 1]$ are defined  for any $n =1, 2, \dots$ and $u\in \Ab^*$ with $|u|=n$ by
\begin{align}\label{def-Iu}
I_u = \Big[\sum_1^n u_k b^{-k},  \sum_1^n u_k b^{-k} + b^{-n} \Big).
\end{align}

As in the introduction,   a random weight  is a  random variable $W \ge 0$  which is non-constant and  $\E[W]=1$. Let $(W(u))_{u\in \Ab^*\setminus \{\emptyset\}}$ denote the  independent copies of $W$,  and as usual, on the root vertex, we set $W(\emptyset) \equiv 1$.  Then, for any  integer $n\ge 0$,  Mandelbrot \cite{M74} defined the random measure  $\mu_n$ on $[0,1]$ by 
\begin{align}\label{def-Lebn}
\mu_n (dt): = \sum_{|u|=n} \Big(\prod_{j=0}^n W(u|_j) \Big) \indi_{I_u} (t)dt. 
\end{align}
\begin{remark*}
Note that, since  $W(\emptyset)=1$, for any $n\ge 1$, we have 
\[
\prod_{j=0}^n W(u|_j) = \prod_{j=1}^n W(u|_j). 
\]
So we can use both way of writing but we prefer to write $\prod_{j=0}^n W(u|_j)$ for avoiding the notation $\prod_{j=1}^0 W(u|_j)$ later.  However, in  \S \ref{sec-mccm} and  \S \ref{sec-clt}, we shall note that for $n\ge 1$, 
\[ 
 \prod_{j=1}^{n} W(u|_j)^2 =  \prod_{j=0}^{n} W(u|_j)^2 \an  \prod_{j=1}^{n} \frac{W(u|_j)^2}{\E[W^2]} \ne \prod_{j=0}^{n} \frac{W(u|_j)^2}{\E[W^2]}. 
\]
\end{remark*}

 Let $\mathscr{F}_n$ denote the sigma-algebra defined by 
\begin{align}\label{def-fil-F}
\mathscr{F}_n: = \sigma\Big(\Big\{ W(u)\Big| \text{$u\in \Ab^*$ with $|u|\le n$}\Big\}\Big) \quad \text{for all $n\ge 0$}. 
\end{align}
Note that $\mathscr{F}_0$ is the trivial sigma-algebra.  We also set
\[
\mathscr{F}_\infty= \bigvee_{n\ge 1} \mathscr{F}_n =   \sigma\Big\{\bigcup_{n\ge 1} \mathscr{F}_n\Big\}.
\]
Clearly,  $(\mu_n)_{n\ge 0}$ is a measure-valued martingale with respect to the filtration $(\mathscr{F}_n)_{n \ge 0}$ and almost surely, the sequence converges weakly to a limiting random measure $\mu_\infty$: 
\[
\mu_n \xrightarrow[a.s.]{n\to\infty}\mu_\infty. 
\]
The random measure $\mu_\infty$ is usually referred to as the Mandelbrot canonical cascade measure (MCCM). 

By  Kahane's fundamental  work \cite{Kahane-Peyriere-advance}, the following assertions are equivalent: 
\begin{itemize}
\item[(i)] the random measure $\mu_\infty$ is non-degenerate (that is $\E[\mu_\infty([0,1])] >0$);
\item[(ii)] the martingale  $(\mu_n([0,1]))_{n\ge 1}$ is uniformly integrable and hence 
\[
\mu_n([0,1])\xrightarrow[a.s. \an  L^1]{n\to\infty} \mu_\infty([0,1]);
\]
\item[(iii)] $W$ satisfies Mandelbrot-Kahane's non-degeneracy condition:
\begin{align}\label{non-deg}
\E[W\log W] <\log b.
\end{align}
\end{itemize}
Moreover,  Kahane \cite{Kahane-Peyriere-advance} proved that, for any $1<p<\infty$,  the martingale $(\mu_n([0,1]))_{n\ge 1}$ is $L^p$-uniformly bounded if and only if $W$ satisfies Kahane's $L^p$-condition:
\begin{align}\label{Kahane-Lp-bdd} 
\E[W^p]<b^{p-1}. 
\end{align}

\begin{remark}\label{rem-Lp-non-deg}
From Kahane's statements, we see that Kahane's $L^p$-condition \eqref{Kahane-Lp-bdd}  implies the $L^r$-condition for all $1<r<p$ and it also implies  the non-degeneracy condition \eqref{non-deg}.   In fact, one can see these implications directly from the following equivalence: 
\[
\E[W^p]<b^{p-1} \Longleftrightarrow \|W\|_{L^{p-1}(W \PP)} = \Big( \int  W^{p-1} \cdot Wd\PP\Big)^{\frac{1}{p-1}} <b.
\]
\end{remark}

\subsection{Actions on general measures by multiplicative cascades}

\subsubsection{General measure case}\label{sec-g-m}

In the above construction \eqref{def-Lebn}, one may replace the Lebesgue measure by general measures. More precisely,  for any  finite positive  Borel measure $\nu$ on $[0,1]$, define the random measure  $Q_n\nu$ on $[0,1]$ by 
\begin{align}\label{def-Qnu-gen}
Q_n\nu (dt): = \sum_{|u|=n} \Big(\prod_{j=0}^n W(u|_j) \Big) \indi_{I_u} (t)\nu(dt). 
\end{align}
By Kahane's general $T$-martingale  theory  \cite{Kahane-1987}, almost surely, $Q_n\nu$ converges weakly to a random measure, denoted by $Q\nu$: 
\[
Q_n\nu \xrightarrow[a.s.]{n\to\infty} Q\nu.
\]
In what follows,  the maps  $\nu\mapsto Q_n \nu$ and $\nu\mapsto Q\nu$  will be referred to as  multiplicative cascade actions (MC-actions) on $\nu$.    And we say that $Q$ acts fully on $\nu$ if the martingale $(Q_n\nu([0,1]))_{n\ge 1}$ is uniformly integrable, or equivalently, for all $n\ge 1$, 
\[
Q_n\nu([0,1]) = \E[Q\nu([0,1])|\mathscr{F}_n] \quad a.s. 
\]

\subsubsection{Squared MC-actions}\label{sec-sq-mc}
In the construction of the Mandelbrot canonical cascade measures as well as the general constructions of $Q_n\nu$ in \eqref{def-Qnu-gen},  the random variable  $W$ is referred to as the {\it initial random weight} and    the martingale 
\[
Q_n\nu([0,1])  = \sum_{|u|= n}  \Big(\prod_{j=0}^n W(u|_j) \Big) \nu(I_u)
\]
is referred to as the {\it MC-total mass martingale}. 

In our study,  under the assumption $\E[W^2]<\infty$,   we are led naturally to the study of the  MC-actions associated  to the new initial random weight $W^{(2)} = W^2/\E[W^2]$, see \S \ref{sec-squared} for more details. Such MC-actions will be referred to as  the {\it squared MC-actions} and denoted by $Q_n^{(2)}$ and $Q^{(2)}$, the  related MC-martingales will also be referred to as {\it squared MC-martingales}.

\subsection{Dimensions of measures}

In this subsection, we briefly review  basic materials on dimensions for measures on the one-dimensional real line $\R$.

{\flushleft \it (i) Hausdorff dimension.} Given any subset $A\subset \R$, let $\dim_H(A)$ denote its Hausdorff dimension.  The Hausdorff dimension of a finite Borel measure $\nu$ on $\R$ is defined by  (see \cite[Section 1.9, pp. 42--43]{BBS}): 
\[
\dim_H(\nu):= \inf\{\dim_H(A):    A\subset \R \an    \nu(A^c) = 0 \}. 
\]
And, the lower Hausdorff dimension of $\nu$ is defined as 
\[
\underline{\dim}_H(\nu):= \inf\{\dim_H(A):    A\subset \R \an    \nu(A) > 0 \}. 
\]
The lower Hausdorff dimension $\underline{\dim}_H(\nu)$ was denoted   $\dim_{*} (\nu)$ in \cite{Fan-aihua-lq} and coincides with the lower local dimension \cite[Theorem 1.2]{Fan-dimension}.  A measure $\nu$ is called unidimensional if $\underline{\dim}_H(\nu) = \dim_H(\nu)$.

{\flushleft \it (ii) Fourier dimension.} Given  a finite Borel measure $\nu$ on $\R$, its   Fourier transform  $\widehat{\nu}(\zeta)$ is defined by 
\[
\widehat{\nu}(\zeta)  = \int_\R  e^{-i 2 \pi x \cdot \zeta} d\nu(x). 
\] 
The Fourier dimension $\dim_F(\nu)$ is defined as  (see \cite[Section 8.2, p. 231]{BBS}): 
\begin{align}\label{def-Fdim}
\dim_F(\nu):  = \sup\Big\{ D \in [0, 1]: \text{$| \widehat{\nu}(\zeta)| \le c_D (1+ |\zeta|)^{-  D /2}$ for all $\zeta \in \R$} \Big\}.
\end{align}

For any measure supported on a compact subset of $\R$,  its  Fourier dimension has the following well-known equivalent formulation due to Kahane \cite[Lemma 1, p. 252]{Kahane-book} (the following equivalence can also be derived, for instance, from \cite[Lemma 9.A.4]{Wolff}). 
\begin{lemma}\label{lem-Kahane}
Let $\nu$ be a finite Borel measure  on $[0,1]$. If $\nu$  is supported on  a sub-interval $I\subset [0,1]$ of length $|I|<1$,  then its Fourier dimension is given by 
\begin{align}\label{Fdim-lattice}
\dim_F(\nu)= \sup\Big\{D \in [0, 1]: \text{$| \widehat{\nu}(n)| \le c_D (1+ |n|)^{-  D /2}$ for all $n \in  \Z$} \Big\}.
\end{align}
\end{lemma}

As we mentioned in Lemma \ref{lem-trans-coef} and Remark \ref{rem-ka} in the introduction, in fact,  the support assumption in Lemma~\ref{lem-Kahane}  can be removed.

It should also be noticed that, in the definition of the Fourier dimension \eqref{def-Fdim} of $\dim_F(\nu)$, we require the restriction  $D \le 1$ and thus $\dim_F(\nu)\le 1$.   However, when considering absolutely continuous measures, it is also natural to consider the definition (see \cite[formula (8.14), p. 233]{BBS}) without this  restriction: 
\[
\widetilde{\dim_F}(\nu): = \sup\Big\{\sigma \ge 0: \text{$| \widehat{\nu}(\zeta)| \le c(\sigma) (1+ |\zeta|)^{-  \sigma /2}$ for all $\zeta \in \R$} \Big\}. 
\]
Indeed, we shall use the elementary fact \S \ref{sec-grv} below:  the Lebesgue measure $\lambda|_I$ restricted on any sub-interval  $I\subset [0,1]$,  satisfies
\[
\widetilde{\dim_F}(\lambda|_I)= 2. 
\]

\subsubsection{Salem measures}\label{sec-salem}

For any finite Borel measure $\nu$ on $\R$,  
$
\dim_F(\nu) \le \dim_H(\nu)$ (see for instance \cite[formula (8.12), p. 231]{BBS}). 
We say that  $\nu$ is a {\it Salem measure} if 
\[
\dim_F(\nu)= \dim_H(\nu). 
\]

The reader is referred to \cite{Laba} (the definition there is slightly different) for more discussions on recent works on  Salem measures and  the related topics of Salem sets.

\subsubsection{$L^p$-dimensions}\label{sec-def-Dp}
  We shall need the definition of  $L^p$-dimensions introduced by R\'enyi \cite{Renyi}.  Let $\nu$ be a Borel probability measure on $[0,1)$.  For any $p>1$, the  $L^p$-dimension of $\nu$ is defined by
\begin{align}\label{def-D-p}
\underline{\dim}_p(\nu): = \liminf_{n\to\infty} \frac{-\log \big(\sum_{|u|=n}  \nu(I_u)^p\big) }{ n  (p-1)\log b}.
\end{align}
One can show that
$
0\le \underline{\dim}_p(\nu) \le 1. 
$
By H\"older's inequalities, one can show that  the  function  $p\mapsto \underline{\dim}_p(\nu)$ is non-increasing, while the function $p \mapsto (p-1) \underline{\dim}_p(\nu)$ is concave and non-decreasing on $(1, \infty)$.  And  they are both  continuous on $(1, \infty)$. See  \cite[Section 2.6, pp. 70--72]{BBS} and  \cite[p. 261]{Lq-survey}  for more details.

In the litterature,  the  lower $L^p$-dimension of $\nu$ is defined as (see  \cite[Section 2.6, pp. 70--72]{BBS})
\[
D_\nu(p)=   \liminf_{r\to 0^{+}} \frac{  \log \big(\int_{[0,1]} \nu(B(x,r))^{p-1}  d\nu(x)\big)}{ (p-1)\log r }. 
\]
It is known that (see \cite{Fan-dimension} and \cite[Lemma 2.6.6]{BBS}) that 
$
D_\nu(p) = \underline{\dim}_p(\nu). 
$

\section{Standard inequalities in vector-valued martingale theory}\label{sec-pre-bis}

We are going to use the theory of martingales in Banach spaces and the UMD properties for Banach spaces.   For the background of martingales in Banach spaces, the reader is referred to Pisier's recent book \cite[Chapter 5 and Chapter 10]{Pisier-book}. 

\subsection{Notation and conventions}\label{sec-not}
Throughout the paper,  slightly abusing the notation, for any index set $I$,  we set $\Delta  = \Delta_I= \{\pm 1\}^I$ (the index set will vary in different places).  Elements of $\Delta$ are denoted by $\epsilon= (\epsilon_i)_i$ and $\Delta$ is  equipped with the probability measure  
\begin{align}\label{def-dep}
d \epsilon = \Big(\frac{\delta_{-1} + \delta_1}{2}\Big)^{\otimes I}.
\end{align}

Let $(\Sigma,  \mathcal{B},  \rho)$  be  any measure space with $\rho$ a non-negative $\sigma$-finite measure.  Then for any real or complex Banach space (or quasi-Banach space) $\fX$ and any $0<p<\infty$, we denote $L^p(\Sigma, \mathcal{B}, \rho; \fX)$ the space of $\fX$-vector valued $L^p$-Bochner integrable functions $f: \Sigma\rightarrow \fX$.  When the measure space is clear, we shall use  the simplified notation
\[
 L^p(\rho; \,\fX), \,  L^p(\Sigma; \,\fX) \,\, \text{or}\,\,  L^p(\fX). 
\]

For any $0<r< \infty$, let
$
\ell^r : = \ell^r(\N)$ be 
the space of $\ell^r$-summable complex sequences, equipped with the standard norm (when $r\ge 1$) or quasi-norm (when $0<r<1$).  For any $v = (v_n)_{n\ge 1}\in \C^\N$, write 
\[
|v| : = (|v_n|)_{n\ge 1}. 
\]
Hence, for any vector-valued function  $f\in L^p(\ell^r)$, we  define  $g= |f|\in L^p(\ell^r)$ in the above sense and 
\[
\| f\|_{L^p(\ell^r)}  = \| g \|_{L^p(\ell^r)}. 
\]

Given a Banach space $\fX$ and  any probability space $(\Omega,  \mathscr{G}, \PP)$ equipped with an increasing filtration  of sub-sigma-algebras
 \[
 \mathscr{G}_1\subset \cdots  \mathscr{G}_n \subset  \mathscr{G}_{n+1} \subset \cdots \subset  \mathscr{G},
 \] 
 we may define $\fX$-valued martingales adapted to the above filtration. In what follows, we shall frequently say that  $(Z_n)_{n\ge 1}$ is an $\fX$-valued martingale, without referring to the filtration and sometimes we shall use the simplified notation $\E_n[\cdot]$ for the conditional expectation $\E[\cdot|\mathscr{G}_n]$.

 \subsection{Kahane-Khintchine inequalities}  
 
 The celebrated Kahane-Khintchine inequalities (see, e.g., \cite[Theorem 5.2, p.154]{Pisier-book}) state that, for any $0< p_1 \neq p_2<\infty$,  for an arbitrary real or complex Banach space $\fX$ and an arbitrary  finitely supported sequence $(x_i)_{i}$ in $\fX$, 
 \begin{align}\label{KK-ineq}
\big \| \sum_i \epsilon_i x_i\big \|_{L^{p_1}(\Delta, d\epsilon; \, \fX)} \asymp_{p_1, p_2} \big \| \sum_i \epsilon_i x_i\big \|_{L^{p_2}(\Delta, d\epsilon; \, \fX)}.
 \end{align}
  When $\fX$ is the one-dimensional space $\R$ or $\C$, the above inequalities \eqref{KK-ineq} reduce to the  classical Khintchine inequalities. 
 The following elementary consequence of the Kahane-Khintchine inequalities  will be useful.   For any $1\le  q<\infty$ and any finite  sequence of scalar functions $(f_i)_i$ in  $L^q(\Sigma, \rho)$, 
\begin{align}\label{LpLq-abs}
\Big\| \sum_i \epsilon_i f_i\Big\|_{L^p(\Delta, d\epsilon; \, L^q(\Sigma, \rho))} \asymp_{p,q}    \Big\| \Big(\sum_i  | f_i|^2\Big)^{1/2}\Big\|_{L^q(\Sigma, \rho)}.
\end{align}

\subsection{UMD Banach spaces}

A Banach space $\fX$ is said to be a UMD space if there exists an exponent $1<r<\infty$  (which is then equivalent to hold for all $1<r<\infty$)  and a constant $C(r, \fX)>0$ such that  for any probability space $(\Omega,  \mathscr{G}, \PP)$ equipped with an increasing filtration  of sub-sigma-algebras and  any $\fX$-valued martingales $(Z_n)_{n\ge 0}$ in $L^r(\fX)$, 
\begin{align}\label{def-UMD}
\sup_{\theta \in \{\pm 1\}^\N}    \sup_N\Big\|\sum_{n = 0}^N    \theta_n  dZ_n  \Big\|_{L^r(\Omega; \fX)} \le C(r, \fX)   \cdot \sup_N\|Z_N\|_{L^r(\Omega; \fX)},
\end{align}
where $dZ_n= Z_n-Z_{n-1}$ with the convention $Z_{-1}: =0$ such that $dZ_0=Z_0$.   An equivalent definition of the UMD property of $\fX$ is to replace the inequality \eqref{def-UMD} by the following  Burkholder inequalities for UMD spaces (see \cite[Theorem 5.2 and Proposition 5.10]{Pisier-book} for this equivalence): for any $N$, 
\begin{align}\label{burk-ineq}
\| Z_N\|_{L^r(\fX)} \asymp_{r, \fX}  \Big\|\sum_{n = 1}^N    \epsilon_n  dZ_n  \Big\|_{L^r(\Delta \times \Omega; \fX)}.
\end{align}

The classical examples of UMD Banach spaces are the usual Lebesgue spaces $L^p$ and the sequence spaces $\ell^p$ for $1< p<\infty$.  The reader is referred to \cite[Chapters 5-6 and Chapter 10]{Pisier-book} for the general theory on the UMD Banach spaces.

 \subsection{Bourgain-Stein inequalities for UMD spaces}
 
  The Bourgain-Stein inequalities (see, e.g.,  \cite[Theorem 5.60, p. 196]{Pisier-book}) are given as follows.  Let $\fX$ be a UMD Banach space and let $1<r<\infty$. Then there exists a constant $C(r, \fX)>0$ such that  for  any probability space $(\Omega,  \mathscr{G}, \PP)$ equipped with an increasing filtration  of sub-sigma-algebras  and  any sequence $(f_n)_{n\ge 1}$  (which is not necessarily adapted) of  vector-valued functions in $L^r(\Omega, \mathscr{G}, \PP; \, \fX)$, 
 \begin{align}\label{BS-ineq}
\Big\| \sum_{n\ge 1} \epsilon_n \E_n[f_n] \Big\|_{L^r(\Delta\times \Omega, d\epsilon\times d\PP; \, \fX)} \le C(r, \fX)  \Big\| \sum_{n\ge 1} \epsilon_n f_n \Big\|_{L^r(\Delta\times \Omega, d\epsilon\times d\PP; \, \fX)}. 
 \end{align}
Note that,   in the application of Bourgain-Stein inequalities, it is of crucial importance that the sequences are only required to be  measurable with respect to the largest sigma-algebra.

 \subsection{Martingale types and Rademacher types}
Fix $1< p\le 2$. A Banach space $\fX$ is said to have martingale type $p$  (see \cite[Definition 10.41, p. 409]{Pisier-book}) if there exists a constant $C(p, \fX)>0$ such that all $\fX$-valued martingales $(Z_n)_{n\ge 0}$ in $L^p(\fX)$ satisfy
\begin{align}\label{def-Mtype}
\sup_n    \E \big[ \| Z_n\|_\fX^p\big] \le  C(p, \fX) \sum_{n \ge 0}   \E\big[ \|dZ_n\|_\fX^p\big].
\end{align}
The inequality \eqref{def-Mtype} implies in particular that,  for any family of {\it independent and  centered}  random variables $(F_k)_{k=1}^n$  in $L^p(\fX)$, 
\begin{align}\label{def-ind-Mtype}
\E\Big[ \Big\|\sum_{k=1}^n F_k \Big\|_{\fX}^p\Big] \le C(p, \fX)   \sum_{k =1}^n \E\big[\|F_k \|_{\fX}^p\big].  
\end{align} 

A closely related  property is the Rademacher type (see \cite[formula (10.38), p. 407]{Pisier-book}) defined as follows. Let $1<p\le 2$. A Banach space $\fX$ is said to have  Rademacher type $p$ (or merely said to have type $p$) if there is a constant $C(p, \fX)$ such that for any finite sequence $(x_j)$ in $\fX$, 
\[
\Big\| \sum_{j} \epsilon_j x_j \Big\|_{L^2(\Delta, d\epsilon; \, \fX)} \le C(p, \fX) \Big( \sum_{j} \| x_j\|_\fX^p  \Big)^{1/p}. 
\]
Clearly, if a Banach space $\fX$ has Rademacher type $p$ with $1< p\le 2$, then it has Rademacher type $r$ for all $1<r \le p$.   It is also clear from \eqref{def-ind-Mtype} that martingale type $p$ implies Rademacher type $p$.   Conversely, for UMD spaces, Rademacher type $p$ implies martingale type $p$, see \cite[Corollary~10.23, Proposition~10.31 and Proposition~10.40]{Pisier-book}. 

 For any $2\le q<\infty$, the sequence space $\ell^q$ is a UMD space and has Rademacher type $2$ (see \cite[Proposition~10.36, p. 407]{Pisier-book}).  Thus, we have   the following well-known fact. 

\begin{proposition}\label{prop-type}
  For any  $2\le q<\infty$, the space $\ell^q$ has martingale type $p$ for all $1<p\le 2$. 
\end{proposition}

\subsection{Some useful elementary inequalities}
The following inequalities will be used frequently:  for any $0< \lambda \le 1$ and $\Lambda\ge  1$ and any non-negative numbers $a_i\ge 0$,  
\begin{align}\label{2-ele-ineq}
\Big(\sum_i a_i \Big)^{\lambda} \le \sum_i a_i^\lambda \an   \Big(\sum_i a_i \Big)^{\Lambda} \ge \sum_i a_i^\Lambda.
\end{align}

We are going to use the Minkowski inequalities: Let $(\Sigma_1, \rho_1), (\Sigma_2, \rho_2)$ be two measure spaces equipped  with $\sigma$-finite non-negative measures $\rho_1, \rho_2$, which are not necessarily probability measures.  Then for any $p_1, p_2$ with $1\le p_1<p_2<\infty$ and any measurable function $f: \Sigma_1\times \Sigma_2 \rightarrow \R$, we have  (see \cite[A1]{S70})
\[
\Big\{ \int_{\Sigma_2}      \Big(   \int_{\Sigma_1}  |f(x_1, x_2)|^{p_1}  d\rho_1(x_1) \Big)^{\frac{p_2}{p_1}}       d\rho_2(x_2) \Big\}^{\frac{1}{p_2}} 
\le  \Big\{ \int_{\Sigma_1}      \Big(   \int_{\Sigma_2}  |f(x_1, x_2)|^{p_2}  d\rho_2(x_2) \Big)^{\frac{p_1}{p_2}}       d\rho_1(x_1) \Big\}^{\frac{1}{p_1}}.
\]
In particular,  if $1<p<q<\infty$ and $(X_n)_{n\ge 1}$ is a sequence of non-negative random variables, then 
\begin{align}\label{Min-pq}
\E\Big[ \Big\{ \sum_{n=1}^\infty  X_n   \Big\}^{\frac{p}{q}}\Big]  \ge   \Big\{ \sum_{n=1}^\infty   \Big(  \E  \big[  X_n^{\frac{p}{q}}\big]\Big)^{\frac{q}{p}}  \Big\}^{\frac{p}{q}}.
\end{align}

\section{MC-actions on vector-measures}\label{sec-CAV}

\subsection{MC-actions on finitely additive vector measures}
Throughout the whole paper, by vector measures, we always mean finitely additive ones.   The actions by multiplicative cascades on  vector measures play a crucial role in our formalism.   Now we define such actions as follows.

Recall the definition \eqref{def-Iu} for the $b$-adic intervals $I_u$.   For any $n\ge 1$, let $\mathscr{B}_n^b$ be  the atomic sigma-algebra generated by the $b$-adic intervals $I_u\subset [0,1]$ with $|u|=n$. And,  for $n=0$, let $\mathscr{B}_0^b$  be the trivial sigma-algebra on $[0,1]$.   Let
\begin{align}\label{def-B-alg}
\mathfrak{B}^b: = \bigcup_{n\ge 1} \mathscr{B}_n^b. 
\end{align}
Note that $\mathfrak{B}^b$ is not a sigma-algebra.

Let  $\fX$ be any complex Banach space. By an $\fX$-valued  vector measure on $[0,1]$, we  mean a   finitely additive mapping  $\mathcal{V}: \mathfrak{B}^b \rightarrow \fX$.   
Given any vector measure $\mathcal{V}: \mathfrak{B}^b \rightarrow \fX$, we define the cascade action $Q_n$ on $\mathcal{V}$ as follows. For any $n\ge 0$, let $Q_n \mathcal{V}$ be  the random vector measure  on $\mathfrak{B}^b$ given as
\begin{align}\label{def-ran-vec}
Q_n\mathcal{V} (dt):  =   \sum_{|u|=n} \Big(\prod_{j=0}^n W(u|_j) \Big) \indi_{I_u} (t)\mathcal{V}(dt). 
\end{align}
That is, 
\[
Q_n\mathcal{V} (A)=    \sum_{|u|=n} \Big(\prod_{j=0}^n W(u|_j) \Big)  \mathcal{V}  (I_u \cap A) \quad \text{for all $A\in \mathfrak{B}^b$}.  
\]
With respect to the natural filtration $(\mathscr{F}_n)_{n\ge 0}$ defined in \eqref{def-fil-F},  the sequence  $(Q_n\mathcal{V})_{n\ge 0}$ is a martingale taking values in the space of vector measures defined on $\mathfrak{B}^b$.  This leads to an $\fX$-valued martingale 
\begin{align}\label{def-mart}
M_n^{\mathcal{V}} : = Q_n\mathcal{V}([0,1] )=  \sum_{|u|=n} \Big(\prod_{j=0}^n W(u|_j) \Big) \mathcal{V}(I_u), \quad n\ge 0.
\end{align}
The martingale increments of  \eqref{def-mart}  are given by 
\begin{align}\label{def-dM}
dM_n^\mathcal{V}: = M_n^\mathcal{V}- M_{n-1}^\mathcal{V}=  \sum_{|u|=n}  \Big(\prod_{j=1}^{n-1} W(u|_j)\Big) \mathring{W}(u) \cdot \mathcal{V}(I_u), \quad n\ge 1. 
\end{align}
By convention, we set 
$
dM_0^\mathcal{V}  =  M_0^\mathcal{V} = \mathcal{V}([0,1]). 
$

\begin{remark*}
Vector-valued multiplicative processes or multiplicative chaos acting on Banach spaces with nice properties of martingale convergence were studied in other context than that of this paper. See \cite{JLN73, Big92, Barral2001, LRR03, JSW19}. 
\end{remark*}

\subsection{Growth rate of vector measures}\label{sec-vm}
Given a vector measure  $\mathcal{V}: \mathfrak{B}^b \rightarrow \fX$,   we  can define another  $\fX$-valued martingale  $(T^\mathcal{V}_n)_{n\ge 0}$ on the probability space $([0,1], dt)$ with respect to the  standard $b$-adic filtration   $(\mathscr{B}_n^b)_{n\ge 0}$: 
\begin{align}\label{def-H-mart}
T^\mathcal{V}_n :  = b^n \sum_{|u|=n} \mathcal{V}(I_u) \indi_{I_u}, \quad n\ge 0. 
\end{align}
Clearly, $T^\mathcal{V}_n \in L^p([0,1], dt; \fX)$ for any $1\le p\le \infty$, where $L^p([0,1], dt; \fX)$  denotes the space of $\fX$-valued Bochner integrable functions.

\begin{definition*}[Growth rate of vector measures]  Given any vector measure  $\mathcal{V}: \mathfrak{B}^b \rightarrow \fX$ which is not identically zero, we define  a growth rate function  $H_\mathcal{V}: [1, \infty) \rightarrow [0, \infty]$ as 
\begin{align}\label{def-rsigma}
H_\mathcal{V}(p):=  \limsup_{n\to\infty} \frac{\log \|T_n^\mathcal{V}\|_{L^p(\fX)}^p}{n \log b}  = p-1 + \limsup_{n\to\infty}\frac{ \log \Big(\sum_{|u|=n} \| \mathcal{V}(I_u)\|_{\fX}^p\Big)}{n\log b},
\end{align}
where $L^p(\fX) = L^p([0,1], dt; \fX)$.  
\end{definition*}

\begin{remark*}
For  vector measure $\mathcal{V}$ which is non-identically zero, we indeed have $H_\mathcal{V}(p) \ge 0$ for all $p\ge 1$. Indeed,  there exists $n_0\ge 0$ with $T_{n_0}^\mathcal{V}\ne 0$, then  by the martingale property, $\|T_n^\mathcal{V}\|_{L^p(\fX)} \ge \|T_{n_0}^\mathcal{V}\|_{L^p(\fX)}$ for all $n\ge n_0$ and hence 
\[
H_\mathcal{V}(p) =  \limsup_{n\to\infty} \frac{\log \|T_n^\mathcal{V}\|_{L^p(\fX)}^p}{n \log b} 
\ge   \limsup_{n\to\infty} \frac{\log \|T_{n_0}^\mathcal{V}\|_{L^p(\fX)}^p}{n \log b}  = 0.
\]
\end{remark*}

Note that  if $\fX= \R$ or $\fX=\C$ and   the vector measure $\mathcal{V}$ is a probability measure $\nu$ on $[0,1]$, then by comparing the definition \eqref{def-rsigma}  and the definition \eqref{def-D-p} for  the $L^p$-dimension $\ldim_p(\nu)$, we have 
\begin{align}\label{H=pdim}
\frac{H_\nu(p)}{p-1}=  1- \ldim_p(\nu) \quad \text{for all $p>1$}.
\end{align}

\subsection{A general $L^p$-estimate via martingale type inequalities}\label{sec-Mtype}

\begin{proposition}\label{prop-mart-type}
Assume that $\fX$ is a  Banach space   with martingale type  $p \in (1,2]$.  Then there exists a constant $C = C(p, \fX)$ such that for any  vector measure $\mathcal{V}: \mathfrak{B}^b \rightarrow \fX$, 
\begin{align}\label{sigma-mart-p}
\sup_{n\ge 1} \E\big[\big\|Q_n\mathcal{V}([0,1])\big\|_{\fX}^p\big] \le      C   \sum_{n\ge 0} \Big( \frac{\E[W^p]}{b^{p-1}}\Big)^n  \cdot    
\|T_n^\mathcal{V}\|_{L^p(\fX)}^p.
\end{align}
\end{proposition}

Proposition \ref{prop-mart-type} has a  useful consequence given as follows.  Given any random weight $W$, we define the structure function as  \footnote{For the convenience of our computations,  in this paper, $\varphi_W$ is defined differently  from  the standard structure function  (see \cite[Formula (4)]{Kahane-Peyriere-advance}    or \cite[Formula (3.9)]{Heu}), which is usually defined as 
$
\varphi_W(t)/\log b = \log_b( \E[W^t]) - (t-1). 
$ }
 \begin{align}\label{def-phi-W}
\varphi_W(t): = \log \E[W^t] - (t-1) \log b \in \R\cup\{+\infty\}, 
\end{align}
which is always finite for $0 < t \le 1$ (since $W$ is integrable with $\E[W]=1$). And $\varphi_W(t)<\infty$ if and only if $\E[W^t]<\infty$.    See Figure \ref{phi-graph} for an illustration of the function $\varphi_W$. 

\begin{figure}[H]
\begin{tikzpicture}[scale=1.8]
\draw[->](0,0)--(3.2,0)node[below]{$t$};
\draw[->](0,0)node[below]{$0$}--(0,1.5)node[left]{$\varphi_W(t)$};	
\draw[domain=0:2.5, blue]plot(\x,{0.5*\x^2-0.5*\x-(\x-1)*ln(3)});
\draw [fill] (1,0) circle [radius=0.4pt];
\draw [fill] (0,1.098612288668) circle [radius=0.4pt];
\draw (0,1.098612288668) node[left] {$\log b$};
\draw (1,0) node[below] {$1$};
\draw[]node[below]{$0$}; 
	\end{tikzpicture}\caption{Illustration of the function  $\varphi_W$.}\label{phi-graph}
\end{figure}
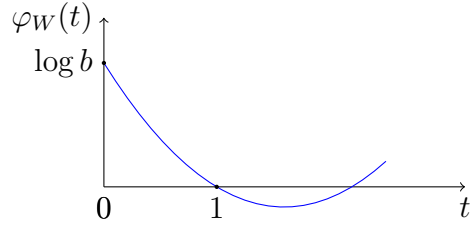

\begin{corollary}\label{cor-HM-comp}
Assume that $\fX$ is a  Banach space  having martingale type $p$ with $p \in (1,2]$ and let $\mathcal{V}: \mathfrak{B}^b([0,1])\rightarrow \fX$ be a vector measure.  Assume that $\E[W^p]<\infty$.  If
\begin{align}\label{H-exp-exp}
 H_\mathcal{V}(p) +  \varphi_W(p) <   0, 
\end{align}
then the martingale $(Q_n \mathcal{V}([0,1]))_{n\ge 1}$ is $L^p$-uniformly bounded. 
\end{corollary}

\begin{remark*}
The condition  \eqref{H-exp-exp} implies that $\varphi_W(p)< 0$ and hence $\E[W^p]<b^{p-1}$.  It follows that our condition \eqref{H-exp-exp} implies the $L^p$-boundedness of the martingale $(\mu_n([0,1]))_{n\ge 1}$ defined in \S \ref{sec-Mandelbrot-cascade}.  Note that when one considers the cascade action on scalar measures, Corollary~\ref{cor-HM-comp} and its proof are restatement of a part of \cite[Theorem A]{Fan-JMPA} and its proof. 
\end{remark*}

\begin{proof}[Proof of Proposition \ref{prop-mart-type}]
Recall the definition \eqref{def-Mtype} of martingale type $p$ of a Banach space.   Consider the martingale $(M_n^\mathcal{V})_{n\ge 1}$  defined in   \eqref{def-mart}.
By the martingale type $p$ assumption on $\fX$, there exists $C>0$ such that for any $N\ge 1$, 
\[
\E[\|M_N^\mathcal{V} \|_{\fX}^p] \le C   \sum_{n=0}^N \E[\|dM_n^\mathcal{V} \|_{\fX}^p],
\]
Now for each $n\ge 1$, since $\{\mathring{W}(u): |u|=n\}$ are mutually independent centered random variables, with respect to the conditional expectation $\E_{n-1}[\cdot] = \E[\cdot|\mathscr{F}_{n-1}]$,   we  apply again the martingale type $p$ inequality \eqref{def-ind-Mtype}  to the expression \eqref{def-dM} for $dM_n^\mathcal{V}$  and obtain 
\begin{align*}
\E_{n-1}[ \|dM_n^\mathcal{V}\|_{\fX}^p]  & =    \E_{n-1}\Big[   \Big\|  \sum_{|u|=n}  \Big(\prod_{j=0}^{n-1} W(u|_j)\Big) \mathring{W}(u) \cdot \mathcal{V}(I_u)   \Big\|_{\fX}^p \Big]  
\\
&  \le C    \sum_{|u|=n}    \E_{n-1}\Big[ \Big\| \Big(\prod_{j=0}^{n-1} W(u|_j)\Big) \mathring{W}(u) \cdot \mathcal{V}(I_u)   \Big\|_{\fX}^p\Big]
\\
& = C   \E[|\mathring{W}(u)|^p]    \cdot    \sum_{|u|=n}     \Big(\prod_{j=0}^{n-1} W(u|_j)^p \Big)  \cdot     \| \mathcal{V}(I_u)   \|_{\fX}^p. 
\end{align*}
Hence, by $\E[|\mathring{W}|^p] \le 2^p  \E[W^p]$, we have 
\[
\E[ \|dM_n^\mathcal{V}\|_{\fX}^p]  \le   C  2^p (\E[W^p])^{n}  \cdot    \sum_{|u|=n}       \| \mathcal{V}(I_u)   \|_{\fX}^p.  
\]
By noting the equality
\[
\|T_n^\mathcal{V}\|_{L^p(\fX)}^p = b^{n(p-1)} \sum_{|u|=n}  \| \mathcal{V}(I_u)\|_{\fX}^p, 
\]
we complete the proof of the proposition. 
\end{proof}

\begin{proof}[Proof of Corollary \ref{cor-HM-comp}]
The assumption \eqref{H-exp-exp} implies that there exists $\delta>0$ such that 
\[
 \|T_n^\mathcal{V}\|_{L^p(\fX)}^p  \le       \Big(  \frac{b^{p-1}}{\E[W^p]}  \Big)^n    e^{-n \delta} \quad  \text{for large enough $n$}. 
\]
Therefore,  Corollary \ref{cor-HM-comp} follows from Proposition \ref{prop-mart-type}.
\end{proof}

\section{Equivalent $L^p(\ell^q)$-estimate of vector-valued MC-actions}\label{sec-Lplq}

This section is devoted to the special case $\fX= \ell^q$ for  $2\le q<\infty$ (by  Proposition~\ref{prop-type},  $\fX$ has martingale type $2$). In this case, under a mild condition on the  initial random weight $W$,  the upper estimate  obtained in Proposition~\ref{prop-mart-type} can be improved to be a two sided {\it equivalent estimate}. 

For any $1< p <2 \le q <\infty$ and  any vector measure $\mathcal{V}: \mathfrak{B}^b \rightarrow \ell^q$, we define a characteristic quantity
\begin{align}\label{chi-VWpq}
\chi(\mathcal{V}, W, p, q): =   \E\Big[ \Big\| \sum_{|u|\ge 1}    \Big(  \prod_{j=0}^{|u|-1} W(u|_j)^2 \Big)   \cdot  \big| \mathcal{V}(I_u)  \big|^2       \Big\|_{\ell^{q/2}}^{p/2} \Big].
\end{align}
Recall the definition \eqref{def-mart} of $Q_n\mathcal{V}([0,1])$: 
\[
Q_n\mathcal{V} ([0,1]) = \sum_{|u|=n}  \Big(\prod_{j=0}^n W(u|_j)\Big) \cdot \mathcal{V}(I_u), \quad n\ge 1. 
\] 

\begin{theorem}\label{thm-precise}
Assume that $1<p<2\le q<\infty$.   Let  $\mathcal{V}: \mathfrak{B}^b\rightarrow \ell^q$ be a vector measure. Suppose that $\E[W^q]<\infty$.  Then
\begin{multline*}
\|\mathcal{V}([0,1])\|_{\ell^q}^p +    \|\mathring{W}\|_1^p\cdot    \chi(\mathcal{V}, W, p, q)   
\\
\lesssim_{p,q}\sup_{n\ge 1} \E \Big[ \big  \| Q_n\mathcal{V}([0,1])\big\|_{\ell^q}^p \Big]  \lesssim_{p,q} 
\\
  \lesssim_{p,q}    \|\mathcal{V}([0,1]) \|_{\ell^q}^p +     \|\mathring{W}\|_q^{p}
 \cdot  \chi(\mathcal{V}, W, p, q). 
\end{multline*}  
In particular, the martingale $(Q_n\mathcal{V}([0,1]))_{n\ge 1}$ is uniformly  bounded in $L^p(\ell^q)$  if and only if the characteristic quantity $\chi(\mathcal{V}, W, p, q)<\infty$. 
\end{theorem}

\begin{remark*}
To  obtain the lower estimate in Theorem \ref{thm-precise},  one does not need the condition $\E[W^q]<\infty$.  For the details, see the proof of the lower estimate in  \S \ref{sec-low}.   
\end{remark*}

\subsection{$\ell^q$-valued Burkholder-Rosenthal martingale inequalities} 

We shall need the $\ell^q$-valued Burkholder-Rosenthal martingale  inequalities obtained by Dirksen and Yaroslavtsev in \cite{PLMS-2019}.

Assume that $1<p<2\le q<\infty$.    For any  centered $\ell^q$-valued martingale $Z = (Z_n)_{n\ge 1}$,   define a norm  $\normmm{Z}_1$ by 
\begin{align*}
\normmm{Z}_1  :& = \Big(\E\Big\|   \Big(\sum_{n\ge 1}\mathbb{E}_{n-1} |dZ_n|^2 \Big)^{1/2}\Big\|_{\ell^q}^p \Big)^{1/p},
\end{align*}
where $\E_n$ denote the conditional expectations and $\E_0 = \E$.   Another two norms  $\normmm{Z}_2$ and $\normmm{Z}_3$ are defined by 
\begin{align*}
\normmm{Z}_2   : & =\Big(\mathbb{E}\Big[ \Big\{ \sum_{n\ge 1} \mathbb{E}_{n-1}  \|dZ_n \|_{\ell^q}^q \Big\}^{p/q}\Big]\Big)^{1/p} \an 
 \normmm{Z}_3  : =\Big(\sum_{n\ge 1} \mathbb{E} \|dZ_n \|_{\ell^q}^p\Big)^{1/p}. 
\end{align*}
Dirksen and Yaroslavtsev's vector-valued Burkholder-Rosenthal inequality    \cite[Theorem 1]{PLMS-2019} states that, under the assumption  $1<p<2\le q<\infty$, 
\[
\sup_{n\ge 1} \|Z_n\|_{L^p(\ell^q)} \approx_{p,q}   \inf_{Z = X+Y}  \Big( \max\{   \normmm{X}_1, \normmm{X}_2 \}   +  \normmm{Y}_3  \Big),
\] 
where the infimum ranges over all decompositions $Z = X+Y$ with $X = (X_n)_{n\ge 1}$ and $Y = (Y_n)_{n\ge 1}$ being centered $\ell^q$-valued martingales.

In particular,  for any $\ell^q$-valued martingale $M = (M_n)_{n\ge 1}$,   Dirksen and Yaroslavtsev's vector-valued Burkholder-Rosenthal inequality    \cite[Theorem 1]{PLMS-2019} applied  to the centered martingale 
\[
\mathring{M} : = (M_n-\E[M_n])_{n\ge 1} = (M_n-\E[M_1])_{n\ge 1}
\]
 and to the trivial decomposition 
 \[
 \mathring{M}= \mathring{M} +0
 \]
  yields the following   simple  form  of  upper estimate:
\begin{align}\label{BR-vec}
\begin{split}
\sup_{n\ge 1} \|M_n\|_{L^p(\ell^q)}^p \lesssim_{p,q} &    \|\E[M_1]\|_{\ell^q}^p +    \E\Big[ \Big\|\Big(\sum_{n\ge 1}\mathbb{E}_{n-1}\big[|dM_n|^2\big]\Big)^{1/2}\Big\|_{\ell^q}^p \Big] 
\\
& \quad +   \E\Big[\Big(\sum_{n\ge1}\E_{n-1}\big[\|dM_n\|_{\ell^q}^q\big]\Big)^{p/q}\Big].
\end{split}
\end{align}
The above simple  form  \eqref{BR-vec} of  upper estimate of \cite[Theorem 1]{PLMS-2019}   is enough for our purpose.

\subsection{The upper estimate}

Consider the total order on $\Ab^*$ defined as follows: for any $u, v\in \Ab^*$, we write  $u<v$ if and only if  either $|u|<|v|$ or $|u| = |v|$ and 
\[
\sum_{j} u_jb^{-j} < \sum_{j} v_j b^{-j}. 
\] 
With respect to the above total order on $\Ab^*$, we  can define an increasing filtration of sigma-algebras $(\mathscr{G}_u)_{u\in \Ab^*}$ by 
\[
\mathscr{G}_u: = \sigma\Big(\Big\{W(v): v\le u\Big\}\Big) \quad  \text{ with $|u|\ge 1$}.
\]
By convention,  if $u = \emptyset\in \Ab^*$ is the empty word,  let $\mathscr{G}_\emptyset$ be the trivial sigma-algebra.

For any $u\in \Ab^*$ with $|u|\ge 1$,  let  $u^{-}$ denote  the unique element in $\Ab^*$ which is largest among all elements strictly smaller than $u$, in notation,  $u^{-}$ is defined as 
\[
u^{-}: = \max\{v\in \Ab^*: v<u \}.
\]
 We denote the martingale difference operator 
\[
D_u(\cdot)= \E_u[\cdot] - \E_{u^{-}} [\cdot],
\]
where $\E_u[\cdot]$ and $\E_{u^{-}} [\cdot]$ denote the conditional expectations $\E_u[\cdot] =\E[\cdot|\mathscr{G}_u]$ and $\E_{u^{-}}[\cdot] =\E[\cdot|\mathscr{G}_{u^{-}}]$. 

For the vector measure $\mathcal{V}: \mathfrak{B}^b([0,1])\rightarrow \ell^q$,  we have 
\begin{align}\label{def-MNV}
\begin{split}
M_N^\mathcal{V}: = Q_N\mathcal{V}([0,1]) & =       \mathcal{V}([0,1]) + \sum_{n=1}^N  \sum_{|u|=n}  \Big(  \prod_{j=0}^{n-1} W(u|_j) \Big) \mathring{W}(u) \mathcal{V}(I_u)
\\
&  =  \mathcal{V}([0,1]) + \sum_{1\le |u|\le N}  \Big(  \prod_{j=0}^{|u|-1} W(u|_j) \Big) \mathring{W}(u) \mathcal{V}(I_u). 
\end{split}
\end{align}
Clearly,  for all $u\in \Ab^*$ with $1\le |u|\le N$, 
\begin{align}\label{Du-form}
\Big(  \prod_{j=0}^{|u|-1} W(u|_j) \Big) \mathring{W}(u) \mathcal{V}(I_u) = D_u \big[ M_N^\mathcal{V}\big]. 
 \end{align}
 By convention, we write $D_\emptyset [M_N^\mathcal{V}] = \mathcal{V}([0,1])$. Then 
 \begin{align}\label{MN-diff}
 M_N^\mathcal{V} = \sum_{|u|\le N} D_u\big[M_N^\mathcal{V}\big]. 
 \end{align}
Consequently, by  the vector-valued Burkholder-Rosenthal inequalities \eqref{BR-vec}, 
\begin{align}\label{MnV-norm}
\begin{split}
\| M_N^\mathcal{V}\|_{L^p(\ell^q)}^p \lesssim_{p,q} &  \|\mathcal{V}([0,1]) \|_{\ell^q}^p +   \underbrace{ \E\Big[ \Big\|\Big(\sum_{1\le |u|\le N}\E_{u^{-}}\big[\big|D_u \big[ M_N^\mathcal{V}\big]\big|^2\big]\Big)^{1/2}\Big\|_{\ell^q}^p \Big]}_{\text{denoted $T_1(N)$}} 
\\
& +   \underbrace{\E\Big[\Big(\sum_{1\le |u|\le N}\E_{u^{-}}\big[\big\|  D_u \big[ M_N^\mathcal{V}\big] \big \|_{\ell^q}^q\big]\Big)^{p/q}\Big]}_{\text{denoted $T_2(N)$}}.
\end{split}
\end{align}

Our next goal is to  control  the two terms $T_1(N)$ and $T_2(N)$. 

\begin{lemma}\label{lem-I}
Given any integer $N\ge 1$. For any $u\in \Ab^*$ with $1\le |u|\le N$, we have 
\[
\E_{u^{-}}\Big[\big|D_u \big[ M_N^\mathcal{V}\big]\big|^2\Big] =\Var(W)  \cdot \Big(  \prod_{j=0}^{|u|-1} W(u|_j)^2 \Big)   \cdot  \big| \mathcal{V}(I_u)  \big|^2. 
\]
In particular, $\E_{u^{-}}\big[\big|D_u \big[ M_N^\mathcal{V}\big]\big|^2\big]$ does not depend on $N$ provided that $1\le |u|\le N$. Hence  
\begin{align*}
 T_1 (N)=    \Var(W)^{p/2} 
 \cdot  \E\Big[ \Big\| \sum_{1\le |u|\le N}    \Big(  \prod_{j=0}^{|u|-1} W(u|_j)^2 \Big)   \cdot  \big| \mathcal{V}(I_u)  \big|^2       \Big\|_{\ell^{q/2}}^{p/2} \Big]. 
\end{align*}
\end{lemma}
\begin{proof}
It follows from \eqref{Du-form} and  the equality $\| |f|^{1/2}\|_{\ell^q}^2   = \|  f\|_{\ell^{q/2}}$ for any $f\in \ell^{q/2}$. 
\end{proof}

\begin{lemma}\label{lem-II-le-I}
Assume that $
\E[W^q]<\infty$. Then 
\begin{align*}
T_2(N) \le   \|\mathring{W}\|_q^p   
 \cdot   \E\Big[ \Big\| \sum_{1\le |u|\le N}    \Big(  \prod_{j=0}^{|u|-1} W(u|_j)^2 \Big)   \cdot  \big| \mathcal{V}(I_u)  \big|^2       \Big\|_{\ell^{q/2}}^{p/2} \Big].
\end{align*}
\end{lemma}

\begin{lemma}\label{lem-convex}
For any $q\ge 2$, and any family of sequences $f_i: \N\rightarrow \C$, we have 
\[
\Big\| \sum_{i} |f_i|^2 \Big\|_{\ell^{q/2}} \ge    \Big(\sum_i \| f_i\|_{\ell^q}^q\Big)^{2/q}. 
\]
\end{lemma}
\begin{proof}
Since $q\ge 2$,  by the   inequality $(a+b)^t \ge a^t + b^t$ for  any $a, b \ge 0$ and $t \ge 1$, we get 
\begin{align*}
\Big\| \sum_{i} |f_i|^2 \Big\|_{\ell^{q/2}}^{q/2}   = \sum_{s\in \N} \Big(\sum_{i}  |f_i(s) |^2\Big)^{q/2} \ge \sum_{s\in \N} \sum_{i}  |f_i(s) |^q  =  \sum_{i}   \sum_{s\in \N} |f_i(s) |^q   =  \sum_i \| f_i\|_{\ell^q}^q. 
\end{align*} 
This is the desired inequality. 
\end{proof}

\begin{proof}[Proof of Lemma \ref{lem-II-le-I}]
Clearly,  for any $u\in \Ab^*$ with $1\le |u|\le N$, 
\[
\E_{u^{-}}\Big[\big\|  D_u \big[ M_N^\mathcal{V}\big] \big \|_{\ell^q}^q\Big]  =   \E\big[ |\mathring{W}|^q \big] \cdot  \Big(  \prod_{j=0}^{|u|-1} W(u|_j)^q \Big)   \cdot  \|\mathcal{V}(I_u)\|_{\ell^q}^q .
\]
Hence by Lemma \ref{lem-convex}, 
\begin{align*}
& \E\Big[\Big(\sum_{1\le |u|\le N}\E_{u^{-}}\big[\big\|  D_u \big[ M_N^\mathcal{V}\big] \big \|_{\ell^q}^q\big]\Big)^{p/q}\Big]
\\
= &    \|\mathring{W}\|_q^p   \cdot   \E\Big[\Big(\sum_{1\le |u|\le N}      \Big(  \prod_{j=0}^{|u|-1} W(u|_j)^q \Big)   \cdot  \|\mathcal{V}(I_u)\|_{\ell^q}^q   \Big)^{p/q}\Big] 
\\
\le & \|\mathring{W}\|_q^p   \cdot   \E\Big[ \Big\| \sum_{1\le |u|\le N}    \Big(  \prod_{j=0}^{|u|-1} W(u|_j)^2 \Big)   \cdot  \big| \mathcal{V}(I_u)  \big|^2       \Big\|_{\ell^{q/2}}^{p/2} \Big].
\end{align*}
This ends the proof. 
\end{proof}

\begin{proof}[Proof of the upper estimate of Theorem \ref{thm-precise}]
Note that  $\Var(W)  =   \|\mathring{W}\|_2^2 \le  \|\mathring{W}\|_q^2$. 
By  \eqref{MnV-norm}, Lemma \ref{lem-I}, Lemma \ref{lem-II-le-I},  we obtain the desired upper estimate of Theorem~\ref{thm-precise}. 
\end{proof}

\subsection{The lower estimate}\label{sec-low}

Recall the product space  $\Delta  = \Delta_I= \{\pm 1\}^I$ equipped with the product measure $d\epsilon$ in \eqref{def-dep}.  
\begin{lemma}\label{lem-abs}
Let $(\Sigma,  \mathcal{B},  \rho)$  be  a measure space equipped  with a non-negative $\sigma$-finite measure  $\rho$. 
 Consider any $p_1\ge 1$, then for any $0<p_2<\infty$ and  any finite sequence  of vector-valued functions $(g_i)_{i}$ in $L^{p_2}(\Sigma, \rho; \ell^{p_1})$, 
 \begin{align}\label{KK-abs}
\big \| \sum_i \epsilon_i g_i\big\|_{L^{p_2}(\Delta\times \Sigma; \, \ell^{p_1})} \asymp_{p_1, p_2}  \big \| \sum_i \epsilon_i |g_i|\big\|_{L^{p_2}(\Delta\times \Sigma; \, \ell^{p_1})},
 \end{align}
 where  $\Delta\times \Sigma$ is equipped with the measure $d\epsilon \otimes d\rho$.
\end{lemma}
 
\begin{proof}
  Indeed, by the Fubini Theorem, we have 
 \[
 \big \| \sum_i \epsilon_i g_i\big\|_{L^{p_2}(\Delta\times \Sigma; \, \ell^{p_1})} = \big \| \sum_i \epsilon_i g_i\big\|_{L^{p_2}( \Sigma; \, L^{p_2} (\Delta; \, \ell^{p_1}))}. 
 \]
Therefore,   \eqref{KK-abs} follows from the following: for any finite sequence  $(x_i)_{i}$ in $\ell^{p_1}$, 
 \[
 \big \| \sum_i \epsilon_i x_i\big\|_{L^{p_2} (\Delta; \, \ell^{p_1})} \asymp_{p_1, p_2}  \big \| \sum_i \epsilon_i |x_i|\big\|_{L^{p_2}(\Delta; \, \ell^{p_1})}.
 \]
 By Kahane-Khintchine inequalities \eqref{KK-ineq}, it suffices to show 
 \begin{align}\label{suf-to-prove}
 \big \| \sum_i \epsilon_i x_i\big\|_{L^{p_1}(\Delta; \, \ell^{p_1})} \asymp_{p_1} \big \| \sum_i \epsilon_i |x_i|\big\|_{L^{p_1}(\Delta; \, \ell^{p_1})}. 
 \end{align}
 Recall the classical  Khintchine inequalities for scalars: for any finite sequence $(a_i)_{i}$ in $\C$, 
 \[
 \big \| \sum_i \epsilon_i a_i\big\|_{L^{p_1}(\Delta)}  \asymp_{p_1} \big \| \sum_i \epsilon_i a_i\big\|_{L^2(\Delta)}  = \| \sum_i \epsilon_i |a_i|\big\|_{L^2(\Delta)}  \asymp_{p_1} \| \sum_i \epsilon_i |a_i|\big\|_{L^{p_1}(\Delta)}. 
 \]
 The desired equivalence \eqref{suf-to-prove} follows immediately. 
\end{proof}

\begin{proof}[Proof of the lower estimate of Theorem \ref{thm-precise}]
Recall the equalities \eqref{def-MNV}, \eqref{Du-form} and \eqref{MN-diff}.  By Burkholder inequality for the UMD Banach space $\ell^q$, for any $N\ge 1$, we have 
\begin{align*}
\|M_N^\mathcal{V}\|_{L^p(\ell^q)}   & \asymp_{p,q}   \Big \| \sum_{|u|\le N} \epsilon_u  D_u\big[M_N^\mathcal{V}\big] \Big\|_{L^p(\Delta\times \Omega; \, \ell^q)}
\\
&  \asymp_{p,q}  \Big \| \sum_{|u|\le N} \epsilon_u  \big| D_u\big[M_N^\mathcal{V}\big]\big| \Big\|_{L^p(\Delta \times \Omega; \, \ell^q)}  \quad \text{(by \eqref{KK-abs})}
\\
  & =  \quad \Big \| \sum_{|u|\le N} \epsilon_u  \big| D_u\big[M_N^\mathcal{V}\big]\big| \Big\|_{L^p(\Delta; \, L^p( \Omega; \, \ell^q))}
\\
&   \gtrsim_{p,q}   \Big \| \sum_{|u|\le N} \epsilon_u   \E_{u^{-}} \big[ \big| D_u\big[M_N^\mathcal{V}\big]\big|\big] \Big\|_{L^p(\Delta; \, L^p( \Omega; \, \ell^q))}  ( \text{Bourgain-Stein \eqref{BS-ineq}})
\\
& = \quad  \Big \| \sum_{|u|\le N} \epsilon_u   \E_{u^{-}} \big[ \big| D_u\big[M_N^\mathcal{V}\big]\big|\big] \Big\|_{L^p(\Omega; \, L^p( \Delta; \, \ell^q))}. 
\end{align*}
Now by \eqref{Du-form}, for any $1\le |u|\le N$,  we have
\[
\E_{u^{-}} \big[ \big| D_u\big[M_N^\mathcal{V}\big]\big|\big] =  \E[|\mathring{W}|] \cdot \Big(  \prod_{j=0}^{|u|-1} W(u|_j) \Big)  \cdot |\mathcal{V}(I_u) |. 
\]
Hence, by Kahane-Khintchine inequalities, 
\begin{align*}
\|M_N^\mathcal{V}\|_{L^p(\ell^q)}  &  \gtrsim_{p,q}  \Big\| \epsilon_\emptyset \mathcal{V}([0,1]) +  \E[|\mathring{W}|] \sum_{1\le |u|\le N}  \epsilon_u   \cdot \Big(  \prod_{j=0}^{|u|-1} W(u|_j) \Big)  \cdot |\mathcal{V}(I_u) |  \Big\|_{L^p(\Omega; \, L^p( \Delta; \, \ell^q))} 
\\
& \asymp_{p,q}   \Big\| \epsilon_\emptyset \mathcal{V}([0,1]) +  \E[|\mathring{W}|] \sum_{1\le |u|\le N}  \epsilon_u   \cdot \Big(  \prod_{j=0}^{|u|-1} W(u|_j) \Big)  \cdot |\mathcal{V}(I_u) |  \Big\|_{L^p(\Omega; \, L^q( \Delta; \, \ell^q))},
\end{align*}
where in the last equivalence, we have replaced the norm of $L^p(\Delta; \ell^q)$ by that of  $L^q(\Delta; \ell^q)$.

Note that  for any non-negative sequences $f, g \in \ell^{q/2}$ (with $q\ge 2$), 
 \begin{align}\label{pos-f-g}
 \| f+ g\|_{\ell^{q/2}} \ge 2^{\frac{2}{q}-1} (\| f\|_{\ell^{q/2}} + \| g\|_{\ell^{q/2}}).
 \end{align}
 This combined with  \eqref{LpLq-abs}  yields the following pointwise (with respect to $\Omega$) estimate  : 
\begin{align*}
&  \Big\| \epsilon_\emptyset \mathcal{V}([0,1]) +  \E[|\mathring{W}|] \sum_{1\le |u|\le N}  \epsilon_u   \cdot \Big(  \prod_{j=0}^{|u|-1} W(u|_j) \Big)  \cdot |\mathcal{V}(I_u) |  \Big\|_{L^q( \Delta; \, \ell^q)}
 \\
&\asymp_{p,q}    \Big\|  \Big ( |\mathcal{V}([0,1])|^2 +  (\E[|\mathring{W}|])^2 \sum_{1\le |u|\le N}   \Big(  \prod_{j=0}^{|u|-1} W(u|_j)^2 \Big)  \cdot |\mathcal{V}(I_u) |^2 \Big)^{1/2}\Big\|_{\ell^q}
\\
&=  \quad  \Big\|   |\mathcal{V}([0,1])|^2 +  (\E[|\mathring{W}|])^2 \sum_{1\le |u|\le N}   \Big(  \prod_{j=0}^{|u|-1} W(u|_j)^2 \Big)  \cdot |\mathcal{V}(I_u) |^2  \Big\|_{\ell^{q/2}}^{1/2}
\\
& \gtrsim_{q}  \| \mathcal{V}([0,1])\|_{\ell^q} +   (\E[|\mathring{W}|]) \cdot  \Big\|    \sum_{1\le |u|\le N}   \Big(  \prod_{j=0}^{|u|-1} W(u|_j)^2 \Big)  \cdot |\mathcal{V}(I_u) |^2  \Big\|_{\ell^{q/2}}^{1/2} . 
 \end{align*}
Finally, by applying an inequality similar to \eqref{pos-f-g} valid for $L^p(\Omega)$ and letting $N\to\infty$, we obtain 
 \begin{align*}
  \sup_{n\ge 1}\|M_n^\mathcal{V}\|_{L^p(\ell^q)}^p  &  \gtrsim_{p,q}    \|\mathcal{V}([0,1])\|_{\ell^q}^p +    (\E[|\mathring{W}|])^p\cdot    \E\Big[ \Big\|   \sum_{ |u|\ge 1}   \Big(  \prod_{j=0}^{|u|-1} W(u|_j)^2 \Big)  \cdot |\mathcal{V}(I_u) |^2  \Big\|_{\ell^{q/2}}^{p/2}\Big]. 
 \end{align*}
 This ends the proof of the lower estimate. 
\end{proof}

\section{Sharp two-sided $L^p(\ell^q)$-bounds for MCCM}\label{sec-mccm}
We now turn back to the Mandelbrot's canonical models,  namely, the MC-action on the Lebesgue measure on $[0,1]$. Recall the resulting sequence of random measures $(\mu_n)_{n\ge 1}$ given by \eqref{def-Lebn} and the limiting cascade measure $\mu_\infty$.   In this section, we shall always assume  that the Mandelbrot-Kahane's non-degeneracy condition 
\begin{align}\label{K-cond}
\E[W\log W]<\log b
\end{align}
 is satisfied. Then the random measures $\mu_\infty$ is non-degenerate and the martingale $(\mu_n([0,1]))_{n\ge 1}$ is uniformly integrable and hence 
\begin{align}\label{mu-mart}
\widehat{\mu}_n(s)= \E_n[\widehat{\mu}_\infty(s)] = \E[\widehat{\mu}_\infty(s) |\mathscr{F}_n] \,\, a.s. \,\, \text{for all $  s\in \N$}. 
\end{align}

If $\E[W^2]<\infty$,  then the new random weight $W^{(2)} = W^2/\E[W^2]$ can be used to construct new MC-action. We call this new MC-action the squared MC-action and denote the corresponding operators by $Q^{(2)}_n$ and $Q^{(2)}$.   The squared-MC action on Lebesgue measure $\lambda(dt)= dt$ naturally gives rise to the following martingale (which we call the squared martingale): 
\begin{align}\label{eqn-def-mnab}
\mathscr{M}_0(W^{(2)})=1 \an 
\MAB:=(Q^{(2)} \lambda) ([0,1]) = \frac{1}{b^n}\sum_{|u|=n}\prod_{j=1}^n \frac{W(u|_j)^2}{\E[W^2]} \quad \text{for $n\geq 1$}.
\end{align}

Given $0\le \alpha<1$ and $1< p\le 2 \le q <\infty$ with $q> \frac{1}{1 - \alpha}$.  Let $\mathcal{R}= (\mathcal{R}_n(W, b, \alpha, p, q))_{n\ge 1}$ be the sequence  given by 
\begin{align}\label{def-Rn}
\mathcal{R}_n  (W, b, \alpha, p, q)  : =  \E \big[  \MAA^{\frac{p}{2}}\big] \cdot \Big(  \frac{\E[W^2]}{b^{1- 2\alpha- \frac{2}{q}}}\Big)^{\frac{np}{2}}. 
\end{align}
Recall  the $(\alpha,p,q)$-norm  $\mathcal{N}^{(\alpha, p, q)}(\mu_\infty)$  defined in  \eqref{eq-npq}:
\begin{align*}
\mathcal{N}^{(\alpha, p, q)}(\mu_\infty)=   \Big(\E\Big[  \Big\{ \sum_{n \ge 1}    \big| n^{\alpha} \cdot \widehat{\mu}_\infty(n)\big|^q  \Big\}^{\frac{p}{q}}\Big] \Big)^{\frac{1}{p}}   = \Big\|  \big( n^{\alpha} \cdot \widehat{\mu}_\infty(n) \big)_{n \ge 1}  \Big\|_{L^p(\PP; \, \ell^q)}\in[0,\infty].
\end{align*}
\begin{theorem}\label{thm-CCM}
Let $0\le \alpha<1$.  Assume that $1< p\le 2 \le q <\infty$ and $q> \frac{1}{1 - \alpha}$.  Assume moreover that $\E[W^q]<\infty$. Then there exist two constants $c_1, c_2>0$ depending on $b, \alpha, p, q$ such that 
\begin{align*}
c_1  \cdot   \frac{\|\mathring{W}\|_1}{\|W\|_2}  \cdot  \big\|\mathcal{R}\big\|_{\ell^{\frac{p}{q}}}^{\frac{1}{p}}
 \le   \NQ(\mu_\infty) 
\le c_2  \cdot  \|\mathring{W}\|_q     \cdot \big\|\mathcal{R}\big\|_{\ell^1}^{\frac{1}{p}}.
\end{align*}

\end{theorem}

\subsection{Initial step of the proof of Theorem \ref{thm-CCM}}

Under the Mandelbrot-Kahane's non-degeneracy condition \eqref{K-cond}, by  \eqref{mu-mart} and the standard fact about uniformly integrable vector-valued martingales in a finite-dimensional Banach space, for any finite  $N\ge 1$ and any $0\le \alpha<1$, 
\[
\E\Big[\Big\{\sum_{s =  1}^N \big|s^{\alpha} \widehat{\mu}_\infty (s)\big|^q\Big\}^{p/q}\Big]  = \sup_{n\ge 1}  \E\Big[\Big\{\sum_{s =  1}^N \big|s^{\alpha} \widehat{\mu}_n(s)\big|^q\Big\}^{p/q}\Big]. 
\]
By monotone convergence theorem, one can  easily see that
\begin{align*}
\big[\NQ(\mu_\infty)\big]^p =  \E\Big[\Big\{\sum_{s =  1}^\infty \big|s^{\alpha} \widehat{\mu}_\infty(s)\big|^q\Big\}^{\frac{p}{q}}\Big]  & = \sup_{n\ge 1}  \E\Big[\Big\{\sum_{s =  1}^\infty \big|s^{\alpha} \widehat{\mu}_n(s)\big|^q\Big\}^{\frac{p}{q}}\Big] 
 =  \sup_{n\ge 1} \E \Big[ \big  \| Q_n\mathcal{V}([0,1])\big\|_{\ell^q}^p \Big], 
\end{align*}
where  $(Q_n\mathcal{V}([0,1]))_{n\ge 1}$ is the $\ell^q$-valued martingale  defined in \eqref{def-mart} with respect to  the vector measure  $\mathcal{V} = \mathcal{V}_{\alpha, q}: \mathfrak{B}^b \rightarrow \ell^q$ defined as (with $q> \frac{1}{1- \alpha}$)
\begin{align}\label{def-nuq}
\mathcal{V}(A): = \mathcal{V}_{q,\alpha}(A) =   \Big( s^\alpha \int_{A}e^{-i 2\pi st} dt \Big)_{s\ge 1} \in \ell^q \quad \text{for all $A\in \mathfrak{B}^b([0,1])$}.
\end{align}
Therefore, by Theorem \ref{thm-precise}, we only need to estimate the quantity $\chi(\mathcal{V}, W, p, q)$ defined in \eqref{chi-VWpq}: 
\[
\chi(\mathcal{V}, W, p, q) = \E \Big[   \Big\{     \sum_{s \ge 1}       \Big(   \sum_{|u|\ge 1}    \Big(  \prod_{j=0}^{|u|-1} W(u|_j)^2 \Big)   \cdot  \big| \mathcal{V}(I_u)(s)  \big|^2    \Big)^{q/2}  \Big\}^{p/q}\Big]. 
\]
In particular, since $\mathcal{V}([0,1])=0$ is the zero vector,  we have 
\begin{align}\label{2-side-zero}
 \|\mathring{W}\|_1^p \cdot  \chi(\mathcal{V}, W, p, q) \lesssim_{p,q}       \big[ \NQ(\mu_\infty)\big]^p    \lesssim_{p,q}   \|\mathring{W}\|_q^p \cdot \chi(\mathcal{V}, W, p, q). 
\end{align}
 
In the current situation, by direct computation,  for any $u\in \Ab^*$ with $|u| = n$, we have 
\begin{align}\label{V-exp}
\big| \mathcal{V}(I_u)(s)  \big| = \frac{s^{-(1-\alpha)}|  \exp(i  2\pi s b^{-n}) -1|}{2\pi} \quad \text{for all $s\ge 1$},
\end{align}
which depends only on $s$ and  the level $|u|$ of $u\in \Ab^*$.  
Then, by setting
\[
\Theta_\alpha(n, s):  =\frac{s^{-2(1-\alpha)} |  \exp(i  2\pi s b^{-n}) -1|^2}{4\pi^2} \quad \text{for all $n, s \ge 1$}, 
\]
we obtain 
\begin{align*}
\chi(\mathcal{V}, W, p, q)  &  = \E \Big[   \Big\{    \sum_{s \ge 1}       \Big(   \sum_{|u|\ge 1}    \Big(  \prod_{j=0}^{|u|-1} W(u|_j)^2 \Big)    \Theta_\alpha(|u|, s)    \Big)^{q/2}  \Big\}^{p/q}\Big]. 
\end{align*}

\subsection{Decomposition of $\chi(\mathcal{V}, W, p,q)$}
We are going to use the following elementary properties of the quantity $\Theta_\alpha(n,s)$: 
\begin{itemize}
\item for all $n, s \in \N$, 
\begin{align}\label{trivial-bdd}
  \Theta_\alpha(n, s)\le  s^{-2(1-\alpha)};
\end{align}
\item  there exist two constants $c_1, c_2>0$ such that 
\begin{align}\label{fine-bdd}
c_1  s^{2\alpha} b^{-2n}  \le \Theta_\alpha(n, s) \le  c_2   s^{2\alpha} b^{-2n} \quad \text{for all $1\le s\le b^n$}.
\end{align}
\end{itemize}
According to \eqref{trivial-bdd} and \eqref{fine-bdd},  we can bound $\chi(\mathcal{V}, W, p, q)$ from both sides by 
\begin{align}\label{chi-2-12}
\chi_2\lesssim_{p,q}  \chi(\mathcal{V}, W, p, q)   \lesssim_{p,q} \chi_1 + \chi_2,
\end{align}
where $\chi_1$ and $\chi_2$ are given by 
\begin{align}\label{def-chi-I}
\chi_1 : =  \E \Big[   \Big\{     \sum_{s \ge 1}       \Big(   \sum_{|u|\ge 1}    \Big(  \prod_{j=0}^{|u|-1} W(u|_j)^2 \Big)   \cdot    s^{-2 (1-\alpha)} \cdot  \indi (s \ge  b^{|u|}) \Big)^{q/2}  \Big\}^{p/q}\Big]
\end{align}
and 
\begin{align}\label{def-chi-II}
\chi_2: = \E \Big[   \Big\{     \sum_{s \ge 1}       \Big(   \sum_{|u|\ge 1}    \Big(  \prod_{j=0}^{|u|-1} W(u|_j)^2 \Big)   \cdot    s^{2 \alpha} \cdot   b^{-2|u|}  \cdot \indi(1\le s\le b^{|u|})  \Big)^{q/2}  \Big\}^{p/q}\Big].
\end{align}

Our next goal is to obtain a sharp upper estimate of both quantities  $\chi_1, \chi_2$ and a sharp lower estimate of the quantity $\chi_2$.

\subsection{Outline of the proof of upper estimate of $\chi_1, \chi_2$}\label{sec-outline}
Here we briefly outline the strategy for obtaining a sharp upper estimate of $\chi_1, \chi_2$.

{\flushleft \it Step 1.  Littlewood-Paley decomposition.}
 We shall first use the Littlewood-Paley decomposition for the summation $\sum_{s\ge 1}$ in both expressions \eqref{def-chi-I} and \eqref{def-chi-II} for $\chi_1$ and $\chi_2$: 
\begin{align}\label{sum-sum}
\sum_{s\ge 1}  f(s) = \sum_{l=1}^\infty \sum_{b^{l-1}\le s < b^l}  f(s),
\end{align}
where $f(s)$ has the form 
\[
f(s) =\Big( \sum_{|u|\ge 1}  x(s, u) \Big)^{q/2}. 
\]
{\flushleft \it Step 2. Pointwise control of  the summands in $b$-adic intervals. }  
We  then control  pointwisely all the summands in $\sum_{b^{l-1}\le s< b^l}$ by a single term $T_l$ (independent of $s$): 
\[
  \max_{b^{l-1}\le s<b^l} f(s)  =     \max_{b^{l-1}\le s<b^l}    \Big( \sum_{|u|\ge 1}  x(s, u) \Big)^{q/2}  \le   T_l: =  \Big( \sum_{|u|\ge 1}   x^*(l, u) \Big)^{q/2},
\] 
where $x^*(l, u)$ is taken to be a suitable random variable such that 
\[
x^*(l, u) \ge  \max_{b^{l-1}\le s<b^l}    x(s, u). 
\]
Hence we obtain an upper estimate of the form
\[
\sum_{s=1}^\infty f(s) \lesssim \sum_{l=1}^\infty b^l T_l   =  \sum_{l=1}^\infty   b^l \Big(\sum_{|u|\ge 1} x^*(u, l) \Big)^{q/2}. 
\]
Then  by applying  the subadditive inequality \eqref{2-ele-ineq} (since  $p/q\le 1$), we obtain 
\begin{align}\label{two-times}
\Big\{\sum_{s=1}^\infty f(s)\Big\}^{p/q}  \lesssim \Big\{\sum_{l=1} b^l T_l\Big\}^{p/q} \le  \sum_{l=1}^\infty b^{lp/q}  \cdot T_l^{p/q} =  \sum_{l=1}^\infty b^{lp/q}  \cdot \Big(\sum_{|u|\ge 1} x^*(u, l) \Big)^{p/2}. 
\end{align}

{\flushleft \it Step 3. Smoothing over spherical summation blocks.}  Since $p/2\le 1$, we may again apply the subadditive inequality \eqref{2-ele-ineq} to obtain an upper estimate of $T_l^{p/q}$: 
\[
\sum_{|u|\ge 1} = \sum_{n=1}^\infty\sum_{|u|=n}. 
\]
More precisely, we have 
\[
T_l^{p/q} = \Big(\sum_{|u|\ge 1} x^*(u, l) \Big)^{p/2}\le   \sum_{|u|\ge 1} x^*(u, l)^{p/2}. 
\]
However, it turns out that the above upper estimate of $T_l^{p/q}$  is not  sufficient for our purpose. Instead,   we need to first decompose the summation in $T_l^{p/q}$ into summations over smaller blocks according to the  spheres of the tree $\Ab^*$ and then apply the subadditive inequality \eqref{2-ele-ineq}: 
\[
T_l^{p/q} =  \Big(\sum_{|u|\ge 1} x^*(u, l) \Big)^{p/2} =   \Big(\sum_{n=1}^\infty \sum_{|u| = n } x^*(u, l) \Big)^{p/2}  \le \sum_{n=1}^\infty  \Big ( \sum_{|u| = n } x^*(u, l) \Big)^{p/2}. 
\] 
And we obtain the upper estimate: 
\begin{align*}
\Big\{\sum_{s=1}^\infty f(s)\Big\}^{p/q} \lesssim  \sum_{l=1}^\infty b^{lp/q}  \cdot \sum_{n=1}^\infty  \Big ( \sum_{|u| = n } x^*(u, l) \Big)^{p/2}. 
\end{align*}

{\flushleft \it Step 4. Separation-of-variables reduction of $x^*(u, l)$. } 
The random variables $x^*(u, l)$ admit the following  separation-of-variables upper estimate:  
\[
x^*(u, l) \lesssim     A(u)B_n(l)  \quad \text{for $|u|=n$}
\]
with $B_n(l)\ge 0$  deterministic and $A(u)\ge 0$ random.  
Therefore, 
\begin{align*}
\Big\{\sum_{s=1}^\infty f(s)\Big\}^{p/q}  & \lesssim  \sum_{l=1}^\infty b^{lp/q}  \cdot \sum_{n=1}^\infty  \Big ( \sum_{|u| = n } A(u) B_n(l) \Big)^{p/2}
\\
& =   \sum_{n=1}^\infty  \Big ( \sum_{|u| = n } A(u) \Big)^{p/2}  \Big(\sum_{l=1}^\infty b^{lp/q}  B_n(l)^{p/2}\Big). 
\end{align*}
Since $B_n(l)$ are deterministic, we  finally obtain the following  upper estimate which is already convenient for further analysis: 
\begin{align*}
\E\Big[ \Big\{\sum_{s=1}^\infty f(s)\Big\}^{p/q} \Big]  \lesssim     \sum_{n=1}^\infty   \E\Big[  \Big ( \sum_{|u| = n } A(u) \Big)^{p/2}  \Big] \Big(\sum_{l=1}^\infty b^{lp/q}  B_n(l)^{p/2}\Big). 
\end{align*}

{\flushleft \it Step 5. Small moments of the squared martingales. } The quantity  $\E[   ( \sum_{|u| = n } A(u) )^{p/2} ]$  in Step 4 turns out to be related to the $p/2$ moment of the following martingale  associated to the  random weight $W^2/\E[W^2]$ (which we refer as the squared martingale): 
\[
\Big( \frac{1}{b^n}\sum_{|u|=n}\prod_{j=1}^n \frac{W(u|_j)^2}{\E[W^2]}  \Big)_{n\ge 1}.
\]
The study of the small moments of the above martingale should be divided into three cases: squared sub-critical regime, squared critical regime and squared super-critical regime. Therefore,  the well-studied small moments of the additive martingales with  critical and super-critical parameters in BRW  will play a key role here. 

\subsection{Sharp upper estimate for $\chi_1$ and $\chi_2$}

Recall the definition \eqref{def-Rn} of the sequence  
\[
\mathcal{R} = (\mathcal{R}_n  (W, b, \alpha, p, q))_{n\ge 1}.
\]
\begin{lemma}\label{lem-chi-I}
Let $0<\alpha<1$.  Assume that $1< p\le 2 \le q <\infty$ and $q> \frac{1}{1 - \alpha}$.  Then there exists a constant $c(b, \alpha, p, q)>0$ such that 
\[
\max(\chi_1, \chi_2)   \le   c(b, \alpha, p, q)  \cdot \| \mathcal{R}\|_{\ell^1}.
\]
\end{lemma}

\begin{proof}

For simplifying the notation, for all $u\in \Ab^*$ with $|u|\ge 1$ and for all $n\ge 1$, we write 
\begin{align}\label{def-Xu}
X(u) : =  \prod_{j=0}^{|u|-1} W(u|_j)^2  \ge 0  \an Y_n: = \sum_{|u|=n} X(u)
\end{align}
and 
\[
F(s): = \sum_{|u|\ge 1}   X(u)   \cdot    s^{-2 (1 -\alpha)} \cdot  \indi (s  \ge  b^{|u|}).
\]
Recall the definition \eqref{eqn-def-mnab} for $\MAB$.  For all $n\ge 1$, we have 
\begin{align}\label{Yn-Mn}
Y_n = \sum_{|u|=n} \prod_{j=0}^{n-1} W(u|_j)^2 =    b  \sum_{|v|=n-1}  \prod_{j=0}^{n-1} W(v|_j)^2 =   b( b \E[W^2])^{n-1}    \MAA.
\end{align}
Then
\begin{align*}
\chi_1 & =   \E \Big[   \Big\{   \sum_{s \ge 1}       \Big(   \sum_{|u|\ge 1}   X(u)   \cdot    s^{-2 (1 -\alpha)} \cdot  \indi (s \ge b^{|u|}) \Big)^{q/2}  \Big\}^{p/q}\Big]
= \E \Big[   \Big\{     \sum_{l  = 1}^\infty    \sum_{b^{l-1} \le s < b^{l}}  F(s)^{q/2} \Big\}^{p/q}\Big].
\end{align*}
For all integers in the  $b$-adic interval $b^{l-1}\le s < b^l$, we have
\begin{align*}
F(s)\le  T_l: =  & \sum_{|u|\ge 1} X(u)
\cdot   b^{-2(l-1)(1 - \alpha)}  \cdot \indi(l \ge |u|)  =  \sum_{n=1}^\infty   \underbrace{ Y_n
\cdot   b^{-2(l-1)(1 - \alpha)}  \cdot \indi(l  \ge n)}_{\text{denoted $T_l(n)$}}. 
\end{align*}
Therefore, 
\begin{align*}
\chi_1 \le  \E \Big[   \Big\{     \sum_{l  = 1}^\infty    \sum_{b^{l-1} \le s < b^{l}}  T_l^{q/2} \Big\}^{p/q}\Big] \le  \E \Big[   \Big\{    \sum_{l  = 1}^\infty    b^l  \cdot T_l^{q/2} \Big\}^{p/q}\Big].
\end{align*}
Since $p/q<1$ and $p/2<1$, by applying twice the inequalities \eqref{2-ele-ineq}, we obtain 
\begin{align*}
\chi_1 & \le \E \Big[  \sum_{l\ge 1}    b^{lp/q}  \cdot T_l^{p/2} \Big]  =  \E \Big[  \sum_{l  \ge 1}    b^{lp/q}  \cdot  \Big(\sum_{n\ge 1} T_l(n)\Big)^{p/2} \Big] 
\\
& \le \E \Big[  \sum_{l  \ge 1}    b^{lp/q}  \cdot  \sum_{n\ge 1} T_l(n)^{p/2}  \Big] =  \sum_{l  \ge 1}   b^{lp/q}  \cdot  \sum_{n\ge 1}   \E[T_l(n)^{p/2}]. 
\end{align*}
Now 
\begin{align*}
 \E[T_l(n)^{p/2}]  &   = \E \Big[   \Big(Y_n
\cdot   b^{-2(l-1)(1 - \alpha)}  \cdot \indi(l  \ge n)\Big)^{p/2}\Big] = b^{-p (l-1)(1-\alpha)}  \cdot  \E\big [  Y_n^{p/2}\big] \cdot \indi( l  \ge n). 
\end{align*}
It follows that 
\begin{align*}
\chi_1& \le  \sum_{l  \ge 1}   b^{lp/q}  \cdot  b^{-p (l-1)(1-\alpha)}  \sum_{n\ge 1} \E \big[ Y_n^{p/2}\big] \cdot \indi( l \ge n)
\\
& \le  b^{p(1-\alpha)} \sum_{n\ge 1}    \E \big[  Y_n^{p/2}\big] \sum_{l\ge 1}  b^{-pl (1-\alpha- \frac{1}{q})}   \cdot \indi( l  \ge n)
\\
 &\le \frac{b^{p(1-\alpha)}}{1- b^{-p (1-\alpha- \frac{1}{q})} }\sum_{n\ge 1}    \E \big[   Y_n^{p/2}]  b^{-p n (1 - \alpha -  \frac{1}{q})}. 
\end{align*}

Using exactly the same argument as above, we have 
\[
\chi_2\le  \sum_{l  \ge 1}   b^{lp/q}  \cdot  \sum_{n\ge 1}   \E[\widehat{T}_l(n)^{p/2}],
\]
where
\[
\widehat{T}_l(n): =   Y_n \cdot b^{2(l\alpha-n)}  \cdot   \indi(l-1 \le  n).
\]
Hence
\begin{align*}
\chi_2\le  &  \sum_{n\ge 1}   \E\big[Y_n^{p/2}  \big]\cdot    \sum_{l  \ge 1}    b^{lp/q}  b^{p(l\alpha-n)}        \indi(l-1\le n)
 \\
\le & \sum_{n\ge 1}   \E\big[Y_n^{p/2}  \big]     \cdot   \frac{b^{p(\alpha + \frac{1}{q})}}{1 - b^{-p(\alpha + \frac{1}{q})}}  \cdot b^{-pn (1 - \alpha - \frac{1}{q})}.
\end{align*}

Finally, by substituting \eqref{Yn-Mn} and \eqref{def-Rn} into the above final inequalities for $\chi_1$ and $\chi_2$, we  obtain the desired upper estimate. 
\end{proof}

\subsection{Lower estimate for $\chi_2$}

\begin{lemma}\label{lem-low-chi2}
Let $0<\alpha<1$.  Assume that $1< p\le 2 \le q <\infty$ and $q> \frac{1}{1 - \alpha}$.  Then there exists a constant $c'(b, \alpha, p, q)>0$ such that 
\[
\chi_2 \ge \frac{c'(b, \alpha, p, q)}{(\E[W^2])^{p/2}} \cdot \| \mathcal{R}\|_{\ell^{q/p}}. 
\]
\end{lemma}

\begin{proof}
Recall the notation $X(u)$ defined in  \eqref{def-Xu} and write 
\[
H(s,u): = X(u)      s^{2 \alpha}   b^{-2|u|}  \indi(1\le s\le b^{|u|}).
\]
Note that, on the $b$-adic interval $b^{l-1}\le s<b^l$, we have
\begin{align*}
H(s,u) &\ge   X(u) b^{2(l-1)\alpha } b^{-2|u|} \indi(1 \le s \le b^{|u|})   \ge  \underbrace{X(u) b^{2(l-1)\alpha } b^{-2|u|} \indi(l\le |u|)}_{\text{denoted $\widehat{H}(l,u)$}}.
\end{align*}
Hence  
\begin{align*}
\chi_2 &= \E \Big[   \Big\{     \sum_{s \ge 1}       \Big(   \sum_{|u|\ge 1}    H(s,u)\Big)^{q/2}  \Big\}^{p/q}\Big]
\\
&  =  \E \Big[   \Big\{    \sum_{l= 1}^\infty \sum_{b^{l-1}\le s < b^l}       \Big(  \sum_{n=1}^\infty   \sum_{|u| =  n}    H(s,u) \Big)^{q/2}  \Big\}^{p/q}\Big]
\\
& \ge \E \Big[   \Big\{    \sum_{l= 1}^\infty \sum_{b^{l-1}\le s < b^l}       \Big(  \sum_{n=1}^\infty   \sum_{|u|  = n}    \widehat{H}(l,u) \Big)^{q/2}  \Big\}^{p/q}\Big]
\\
&  =   \E \Big[   \Big\{     \sum_{l= 1}^\infty  (b-1) b^{l-1} \Big(  \sum_{n=1}^\infty   \sum_{|u| = n}    \widehat{H}(l,u) \Big)^{q/2}  \Big\}^{p/q}\Big].
\end{align*}
Recall the notation $Y_n$ defined in \eqref{def-Xu}.  Since $q/2\ge 1$, by applying the second inequality in \eqref{2-ele-ineq}, we have
\begin{align*}
\Big(  \sum_{n=1}^\infty   \sum_{|u| = n }    \widehat{H}(l,u) \Big)^{q/2}  & \ge \sum_{n=1}^\infty \Big(\sum_{|u| = n }    \widehat{H}(l,u) \Big)^{q/2}
\\
& = \sum_{n=1}^\infty  \Big(  \sum_{|u|=n} X(u) b^{2(l-1)\alpha } b^{-2n} \indi(l\le n)\Big)^{q/2}
\\
& = \sum_{n=1}^\infty Y_n^{q/2} b^{q(l-1)\alpha } b^{-qn} \indi(l\le n). 
\end{align*}
It follows that 
\begin{align*}
\chi_2 & \gtrsim_{b, \alpha, p, q}  \E \Big[   \Big\{     \sum_{l= 1}^\infty   b^{l}   \sum_{n=1}^\infty  Y_n^{q/2} b^{ql\alpha } b^{-qn} \indi(l\le n)  \Big\}^{p/q}\Big].
\end{align*}
Then since 
\begin{align*}
\sum_{l= 1}^\infty   b^{l}   \sum_{n=1}^\infty  Y_n^{q/2} b^{ql\alpha } b^{-qn} \indi(l\le n)   &  =       \sum_{n=1}^\infty  Y_n^{q/2}   b^{-qn} \sum_{l= 1}^n  b^{l}  b^{ql\alpha }
\\
&   \ge   \sum_{n=1}^\infty  Y_n^{q/2}   b^{-qn}   b^{n}  b^{qn\alpha }   = \sum_{n=1}^\infty  Y_n^{q/2}   b^{-qn(1-\alpha - \frac{1}{q})},
\end{align*}
we have 
\begin{align*}
\chi_2\gtrsim_{b,\alpha, p, q} \E\Big[ \Big\{ \sum_{n=1}^\infty  Y_n^{q/2}   b^{-nq (1 - \alpha - \frac{1}{q})}  \Big\}^{p/q}\Big].
\end{align*}
Therefore, by applying the Minkowski inequality \eqref{Min-pq}, we obtain 
\[
\chi_2 \gtrsim_{b, \alpha, p, q}   \Big\{ \sum_{n=1}^\infty  \Big(   \E[Y_n^{p/2}]  b^{-np (1-\alpha- \frac{1}{q})}    \Big)^{q/p} \Big\}^{p/q}.
\]
By substituting \eqref{Yn-Mn} and \eqref{def-Rn} into the above inequality, we  obtain the desired lower estimate of $\chi_2$. 
\end{proof}

\subsection{Final step of the proof of Theorem \ref{thm-CCM}}
The estimate \eqref{2-side-zero} and \eqref{chi-2-12}, together with Lemma \ref{lem-chi-I} and Lemma \ref{lem-low-chi2} yield the desired upper and lower estimate of  Theorem~\ref{thm-CCM}.

\section{Elementary analysis of the initial random weight $W$}\label{sec-W}

\subsection{The quantity $D_F$ satisfies $D_F\in (0,1)$}\label{sec-pos}

Recall the definition  \eqref{def-DF} of the quantity 
\[
D_F= D_F(W,b).
\]
\begin{lemma}\label{lem-DF}
If $\E[W\log W]<\log b$ and $\E[W^t]<\infty$ for all $t>0$, then 
 $
 D_F\in (0,1).
 $ 
\end{lemma}
\begin{proof}
Recall the definition \eqref{def-phi-W} for the function $\varphi_W$: 
\[
\varphi_W(t) = \log \E[W^t] - (t-1) \log b, \quad t>0
\]
Clearly,  $\varphi_W$ is a convex function on $(0, \infty)$. Indeed, 
\[
\varphi_W'(t)=\frac{\E[W^{t}\log W]}{\E[W^{t}]} - \log b
\]
and, by Cauchy-Schwarz's inequality, 
\[
 \varphi_W''(t)=\frac{\E[W^{t}(\log W)^2]\E[W^{t}]-(\E[W^{t}\log W])^2}{(\E[W^{t}])^2}\ge 0.  
\]

(i).  If  $\E\big[ \frac{W^2}{\E[W^2]} \log \frac{W^2}{\E[W^2]}\big] \le \log b$, then $D_F = 1- \log \E[W^2]/\log b$. Hence we need to show  $\E[W^2]> 1$ and  $\varphi_W(2) = \log \E[W^2] - \log b <0$.  The first inequality follows from the assumption that $\E[W]=1$ and $W\ge 0$ is non-constant. We  now prove the second inequality.  Indeed, in this case,  since $\varphi_W'$ is increasing,  we have 
\begin{align}\label{ineq-f-d}
\varphi_W(2)=  \varphi_W(2)-\varphi_W(1) = \int_1^2 \varphi_W'(s)ds \le    \varphi_W'(2).
\end{align}
Now the condition $\E[ \frac{W^2}{\E[W^2]} \log \frac{W^2}{\E[W^2]}] \le \log b$ can be re-written as  
$
2 \varphi_W'(2)  \le \varphi_W(2).
$
Consequently,  $\varphi_W(2)\le \varphi_W'(2) \le \varphi_W(2)/2$ and hence $\varphi_W(2)\le 0$.  Moreover, if the equality $\varphi_W(2)=0$ holds, then  $\varphi'_W(2) = \varphi_W(2) =0$ and the inequality \eqref{ineq-f-d} becomes an equality. This would  imply that $\varphi_W'$ is a constant function on the interval $[1, 2]$ and hence
\[
0 =  \varphi_W'(2)= \varphi_W'(1)= \E[W\log W]-\log b.
\]
But this contradicts to the assumption $\E[W\log W]<\log b$. Hence we must have the desired inequality $\varphi_W(2)<0$.

(ii).  If  $\E\big[ \frac{W^2}{\E[W^2]} \log \frac{W^2}{\E[W^2]}\big] >  \log b$, then since $\E[W^{2t}]  = b^{2t-1} e^{\varphi_W(2t)}$, we need to show 
\begin{align}\label{F-pos-proof}
D_F= 1  -   \inf_{1/2\le t\le 1} \frac{\log  \E[b^{1-t}W^{2t}]}{ t \log b}  =  - \inf_{1/2\le t\le 1}   \frac{ \varphi_W(2t)}{t\log b} \in (0,1). 
\end{align}
The inequality $D_F>0$ follows immediately since $\varphi_W(1) =0$ and by the Mandelbrot-Kahane's non-degeneracy condition, 
$
\varphi_W'(1)= \E[W\log W]-\log b <0. 
$ Finally,  the inequality $D_F<1$ holds if and only if 
\[
\inf_{1/2\le t\le 1 } \log \E[b^{1-t} W^{2t}]> 0. 
\]
Assume by contradiction that $\inf_{1/2\le t\le 1 } \log \E[b^{1-t} W^{2t}]=0$. Then by continuity of the function $t\mapsto \log \E[b^{1-t} W^{2t}]$, there exists $t_0\in [1/2, 1]$ such that 
$
 \E[b^{1-t_0} W^{2t_0}] = 1$. Clearly, $t_0\ne 1$ since $\E[W^2]>1$. But if $t_0\in [1/2, 1)$, then  $2t_0\ge 1$, we have $\E[W^{2t_0}] >1$ and $b^{1-t_0}>1$, contradicts to the assumption  $\E[b^{1-t_0} W^{2t_0}]=1$. This completes the proof of $D_F<1$. 
\end{proof}

\subsection{Biggins-Kyprianou's boundary case}\label{sec-BK-transformation}
Following the terminology in  \cite{Bigins-EJP}, 
 we say that  the random weight $W$  is in the  {\it  boundary case} or more precisely  in the {\it Biggins-Kyprianou's boundary case},  if it has a {\it Biggins-Kyprianou transform} as follows:
\begin{align}\label{bdary-reg}
W= \frac{e^{-\beta \xi}}{\E[e^{-\beta\xi}]},
\end{align}
where $\beta>0$ is a positive number and $\xi$ is a random variable  taking values in $(-\infty, +\infty]$ such that (we use the convention $(+\infty)\cdot e^{-(+\infty)} = 0$)
\begin{align}\label{xi-cond}
\E[\xi e^{-\xi}] = 0, \quad \E[e^{-\xi}] = \frac{1}{b}
\end{align}
and 
\begin{align}\label{xi-cond-bis}
\PP(\xi\in (-\infty, 0))> 0. 
\end{align}
In the boundary case,  the Biggins-Kyprianou transform of $W$ is unique, that is,  the pair $(\beta, \xi)$ is uniquely determined by $W$. See  \S \ref{sec-xi} in the Appendix for the details.

\begin{remark*}
Assume that  $\xi$ satisfies the condition \eqref{xi-cond}. Then $\PP(\xi\in (-\infty, 0)) = 0$ if and only if  
\[
\xi = \left\{
\begin{array}{cl}
 0 & \text{with probability $b^{-1}$}
\\
 +\infty &  \text{with probability $1-b^{-1}$}
\end{array}
\right..
\]
\end{remark*}

Clearly, not all random weights have Biggins-Kyprianou transforms. On the other hand, if $W$ is assumed to have moments of all positive  orders, then the lemma below provides a useful criterion for $W$ to be in the boundary case.  This result is  due to Jaffuel (see \cite[the ArXiv version, Proposition A.2]{Jaffuel}\footnote{There is a typo in the statement of \cite[the ArXiv version, Proposition A.2]{Jaffuel}: $x_{min}>-\infty$ should be replaced by $x_{min} = -\infty$.}). Here  we state it in a version that is convenient for our purpose. 

\begin{lemma}\label{lem-bdd}
 Assume that $\E[W^t]<\infty$ for all $t>0$. Then $W$ is not in the boundary case if and only if $W$ is bounded and 
$
b \cdot \PP(W= \|W\|_\infty) \ge 1. 
$
\end{lemma}
Assume that $W$ has finite moments of all orders and is in the boundary case, then we can construct the pair $(\beta, \xi)$ as follows.  Recall first the definition \eqref{def-phi-W} of the function 
\[
 \varphi_W(t)  = \log \E[W^t]-(t-1)\log b, \quad t > 0.
\]
Then   the equation 
\[
  \varphi_W'(t)  = \frac{ \varphi_W(t) }{t}
\]
 has a unique  positive solution $t^*\in(0,\infty)$. By taking 
 \begin{align}\label{def-beta}
 \beta = \frac{1}{t^*} \an \xi = -  t^* \log W + \log \E[ b W^{t^*}] ,
 \end{align}
  we obtain the Biggins-Kyprianou transform of $W$.

\begin{example}[Binomial distributions]\label{ex-bad-boy}
For any $t \in (0,1)$,  consider the random variable 
\[
W_t = \left\{
\begin{array}{cl}
t^{-1}& \text{with probability $t$}
\\
0 & \text{with probability $1- t$}
\end{array}
\right. .
\]
One can check directly that all these random variables are not in the  boundary case. 
However, there is a unique $t = b^{-1}$, such that the random variable $W_{1/b}$ can be written as  $W_{1/b}= e^{-\beta\xi}/\E[e^{-\beta\xi}]$ (with non-unique $\beta>0$), where
$
\xi = 0 \cdot \indi(W_{1/b}>0) + (+\infty)  \cdot \indi(W_{1/b}=0).
$
The above $\xi$ satisfies the condition  \eqref{xi-cond}, but violates the condition  \eqref{xi-cond-bis}. Hence by definition, $W_{1/b}$ is not in the boundary case. 
\end{example}

Define a function $\psi = \psi_\xi:  [0, +\infty) \rightarrow \R \cup \{+\infty\}$ by  
\begin{align}\label{def-psi}
\psi(t) = \psi_\xi(t):=\log \E[be^{-t\xi}]. 
\end{align}
Some basic properties of the Biggins-Kyprianou transform and  the function $\psi$ will be useful for us. For future reference, we give these properties in the following Lemma~\ref{lem-psi} and include its routine proof in  \S \ref{sec-psi-fn} of the  Appendix  for the reader's convenience.

\begin{lemma}\label{lem-psi}
If $W$ is in the boundary case and  $\E[W^{1+\delta}]<\infty$  for some $\delta > 0$, then 
\begin{itemize}
\item $\psi$ is $C^\infty$ on $(0, 1+\delta)$; 
\item $\psi(1) = \psi'(1) = 0$; 
\item $\psi$ is strictly decreasing on $[0, 1]$ and strictly increasing on $[1, 1+\delta]$.  In particular, $\psi$ is non-negative on $[0, 1 + \delta]$;
\item $\psi$ is strictly convex;
\item  the following equivalences hold: 
\begin{align}\label{beta-less-1}
\begin{cases}
  \E [ W \log  W]<\log b \Longleftrightarrow 0< \beta<1; 
  \vspace{2mm}
  \\
  \E [ W \log  W] = \log b \Longleftrightarrow \beta = 1;
  \vspace{2mm}
\\
  \E [ W \log  W]>\log b \Longleftrightarrow \beta>1.
\end{cases}
\end{align}
\end{itemize}
\end{lemma}

\subsection{The squared random weight}

We are going to use the obvious fact: if  $W$ is in the boundary case and  $W\in L^2(\PP)$, then  so is 
\[
W^{(2)}: =   \frac{W^2}{\E[W^2]}=\frac{e^{-2\beta\xi}}{\E[e^{-2\beta\xi}]}.
\]
Indeed,  the Biggins-Kyprianou's transform of $W^{(2)}$ will be  crucial in our study of  the small moments asymptotics of the martingale defined in \eqref{eqn-def-mnab} in the squared super-critical regime,  see \S\ref{sec-super}. 

Recall that, as in the introduction,  according to Mandelbrot-Kahane's non-degeneracy condition, we say that  $W$ is 
\begin{itemize}
\item in the squared sub-critical regime if 
\begin{align}\label{sq-sub}
\E[  W^{(2)} \log W^{(2)}] <\log b;
\end{align}
\item in the  squared critical regime if 
\begin{align}\label{sq-critical}
\E[  W^{(2)} \log W^{(2)}]   = \log b;
\end{align}
\item  in the squared super-critical regime if 
\begin{align}\label{sq-super}
\E[  W^{(2)} \log W^{(2)}] > \log b.
\end{align}
\end{itemize}

\begin{lemma}\label{lem-non-sub}
Assume that $\E[W^t]<\infty$ for all $t>0$. If $W$ is either  in the squared critical or the squared super-critical regime, then $W$ is in the boundary case.
\end{lemma}

\begin{proof}
 Suppose that $W$ is not in the boundary case,  we  need to show that $\E[  W^{(2)} \log W^{(2)}] <\log b$. Indeed, we show a stronger inequality  (see  Remark~\ref{rem-Lp-non-deg}): 
\begin{align}\label{ineq-non-sub}
\E[(W^{(2)})^2] = \E\Big[ \Big(\frac{W^2}{\E[W^2]}\Big)^2\Big]<b.
\end{align}
By Lemma \ref{lem-bdd},  under the assumption that  $\E[W^t]<\infty$ for all $t>0$,  the random weight $W$ is not in the boundary case if and only if it is bounded and 
$
b \cdot \PP(W= \|W\|_\infty) \ge 1.
$
If $W$ is a constant random variable, then  $W \equiv 1$ and the assertion is trivial. If $W$ is not a constant, then $\PP(W<\|W\|_\infty)>0$ and 
$
\E[W^4] < \|W\|_\infty^2  \cdot \E[W^2]$. Hence 
\[
\E\Big[ \Big(\frac{W^2}{\E[W^2]}\Big)^2\Big] < \frac{\|W\|_\infty^2}{\E[W^2]} \an 
\E[W^2]  \ge  \|W\|_\infty^2\cdot \PP(W=\|W\|_\infty) \ge  \frac{\|W\|_\infty^2}{b}.
\]
The desired inequality \eqref{ineq-non-sub} follows immediately. 
\end{proof}

\section{Lower bound of the critical exponent $\alpha_c$}\label{sec-squared}

In this section, we always assume that the initial random weight $W$ satisfies
\begin{align}\label{all-m}
\text{$\E[W\log W]<\log b$ and $\E[W^t]<\infty$ for all $t>0$.}
\end{align}

For the MCCM $\mu_\infty$, recall the definitions \eqref{eq-npq} and \eqref{def-a-c} of $\mathcal{N}^{(\alpha, p, q)}(\mu_\infty)$ and $\alpha_c$ respectively:
\begin{align*}
	& \mathcal{N}^{(\alpha, p, q)}(\mu_\infty) =   \Big(\E\Big[  \Big\{ \sum_{s\ge 1}    \big| s^{\alpha} \cdot \widehat{\mu}_\infty(s)\big|^q  \Big\}^{\frac{p}{q}}\Big] \Big)^{\frac{1}{p}}  \in [0, +\infty],
\\
& \alpha_c =  \sup \Big\{\alpha \in \R:  \mathcal{N}^{(\alpha, p, q)}(\mu_\infty) <\infty \text{\, for some $1<p<2<q<\infty$} \Big\}.
\end{align*}

The main result of this section is the following

\begin{proposition}\label{prop-sq-3case} 
We have
$
\alpha_c \ge \frac{ D_F}{2}.
$
\end{proposition}

\subsection{Squared sub-critical regime}\label{sec-sq-sub}

If $W$ is in the squared sub-critical regime:
\[
\E[W^{(2)}\log W^{(2)}]<\log b,
\]
then, by Kahane's work, the martingale $(\MAB)_{n\ge 1}$ is uniformly integrable and has a non-degenerate limit $\mathscr{M}_\infty(W^{(2)})$.  Hence, by using the super-martingale property (since $p/2\le 1$) and the H\"older's inequality, we have 
\begin{align*}
0<  \E[\mathscr{M}_\infty(W^{(2)})^{\frac{p}{2}}]  \le  \E[\MAB^{\frac{p}{2}}]\le 1 \quad \text{for all $n\ge 1$}. 
\end{align*}
This combined with Theorem \ref{thm-CCM} immediately yields 

\begin{lemma}\label{prop-sub}
Suppose that $W$ is in the  squared sub-critical regime. 
Let $0\le \alpha<1$, $1< p\le 2$ and $q> \max\{2, \frac{1}{1 - \alpha}\}$.   Then the following equivalence holds:
\[
\NQ(\mu_\infty)<\infty \Longleftrightarrow 
 \E[W^2] < b^{1- 2\alpha- \frac{2}{q}}. 
\]
In particular,  $\NQ(\mu_\infty)<\infty$ implies $\alpha<\frac{1}{2}$ and  $q> \frac{2}{1- 2\alpha}$.
\end{lemma}

By Lemma \ref{prop-sub}, we have
\begin{corollary}\label{cor-sqsub}
Suppose that $W$ is in the  squared sub-critical regime. Then
\[
\alpha_{c}\geq \frac{D_F}{2}=\frac{1}{2}\Big(1 - \frac{\log \E[W^2]}{\log b}\Big).
\]
\end{corollary}

\subsection{Squared critical regime}
Suppose that $W$ is in  the squared critical regime:
\[
\E[W^{(2)}\log W^{(2)}]=\log b.
\]
Then, for $p/2\le 1$, by the super-martingale property, 
\begin{align*}
\sup_{n\ge 1}\E[\MAB^{\frac{p}{2}}]  \le  \sup_{n\ge 1}(\E[\MAB])^{\frac{p}{2}} = 1. 
\end{align*}
This combined with Theorem \ref{thm-CCM} immediately yields 

\begin{lemma}\label{prop-sq-crit}
Suppose that $W$ is in  the squared critical regime. Let $0\le \alpha<\frac{1}{2}$. Then in the range $1< p\le 2 < q <\infty$ and $q> \frac{2}{1 - 2 \alpha}$, the condition $ \E[W^2] < b^{1- 2\alpha- \frac{2}{q}}$ implies that $\NQ(\mu_\infty)<\infty$.  
\end{lemma}

\begin{remark*}
Combining our proof of Lemma \ref{prop-sq-crit} with  a result of Hu and Shi \cite[Theorem 1.5]{Shi} in BRW,  one  can show that under the assumption of Lemma \ref{prop-sq-crit} and an additional assumption that $\E[W^{-\varepsilon}]<\infty$ for some $\varepsilon$, then 
\[
\NQ(\mu_\infty)<\infty \Longleftrightarrow  \E[W^2]\le b^{1- 2\alpha - \frac{2}{q}}.
\]
\end{remark*}

\begin{corollary}\label{cor-cri}
	Suppose that $W$ is in the  squared critical regime. Then
	\[
	\alpha_{c}\geq \frac{D_F}{2}=\frac{1}{2}\Big(1 - \frac{\log \E[W^2]}{\log b}\Big).
	\]
\end{corollary}

\subsection{Squared super-critical regime}\label{sec-super}

Suppose that $W$ is in the squared super-critical regime:
\[
\E[W^{(2)}\log W^{(2)}]>\log b.
\]
By Lemma \ref{lem-non-sub}, $W$ is in the boundary case with the Biggins-Kyprianou transform: 
\[
W=\frac{e^{-\beta \xi}}{\E[e^{-\beta\xi}]},\quad \E[\xi e^{-\xi}]=0,\quad \E[b e^{-\xi}]=1,
\]
with $\xi$ a random variable taking values in $\R\cup\{+\infty\}$ and  $\beta>0$. 
In particular, 
\begin{align}\label{W2-2beta}
W^{(2)}=\frac{W^2}{\E[W^2]} = \frac{e^{-2\beta \xi}}{\E[e^{-2\beta\xi}]}. 
\end{align}
Since $W$ is assumed to satisfy Mandelbrot-Kahane's non-degeneracy condition, we must have $\beta<1$. Moreover, by Lemma \ref{lem-psi}, $W$ is in the squared super-critical regime if and only if $2\beta > 1$. Therefore, we have 
\[
\frac{1}{2}< \beta <1. 
\]

Recall the definition \eqref{def-psi}: 
\[
\psi(t) = \psi_\xi(t):=\log \E[be^{-t\xi}]. 
\]

\begin{lemma}\label{lem-psi-beta}
For $1/2<\beta<1$, we have 
\[
 D_F=1 -   \inf_{1/2\le t\le 1} \frac{\log  \E[b^{1-t}W^{2t}]}{ t \log b} = \frac{ 2\psi(\beta)}{\log b}. 
\]
\end{lemma}

\begin{proof}
 We can  write 
$
W = b e^{-\beta \xi  - \psi(\beta)},
$
then 
\[
\log  \E[b^{1-t}W^{2t}]=t\log b+\psi(2t\beta)-2t\psi(\beta)
\]
and
\[
D_F=\frac{2\psi(\beta)}{\log b}-\inf_{1/2\leq t\leq 1}\frac{\psi(2t\beta)}{t\log b} .
\]
Hence, we need to show
\[
  0 =  \inf_{1/2\le t\le 1}  \frac{\psi(2t\beta)}{ t}. 
\]
This equality follows from Lemma~\ref{lem-psi}. Indeed,
by Lemma~\ref{lem-psi},  $\psi$ is non-negative and   $\psi(2 t_0 \beta) = \psi(1)=0$ with
$
1/2 < t_0 = \frac{1}{2\beta} <1$. 
\end{proof}

Now write 
\[
\mathscr{M}_n(2\beta): = \MAB.
\]
Using the equality \eqref{W2-2beta} and the definition \eqref{def-psi} of the function $\psi$, we have
\begin{align}\label{Mn-beta-Z}
\mathscr{M}_n(2\beta) = \frac{1}{b^n} \sum_{|u|=n}  \prod_{j=1}^n \Big(    \frac{W(u|_j)^2}{\E[W^2]}  \Big) =  \frac{1}{b^n} \sum_{|u|=n}  \prod_{j=1}^n \Big(    \frac{e^{-2\beta \xi(u|_j)} }{\E[e^{-2\beta\xi}]}  \Big) =    e^{-n \psi(2\beta)} \cdot Z_{n, 2\beta},
\end{align}
where $Z_{n, 2\beta}$ is the partition function given by
\[
Z_{n,2\beta} = \sum_{|u|=n} e^{- 2\beta  \sum_{j=1}^n \xi(u|_j)} = \sum_{|u|=n} e^{- 2 \beta V(u)}
\]
with $V(u)$ defined by
 \begin{align}\label{eqn-BRW-VU}
 V(u) := \sum_{j=1}^{|u|}  \xi(u|_j). 
 \end{align}

The almost sure, distributional and small moment asymptotics  of the above partition function have been investigated by Hu and Shi \cite{Shi},   A\"{\i}d\'{e}kon \cite{Elie-AOP} and Madaule \cite{Mad} in the framework of BRW.     To suit our setting and for the convenience of the readers, a particular version of their works is formulated as follows:

Assume that  
\begin{align}\label{xi-ass}
\E[\xi^2 e^{-\xi}]<\infty. 
\end{align}
By Madaule \cite[Proposition 2.1]{Mad}, for any $\gamma>1$,  there exists $c>0$ such that 
\[
\sup_{n\ge 1}\PP\big( n^{\frac{3\gamma}{2}} Z_{n, \gamma} \ge e^{\gamma x}\big) \le c (1+x) e^{-x} \quad \text{for all $x\ge1$}. 
\]
Then,   for any $0< r<\frac{1}{\gamma}$, 
\[
\sup_{n\ge 1}\E[(n^{\frac{3\gamma}{2}} Z_{n, \gamma})^r]<\infty. 
\]
On the other hand,  by \cite[Lemma~3.6]{Elie-AOP}, there is another constant $c'>0$ such that for any integer $n\geq 1$ and any real number $x\in [0,\frac{3}{2}\log n-1]$, 
\begin{align*}
\PP\big(n^{\frac{3}{2}\gamma}Z_{n,\gamma}\ge e^{\gamma x}\big)\ge  c'e^{-x}. 
\end{align*}
Hence, for any $0< r<\frac{1}{\gamma}$, we have
\begin{align*}
\E\big[(n^{\frac{3\gamma}{2}} Z_{n, \gamma})^r\big]&=\int_{\R}r e^{\gamma x(r-1)}\PP\big(n^{\frac{3}{2}\gamma}Z_{n,\gamma}\ge e^{\gamma x}\big)\gamma e^{\gamma x}dx\\
&\ge c'r\gamma \int_{0}^{\frac{3}{2}\log n-1}  e^{-(1-\gamma r)t} dt\\
&=\frac{c'r\gamma}{1- r\gamma}\big(1-e^{-(1-r\gamma)(\frac{3}{2}\log n-1)}\big).
\end{align*}
Thus,  for any $0< r<\frac{1}{\gamma}$, there exist two constants $c_1 = c_1(r, \gamma), c_2 = c_2(r, \gamma)>0$ such that for any integer $n \ge 1$
\begin{align}\label{s-m-asymp}
	c_1 n^{-\frac{3r\gamma}{2}}\le  \E[(Z_{n, \gamma})^r]  \le c_2 n^{-\frac{3r\gamma}{2}}.
\end{align}

\begin{lemma}\label{prop-super}
Let $0\le \alpha<\frac{1}{2}$ and $q>\frac{2}{1-2\alpha}$. Assume that $\frac{1}{2}<\beta<1$.   Then the following assertions are equivalent:
\begin{itemize}
\item[(A1).]$\NQ(\mu_\infty)<\infty$ holds for all $p$ with $1< p< \frac{1}{\beta}$; 
\item[(A2).] $\NQ(\mu_\infty)<\infty$ holds for all $p$ with $\frac{2}{3\beta}< p< \frac{1}{\beta}$; 
\item[(A3).]  $\NQ(\mu_\infty)<\infty$ holds for some $p$ with $\frac{2}{3\beta}<p< \frac{1}{\beta}$; 
\item[(A4).] $\alpha +q^{-1}  \le    \frac{\psi(\beta)}{\log b}$.
\end{itemize}
\end{lemma}

\begin{proof}
The following implications are obvious  (since the $L^p (\PP, \ell^q)$ norm is non-decreasing in $p$):
\[
\text{(A1)$\Longleftrightarrow$ (A2) $\Longrightarrow$ (A3)}.
\]
Therefore, it remains to show that (A3) $\Longrightarrow$ (A4) and (A4) $\Longrightarrow$ (A2). 

By the definition of the function $\psi$ in \eqref{def-psi} and the relation \eqref{bdary-reg}, we have 
\begin{align*}
e^{-\psi(2\beta) }   \cdot   \frac{\E[W^2]}{b^{1- 2\alpha- \frac{2}{q}}}  = \underbrace{\exp\Big(- 2 \Big[\frac{\psi(\beta)}{\log b} - (\alpha + q^{-1}) \Big] \log b \Big)}_{\text{denoted $T = T(\alpha, \beta, \xi,  q, b)$}}. 
\end{align*}
Thus (A4) holds if and only if $T\le 1$.

Note that the condition \eqref{all-m} clearly implies the condition \eqref{xi-ass}. Therefore, by \eqref{Mn-beta-Z} and \eqref{s-m-asymp}, if $1/2<\beta<1$ and  $1<p<\frac{1}{\beta}$, then there exist $c_1(p, \beta), c_2(p, \beta)>0$, such that  for all $n\ge 1$, 
\begin{align}\label{big-O}
  c_1(p, \beta) e^{-np\psi(2\beta)/2} \cdot n^{-3p\beta/2} \le  \E[\mathscr{M}_n(2\beta)^{\frac{p}{2}}]  \le c_2(p, \beta) e^{-np\psi(2\beta)/2} \cdot n^{-3p\beta/2}. 
\end{align}

(A4) $\Longrightarrow$ (A2):  Assume that (A4)  holds ($T\le 1$) and fix any $p$ with $\frac{2}{3\beta}<p< \frac{1}{\beta}$. Recall the definition of the sequence  $\mathcal{R}= (\mathcal{R}_n(W, b, \alpha, p, q))_{n\ge 1}$ defined in \eqref{def-Rn}:
\[
\mathcal{R}_n  (W, b, \alpha, p, q)  =  \E \big[  \mathscr{M}_{n-1}(2\beta)^{\frac{p}{2}}\big] \cdot \Big(  \frac{\E[W^2]}{b^{1- 2\alpha- \frac{2}{q}}}\Big)^{\frac{np}{2}}. 
\]
 By \eqref{big-O},  there exists $C= C(p, \beta)>0$, such that 
\[
\|\mathcal{R}\|_{\ell^1} \le  C  \sum_{n\ge 1}^\infty T^{np/2} \cdot n^{-3p\beta/2} \le    C \sum_{n\ge 1}^\infty n^{-3p\beta/2}<\infty.
\]
Therefore, by  Theorem~\ref{thm-CCM},  $\NQ(\mu_\infty)<\infty$.  That is, (A2) holds. 

(A3) $\Longrightarrow$ (A4): 
Now assume that (A3) holds.  By Theorem~\ref{thm-CCM},  $\| \mathcal{R}\|_{\ell^{q/p}}<\infty$. Also, by \eqref{big-O},  there exists $c= c(p, \beta)>0$, such that 
\begin{align*}
\| \mathcal{R}\|_{\ell^{q/p}}^{q/p} \ge  c \sum_{n=1}^\infty T^{nq/2} \cdot  n^{-3q\beta/2}.
\end{align*}
It follows that $T\le 1$.  That is, (A4) holds. 
\end{proof}

Note that,  by Lemma \ref{lem-psi}, the function $\psi$ is strictly convex and decreasing  on the interval $[0,1]$,  for $1/2 < \beta <1$, we have 
\[
0< \psi(\beta) < \psi(1/2) < \frac{\psi(0)+\psi(1)}{2}=\frac{\log b}{2} .
\]

\begin{corollary}\label{cor-super-alpha}
Let $1/2 < \beta <1$ and $1<p<\frac{1}{\beta}$. Then  for any $q>  \frac{\log b}{\psi(\beta)}>2$, 
\begin{align*}
\mathcal{N}^{(\alpha, p, q)}(\mu_\infty)<\infty \quad \text{for all 
 $\alpha\le \frac{\psi(\beta)}{\log b} - q^{-1}$}.
 \end{align*}
\end{corollary}

\begin{proof}
Assume that $q>  \frac{\log b}{\psi(\beta)}>2$.  Set $\alpha_0 = \frac{\psi(\beta)}{\log b} - q^{-1}\ge 0$. 
Then,  $0< \alpha_0<\frac{1}{2}$ and 
\[
 \frac{2}{1-2\alpha_0} = \frac{1}{\frac{1}{2}-  \frac{\psi(\beta)}{\log b}  + q^{-1}} < q. 
\] 
Hence, we may apply the equivalence  (A1) $\Longleftrightarrow$ (A4) in Lemma \ref{prop-super}  and obtain 
\[
\text{ $\mathcal{N}^{(\alpha_0, p, q)}(\mu_\infty)<\infty$ for all $1< p< \frac{1}{\beta}$.} 
\]
 This completes the proof since $\mathcal{N}^{(\alpha, p, q)}(\mu_\infty)$ is increasing in $\alpha$. 
\end{proof}
\begin{corollary}\label{cor-sqsup}
If  $W$ is in the squared super-critical regime, then 
\[
\alpha_c\geq  \frac{D_F}{2}=\frac{\psi(\beta)}{\log b}.
\]
\end{corollary}
\begin{proof}
 By Lemma~\ref{lem-psi-beta}, Corollary \ref{cor-super-alpha}  and the definition \eqref{def-a-c} of $\alpha_c$, one gets
	\[
	\alpha_c \ge  \sup_{1< p< 2< q<\infty}  \Big(\frac{\psi(\beta)}{\log b} - q^{-1}\Big)=  \frac{\psi(\beta)}{\log b}=\frac{D_F}{2}.
	\]

\end{proof}

\subsection{Proof of Proposition \ref{prop-sq-3case}}

The desired inequality 
$
\alpha_c\ge \frac{D_F}{2}
$
 follows from 
Corollary~\ref{cor-sqsub}, Corollary~\ref{cor-cri} and Corollary~\ref{cor-sqsup}.

\section{CLT and Stable limit}\label{sec-clt}
This section is devoted to the proofs of Proposition \ref{prop-to-infty} and Proposition \ref{prop-clt}. 

Proposition \ref{prop-clt} is inspired by the work of     Iksanov-Kolesko-Meiners \cite{Iksanov} and  will follow from  Propositions~\ref{thm-sq-sub-critical}, \ref{thm-sq-critical} and \ref{thm-non-cri-super-critical}, where the weak convergences are established under the original probability $\PP$,  and an elementary  Lemma \ref{lemma-cov-law-con-two}, which connects the weak convergences under $\PP$ and under the conditional $\PP^* = \PP(\cdot |\mu_\infty\ne 0)$.

 \subsection{The weak convergences under the original probability $\PP$}

Recall the definition \eqref{def-DF} of $D_F$.  By Lemma~\ref{lem-DF}, 
\[
D_F \in (0,1).
\]
In particular,   if  $W$ is in the squared sub-critical or squared critical regime, then  $1- \frac{\log \E[W^2]}{\log b}>0$ and hence $\E[W^2]<b$. Therefore,  we can define
 \begin{align}\label{eqn-c-c-c-c}
 \varrho=\varrho(W,b):=\frac{\E[\mathring{W}^2]}{4\pi^2\E[W^2]}\sum_{k=1}^\infty \Big(\frac{\E[W^2]}{b}\Big)^{k}\Big|\frac{e^{i2\pi b^{-k}}-1}{2\pi  b^{-k}}\Big|^2\in (0,\infty).
 \end{align}
In the case $b = 2$, we shall need to use the notation 
\begin{align}\label{def-pi}
\varpi: = 
- \frac{2 \E[\mathring{W}^2]}{\pi^2}. 
\end{align}
We are going to see in Lemma \ref{lemma-jihi} and Lemma \ref{coro-97} below that, $\varrho= \E[|\widehat{\mu}_\infty(1)|^2]$ for any $b \ge 2$, while  $\varpi = \E[(\widehat{\mu}_\infty(1))^2]$ for $b = 2$ and $\E[(\widehat{\mu}_\infty(1))^2] =0$ for all $b \ge 3$.  And, in particular, we have $\varrho \pm \varpi>0$.

If  the random weight $W$ is in the squared sub-critical regime (the condition \eqref{sq-sub} is satisfied), then Kahane's result  \cite[Theorem 2]{Kahane-Peyriere-advance}  implies that the martingale $\MAB$ defined in \eqref{eqn-def-mnab}
converges to a non-trivial limit $\mathscr{M}_\infty(W^{(2)})$ almost surely and also  in $L^1$.

\begin{proposition}[Squared sub-critical regime]\label{thm-sq-sub-critical}
If  $W$  is in the squared sub-critical regime, 
then for $b \ge 3$, 
$$
b^{\frac{nD_F}{2}}  \cdot \widehat{\mu}_{\infty}(b^n)\xrightarrow[n\to\infty]{d} \sqrt{\varrho\cdot \mathscr{M}_\infty(W^{(2)})}\mathcal{N}_{\C}(0,1),
$$
where $\mathcal{N}_{\C}(0,1)$ is independent of $\mathscr{M}_\infty(W^{(2)})$. While for $b = 2$, 
\[
2^{\frac{nD_F}{2}}  \cdot  \Big[   \sqrt{ \frac{\varrho}{\varrho+\varpi}} \mathrm{Re} \big(\widehat{\mu}_{\infty}(2^n)\big) + i  \sqrt{ \frac{\varrho}{\varrho-\varpi}} \mathrm{Im} \big(\widehat{\mu}_{\infty}(2^n)\big) \Big]   \xrightarrow[n\to\infty]{d} \sqrt{\varrho\cdot \mathscr{M}_\infty(W^{(2)})}\mathcal{N}_{\C}(0,1).
\]
\end{proposition}


If  $W$  is in the squared critical regime (the  condition \eqref{sq-critical} is satisfied), then  Kahane's result \cite[Theorem 2]{Kahane-Peyriere-advance} implies that the martingale $\MAB $ defined in \eqref{eqn-def-mnab}
converges to $0$ almost surely. Moreover,  by Lemma~\ref{lem-non-sub}, $W$ can be written as ($\beta =\frac12 $): 
\[
W= \frac{e^{-\frac{1}{2}\xi}}{\E[e^{-\frac{1}{2}\xi}]} 
\]
with $\xi$ satisfying the conditions \eqref{xi-cond} and \eqref{xi-cond-bis}.  Recall the definition 
 \eqref{eqn-BRW-VU} of $V(u)$:
 \[
  V(u) = \sum_{j=1}^{|u|}  \xi(u|_j). 
 \] Then $(V(u))_{u\in\mathcal{A}^*}$ is a BRW in the boundary case. Therefore, by  \cite[Theorem 1.1]{Aideon-Shi} (our assumptions  ensure that $W$ satisfies all the conditions in \cite{Aideon-Shi}),  
\begin{align}\label{root-n-M}
n^{1/2} \cdot \MAB\xrightarrow[n\rightarrow\infty]{\text{in probability}} \mathscr{D}_\infty(W^{(2)})\sqrt{\frac{2}{\pi \sigma^2}},
\end{align}
where $\mathscr{D}_\infty(W^{(2)})$ is the almost sure limit of the derivative martingale $\sum_{|u|=n}V(u)e^{-V(u)}$ and 
\begin{equation}\label{def-BRW-var}
\sigma^2 :=b\E[\xi^2 e^{-\xi }] =\E\Big[ \Big(\log \frac{W^2 }{ \E[ b W^{2}]} \Big)^2 \frac{W^2}{\E[W^2]}\Big ] \in (0,\infty).
\end{equation}

\begin{proposition}[Squared critical regime]\label{thm-sq-critical}
If $W$  is in the squared critical regime, 
 then for $b \ge 3$, 
$$
n^{\frac{1}{4}} b^{\frac{n D_F}{2}} \widehat{\mu}_{\infty}(b^n)\xrightarrow[n\to\infty]{d} \sqrt{\varrho\cdot \sqrt{\frac{2}{\pi \sigma^2}}\mathscr{D}_\infty(W^{(2)})}\mathcal{N}_\C(0,1),
$$
where  $\mathcal{N}_\C(0,1)$ is independent of $\mathscr{D}_\infty(W^{(2)})$. While for $b=2$,
\[ 
n^{\frac{1}{4}} 2^{\frac{nD_F}{2}}  \cdot  \Big[   \sqrt{ \frac{\varrho}{\varrho+\varpi}} \mathrm{Re} \big(\widehat{\mu}_{\infty}(2^n)\big) + i  \sqrt{ \frac{\varrho}{\varrho-\varpi}} \mathrm{Im} \big(\widehat{\mu}_{\infty}(2^n)\big) \Big]   \xrightarrow[n\to\infty]{d}  \sqrt{\varrho\cdot \sqrt{\frac{2}{\pi \sigma^2}}\mathscr{D}_\infty(W^{(2)})}\mathcal{N}_\C(0,1).
\]
 \end{proposition}


 Now suppose that $W$ is in the squared super-critical regime (the condition \eqref{sq-super} is satisfied).   Recall that, in the squared super-critical regime, we always assume that   $\log W$ is non-lattice, that is, $\log W$ is not supported on any arithmetic progression.

By Lemma~\ref{lem-non-sub},  $W$ 
 is in the boundary case: 
\begin{align}\label{bdd-case}
W= \frac{e^{-\beta \xi}}{\E[e^{-\beta \xi}]}   \quad  \text{with} \quad    \frac{1}{2} < \beta< 1
\end{align}
and $\xi$ satisfies the conditions \eqref{xi-cond} and \eqref{xi-cond-bis}. 

Recall the definition 
 \eqref{eqn-BRW-VU} of $V(u)$.  Then $
 (V(u))_{u\in\Ab^*}$   is a BRW in the boundary case  (see  \cite[Chapter 1]{shi-zhan-book}).
 Under the non-lattice assumption, by \cite[Theorem 1.1]{Mad}),  the family of the following point processes 
\begin{align*}
\Big\{\sum_{|u|=n}\delta_{V(u)-\frac{3}{2}\log n}: n\in \N\Big\}
\end{align*}
converges in law in the vague sense. More precisely, there exists some point process $\mathcal{Z}_\infty$ such that
\begin{align}\label{eqn-z-infinity-s}
\sum_{|u|=n}\delta_{V(u)-\frac{3}{2}\log n}\xrightarrow[n \rightarrow\infty]{d} \mathcal{Z}_\infty=\sum_{i\in\N}\delta_{x_i}. 
\end{align}
Here the convergence means that,  for any compactly supported continuous function $f\in C_c(\R)$, we have 
\[
  \sum_{|u|=n}  f\Big( V(u)-\frac{3}{2}\log n\Big) \xrightarrow[ n \rightarrow\infty]{d}   \int_\R f  d\mathcal{Z}_\infty =  \sum_{i\in\N} f(x_i). 
\]
By \cite[Theorem 1.1]{Mad} and \cite[Proposition 2.4]{Iksanov}, the point process $\mathcal{Z}_\infty$ in \eqref{eqn-z-infinity-s}  regarded as a counting measure on $\mathbb{R}$ satisfies that almost surely there is only finite number of points on $(-\infty,0)$. Hence we may rearrange the random points $(x_j)_{j\in\N}$ in $\mathcal{Z}_\infty$ such that they are lying in a non-decreasing order on $\mathbb{R}$: 
\[
x_1 \le x_2\le \cdots \le x_n\le \cdots. 
\]
We shall need the following result (see \cite[formula (5.4)]{Iksanov}):  for any $\gamma>1$, we have 
\begin{align}\label{Z-finite}
\sum_{i=1}^\infty e^{-\gamma x_i}<\infty   \quad   a.s. 
\end{align}

Recall that $\E[\widehat{\mu}_\infty(1)]=0$. Using Burkholder's inequalities, we have 
\begin{lemma}\label{lem-wel-def}
Assume that  $W$  is in the squared super-critical regime such that $\log W$ is non-lattice.    Let  $\{\widehat{\mu}^{(i)}_\infty(1)\}_{i\in \N}$ be i.i.d. copies of $\widehat{\mu}_\infty(1)$, which are all independent of the point process $\mathcal{Z}_\infty$. 
 Then there exists  $p\in (1/\beta,2]$ such that, conditioned on $\mathcal{Z}_\infty$, we have 
\[
\sup_{F}  \E \Big[ \Big|   \sum_{i \in F}  e^{-\beta x_i}   \widehat{\mu}^{(i)}_\infty(1) \Big|^p  \Big| \mathcal{Z}_\infty\Big]   \lesssim_p \sum_{i \in \N}  e^{- p \beta x_i}   \cdot  \E[ |\widehat{\mu}_\infty(1)|^p]  <\infty \quad a.s. ,
\]
where the supremum ranges over all finite subsets  $F\subset \N$.  In particular,  given $\mathcal{Z}_\infty$,  the limit 
\[
\sum_{i\in \N} e^{-\beta x_i}\widehat{\mu}^{(i)}_\infty(1) = \lim_{N\to\infty} \sum_{i=1}^N e^{-\beta x_i}\widehat{\mu}^{(i)}_\infty(1)
\]
is well-defined as a  $L^p$-bounded martingale limit. 
\end{lemma}

\begin{remark*}
Note that in general, we do not have 
\[
\sum_{i\in \N}  e^{-\beta x_i} |\widehat{\mu}^{(i)}_\infty(1)|<\infty  \quad a.s. 
\]
\end{remark*}

  \begin{proposition}[Squared super-critical regime]\label{thm-non-cri-super-critical}
If $W$  is in the squared super-critical regime such that $\log W$ is non-lattice, then 
 \begin{align*}
   b^{\frac{ n \cdot D_F}{2} }n^{\frac{3}{2}\beta} \widehat{\mu}_{\infty}\left(b^{n}\right)\xrightarrow[n\rightarrow\infty]{d} \sum_{i\in \N} e^{-\beta x_i}\widehat{\mu}^{(i)}_\infty(1),
 \end{align*}
 where $(x_i)_{i\in\N}$ is the  point process  $\mathcal{Z}_\infty$ in \eqref{eqn-z-infinity-s},  $\{\widehat{\mu}^{(i)}_\infty(1)\}_{i\in \N}$ are i.i.d. copies of $\widehat{\mu}_\infty(1)$  independent of $\mathcal{Z}_\infty$.
Moreover, the limit is non trivial and  the random variable 
\[
\sum_{i\in \N} e^{-\beta x_i} \mathrm{Re} (\widehat{\mu}^{(i)}_\infty(1)) 
\]
can be written as $ \mathscr{D}_\infty(W^{(\frac1\beta)})^\beta \cdot \widetilde{X} $ where $\mathscr{D}_\infty(W^{(\frac1\beta)})>0$ $\PP^*$-a.s., and $\widetilde{X}$ is independent of $\mathscr{D}_\infty(W^{(\frac1\beta)})$ and has a $\frac{1}{\beta}$-stable law. In particular, 
\begin{align}\label{zero-prob}
\PP^*  \Big( \Big| \sum_{i\in \N} e^{-\beta x_i}  \widehat{\mu}^{(i)}_\infty(1)\Big|=0\Big) =0.  
\end{align}
 \end{proposition}

\subsection{Elementary properties of $\widehat{\mu}_\infty(b^n)$}

\begin{lemma}\label{lemma-eqn-function-equality}
Fix $n\in\N$, we have  
\begin{align}\label{eqn-function-equality} 
 \widehat{\mu}_{\infty}\left(b^n\right)\stackrel{d}{=}\frac{1}{b^n}\sum_{|u|=n}  \prod_{j=0}^{n} W(u|_j)  \cdot \widehat{\mu}_{\infty}^{(u)}(1),
\end{align}
where $\{\widehat{\mu}_{\infty}^{(u)}(1):|u|=n\}$ are i.i.d. copies of $\widehat{\mu}_{\infty}(1)$ which are all independent of the sigma-algebra $\mathscr{F}_{n}=\sigma\{W(u): 1\leq |u|\leq n\}$.
\end{lemma}
\begin{proof}
For  $s\in \N$,  by definition one has $\widehat{\mu}_0(s)=0$ and 
\begin{align}\label{eqn-m-differ-s}
\widehat{\mu}_{m}(s)-\widehat{\mu}_{m-1}(s)=\sum_{|u|=m}  \Big(\prod_{j=0}^{m-1} W(u|_j)\Big) \mathring{W}(u) \cdot \int_{I_u} e^{-i2\pi s x}dx , \quad \forall m \in \N,
\end{align}
Notice that if $|u|\leq n$, one has $ \int_{I_u} e^{i2\pi b^n x}dx=0$ and hence 
\[
\widehat{\mu}_{n}\left(b^n\right) =\sum_{m=1}^{n}\left(\widehat{\mu}_{m}\left(b^n\right)-\widehat{\mu}_{m-1}\left(b^n\right)\right) =0 .
\]
It follows that 
\begin{align*}
\widehat{\mu}_{\infty}\left(b^n\right) & =\widehat{\mu}_{\infty}\left(b^n\right)-\widehat{\mu}_{n}\left(b^n\right) 
 =\sum_{m=n+1}^{\infty}\left(\widehat{\mu}_{m}\left(b^n\right)-\widehat{\mu}_{m-1}\left(b^n\right)\right) \\
&=\sum_{m=n+1}^{\infty} \,\,\sum_{|u|=m} \Big(\prod_{j=0}^{m-1} W(u|_j)\Big) \mathring{W}(u) \cdot \int_{I_u} e^{-i 2 \pi b^n x} d x. 
\end{align*}
Note that each $u$ with $|u|=m\ge n+1$ can be uniquely written as  $u= vw$ such that  $|v|=n$ and $|w|=m-n$.  For any $n+1\le j \le m-1$, we have 
\[
(vw)|_j = v (w|_{j-n}) = v  \cdot w|_{j-n}. 
\]
Therefore, by using the elementary identity 
\[
\sum_{|u|=m}   f(u) =     \sum_{|v|=n} \sum_{|w|= m-n} f(vw)
\]
and writing  
\[
\prod_{j=0}^{m-1} W(u|_j)    = \prod_{j=0}^{m-1} W((vw)|_j)  = \Big(\prod_{j=0}^{n} W(v|_j)\Big) \cdot \Big(\prod_{l=1}^{m-n-1} W(v \cdot w|_l)\Big),
\]
we obtain 
\begin{align*}
\widehat{\mu}_{\infty}\left(b^n\right)  &= \sum_{m=n+1}^{\infty} \,\,\sum_{|v|=n} \sum_{|w|= m-n}    \Big(\prod_{j=0}^{n} W(v|_j)\Big) \cdot \Big(\prod_{l=1}^{m-n-1} W(v\cdot w|_l)\Big)      \mathring{W}(vw) \cdot \int_{I_{vw}} e^{-i 2 \pi b^n x} d x\\
&=\sum_{|v|=n}\Big(\prod_{j=0}^{n} W(v|_j)\Big) \sum_{k=1}^\infty \sum_{|w|=k} \Big( \prod_{l=1}^{k-1} W( v\cdot w|_l)  \Big) \mathring{W}(vw) \int_{I_{vw}} e^{-i 2 \pi b^n x} d x,
\end{align*}
where we used the convention (for $k=1$)
\[
 \prod_{l=1}^{k-1} W( v\cdot w|_l)  = 1. 
\]
By a direct computation,  for $|v|=n$,  we have 
\[
\int_{I_{vw}} e^{i 2 \pi b^n x} d x = b^{-n} \int_{I_w} e^{-i 2\pi x} dx. 
\]
Hence
\begin{align*}
\widehat{\mu}_{\infty}\left(b^n\right) 
&= \frac{1}{b^n}\sum_{|v|=n}\Big(\prod_{j=0}^{n} W(v|_j)\Big) \cdot  \underbrace{\sum_{k=1}^\infty \sum_{|w|=k} \Big( \prod_{l=1}^{k-1} W( v\cdot w|_l)  \Big) \mathring{W}(vw) \int_{I_{w}} e^{-i 2 \pi x} d x}_{\text{denoted as $\widehat{\mu}_\infty^{(v)}(1)$}}.
\end{align*}
Clearly,  $\{\widehat{\mu}_{\infty}^{(v)}(1): |v|=n\}$  are i.i.d. copies of $\widehat{\mu}_{\infty}(1)$ and are independent of the sigma-algebra $\mathscr{F}_{n}=\sigma\{W(u)| 1\leq |u|\leq n\}$. This completes the proof.
\end{proof}

\begin{lemma}\label{lemma-jihi}
Assume that $W$ is in the squared sub-critical or squared critical regime.  Then for any $s\in\N$, one has
\begin{align}\label{eqn-mu-s-square}
\E[|\widehat{\mu}_\infty(s)|^2]=\frac{\E[\mathring{W}^2]}{b}\sum_{m=1}^\infty \Big(\frac{\E[W^2]}{b}\Big)^{m-1}\Big|\frac{e^{i2\pi sb^{-m}}-1}{2\pi s b^{-m}}\Big|^2
\end{align}
and 
\begin{align}\label{eqn-mu-s-square-mei-mo}
\mathbb{E}\left[\widehat{\mu}_{\infty}(s)^2\right] = \frac{\E[\mathring{W}^2]}{b}\sum_{m=1}^{+\infty} \Big(\frac{ \mathbb{E}[W^2]}{b}\Big)^{m-1}  \cdot  \frac{1}{b^m}\sum_{|u|=m}\Big(\frac{e^{i 2 \pi s \cdot b^{-m}}-1}{i 2 \pi s  b^{-m}}\Big)^2 e^{i 4\pi s\, t_u},
\end{align}
where  $t_u:=\sum_{j=1}^{|u|} \frac{u_j}{b^j}$ is  the left endpoint of the interval $I_u$. 
\end{lemma}

\begin{proof}
First note that by Lemma \ref{lem-DF}, we have 
$
\E[W^2]<b. 
$
Hence both terms on the right hand sides of \eqref{eqn-mu-s-square} and \eqref{eqn-mu-s-square-mei-mo} are well-defined. 
Fix any integer $s\ge 1$.  Using  orthogonality of the martingale differences, we have 
\begin{align*}
\E[|\widehat{\mu}_\infty(s)|^2] =\sum_{m=1}^\infty\E\big[|\widehat{\mu}_{m}(s)-\widehat{\mu}_{m-1}(s)|^2\big]. 
\end{align*} 
Now by  \eqref{eqn-m-differ-s}  and the orthogonality $\E[\mathring{W} (u) \mathring{W}(v)]=0$ for all $u\ne v$, we get   
\begin{align*}
\E[|\widehat{\mu}_\infty(s)|^2]
&=\sum_{m=1}^\infty (\E[W^2])^{m-1}\E[\mathring{W}^2]b^{m}\Big|\frac{e^{i2\pi sb^{-m}}-1}{2\pi s}\Big|^2\\
&=\frac{\E[\mathring{W}^2]}{b}\sum_{m=1}^\infty \Big(\frac{\E[W^2]}{b}\Big)^{m-1}\Big|\frac{e^{i2\pi sb^{-m}}-1}{2\pi s b^{-m}}\Big|^2.
\end{align*}

The equality  \eqref{eqn-mu-s-square-mei-mo} can be obtained by the same method.
\end{proof}

Recall the definition \eqref{eqn-c-c-c-c} of the constant $\varrho = \varrho(W,b)$  for all $b \ge 2$ and the definition \eqref{def-pi} of the constant $\varpi$ defined only for $b=2$. . 

\begin{lemma}\label{coro-97}
Assume that $W$ is in the squared sub-critical or squared critical regime. Then 
\begin{align}\label{eqn-mu-s-square-s=1}
\E[|\widehat{\mu}_\infty(1)|^2]=\varrho. 
\end{align}
Moreover, for  $b \ge 3$, 
\begin{align}\label{eqn-mu-s-square-mei-mo-1}
\mathbb{E}\left[(\widehat{\mu}_{\infty}(1))^2\right] =0
\end{align}
and if $b = 2$, then 
\begin{align}\label{b2-square}
\mathbb{E}\left[(\widehat{\mu}_{\infty}(1))^2\right] =  \varpi = - \frac{2 \E[\mathring{W}^2]}{\pi^2}. 
\end{align}
\end{lemma}

\begin{proof}
By \eqref{eqn-mu-s-square} and the definition \eqref{eqn-c-c-c-c} of $\varrho$, we get 
the equality \eqref{eqn-mu-s-square-s=1} . 
For $b \ne 2$, the equality \eqref{eqn-mu-s-square-mei-mo-1} follows from  \eqref{eqn-mu-s-square-mei-mo} and the following identity: 
\[
\sum_{|u|=m}  e^{i 4\pi \, t_u}=     \sum_{u_1, \cdots, u_m = 0}^{b-1} \exp\Big( i 4 \pi   \sum_{j=1}^{m} \frac{u_j}{b^j}\Big)  =  \sum_{k = 0}^{b^{m}-1} \exp\Big( \frac{i 4 \pi k}{b^m}\Big) = 0 \quad \text{for all $m\ge 1$}.
\]
While for $b =2$, we have 
\[
\sum_{|u|=m}  e^{i 4\pi \, t_u} = 
\left\{
\begin{array}{cl}
0 & \text{if $m\ge 2$}
 \\
2 & \text{if $m =1$}
\end{array}
\right..
\]
Then the equality \eqref{b2-square} follows from \eqref{eqn-mu-s-square-mei-mo}. 
\end{proof}

In what follows, for a complex random variable $Z=X+iY$,  we denote by 
\[
\mathrm{Cov}(Z): =      \begin{pmatrix}
\Var(X) & \Cov(X, Y)\\
\Cov(X, Y) &\Var(Y) \\
\end{pmatrix}, 
\]
that is, $\mathrm{Cov}(Z)$ denotes the covariance matrix of the real random vector $(X, Y)$.  In particular, the standard complex normal random variable $\mathcal{N}_\C(0,1)$ defined as \eqref{def-cg} satisfies 
\[
\mathrm{Cov}(\mathcal{N}_\C(0,1))  =  \begin{pmatrix}
1/2 & 0 \\
0 & 1/2\\
\end{pmatrix}. 
\]

\begin{lemma}\label{lemma-mean-zero-s}
Assume that $b \ge 3$. If  $W$ is in the squared sub-critical regime or squared critical regime, then 
\begin{align}\label{eqn-A-equal-0}
\E\Big[ b^{n D_F}|\widehat{\mu}_{\infty}(b^n)|^2 \big| \mathscr{F}_{n}\Big] =\varrho\MAB, 
\end{align}
where $\MAB$ is  defined as in \eqref{eqn-def-mnab}.  Moreover,  
\begin{align}\label{eqn-A-mo-equal-0}
\E\Big[b^{n D_F}(\widehat{\mu}_{\infty}(b^n))^2 \big| \mathscr{F}_{n}\Big] =0.
\end{align}
 Hence, conditioned on $\mathscr{F}_{n}$,   the complex random variable $b^{\frac{n D_F}{2}}\cdot \widehat{\mu}_{\infty}(b^n)$ has a covariance matrix given by 
\begin{align}\label{eqn-cov-conditional}
\mathrm{Cov}\Big[b^{\frac{n D_F}{2}}\cdot \widehat{\mu}_{\infty}(b^n)\big| \mathscr{F}_{n}\Big]=\varrho\MAB
\begin{pmatrix}
1/2&0\\
0&1/2\\
\end{pmatrix}.
\end{align}
In particular, 
\begin{align}\label{eqn-b-power-n-rho}
\E\big[b^{ n D_F}|\widehat{\mu}_{\infty}(b^n)|^2 \big] =\varrho \an \E\big[ b^{n D_F}(\widehat{\mu}_{\infty}(b^n))^2 \big] =0
\end{align}
and
\[
\mathrm{Cov}\big[b^{\frac{n D_F}{2}}\cdot \widehat{\mu}_{\infty}(b^n)\big]= \varrho \begin{pmatrix}
1/2&0\\
0&1/2\\
\end{pmatrix}.
\]
\end{lemma}
\begin{proof}
By definition, one has $b^{D_F}=\frac{b}{\E[W^2]}$.
Hence by \eqref{eqn-function-equality}, we have
\begin{align}\label{eqn-alpha-equiv-s-s}
b^{\frac{n D_F}{2}} \widehat{\mu}_{\infty}(b^n)=\frac{1}{(b\E[W^2])^{\frac{n}{2}}}\sum_{|u|=n}  \prod_{j=1}^{n} W(u|_j)  \cdot \widehat{\mu}_{\infty}^{(u)}(1) .
\end{align}
By the independence among $(\widehat{\mu}_\infty^{(u)}(1): |u|=n)$ and $\E[\widehat{\mu}_\infty(1)]=0$, we have 
\begin{align*}
\E\Big[ b^{n D_F}|\widehat{\mu}_{\infty}(b^n)|^2 \big| \mathscr{F}_{n}\Big]
= \E\big[|\widehat{\mu}_{\infty}(1)|^2\big]\frac{1}{b^n}\sum_{|u|=n} \prod_{j=1}^{n} \Big(\frac{W(u|_j)^2}{\E[W^2]}\Big). 
\end{align*}
The desired equality \eqref{eqn-A-equal-0} then follows from  \eqref{eqn-mu-s-square-s=1} and the definition \eqref{eqn-def-mnab} of $\MAB$. 

Similarly,  since 
\[
\E [\widehat{\mu}_{\infty}^{(u)}(1) \widehat{\mu}_{\infty}^{(v)}(1)| \mathscr{F}_n] = 0 \quad \text{for any  $u\ne v$ with $|u|=|v|=n$,}
\]
by \eqref{eqn-alpha-equiv-s-s}, 
 we have 
\begin{align*}
\E\Big[ b^{n D_F}(\widehat{\mu}_{\infty}(b^n))^2 \big| \mathscr{F}_{n}\Big]
 = \frac{ \E[(\widehat{\mu}_{\infty} (1))^2]}{(b\E[W^2])^{n}}\sum_{|u|=n}  \prod_{j=1}^{n} W(u|_j)^2.  
\end{align*}
The desired equality \eqref{eqn-A-mo-equal-0} then follows from \eqref{eqn-mu-s-square-mei-mo-1}. 
\end{proof}

\begin{lemma}\label{lem-b2}
Assume $b  = 2$. If  $W$ is in the squared sub-critical regime or squared critical regime, then  conditioned on $\mathscr{F}_{n}$,   the complex random variable $2^{\frac{n D_F}{2}}\cdot \widehat{\mu}_{\infty}(2^n)$ has a covariance matrix given by 
\[
\mathrm{Cov}\Big[2^{\frac{n D_F}{2}}\cdot \widehat{\mu}_{\infty}(2^n)\big| \mathscr{F}_{n}\Big]=\frac{\MAB}{2}
\begin{pmatrix}
\varrho+\varpi&0\\
0&\varrho-\varpi\\
\end{pmatrix}.
\]
In other words, 
\[
\mathrm{Cov}\Bigg[2^{\frac{nD_F}{2}}  \cdot  \Big[   \sqrt{ \frac{\varrho}{\varrho+\varpi}} \mathrm{Re} \big(\widehat{\mu}_{\infty}(2^n)\big) + i  \sqrt{ \frac{\varrho}{\varrho-\varpi}} \mathrm{Im} \big(\widehat{\mu}_{\infty}(2^n)\big) \Big] \Big| \mathscr{F}_n  \Bigg] = \varrho\MAB
\begin{pmatrix}
1/2&0\\
0&1/2\\
\end{pmatrix}.
\]
\end{lemma}
\begin{proof}
The proof is similar to that of Lemma~\ref{lemma-mean-zero-s}, where we shall use \eqref{b2-square} instead of using \eqref{eqn-mu-s-square-mei-mo-1}.  
\end{proof}

\subsection{Squared sub-critical regime: Proof of Proposition \ref{thm-sq-sub-critical}}

Now suppose that we are in the squared sub-critical regime.
We shall use the conditional  Linderberg-Feller central limit theorem   (see Proposition~\ref{Prop-c-clt} in the Appendix) to prove Proposition~\ref{thm-sq-sub-critical}.

For $|u|=n$,  we write
\begin{align}\label{eqn-yu-uuu}
\VU=b^{-n}\prod_{j=1}^{n} \frac{W(u|_j)^2}{\E[W^2]}.
\end{align}
Notice that $\{\VU: |u|=n\}$ are $\mathscr{F}_n$-measurable and
\[
\MAB=\sum_{|u|=n}\VU.
\]

\begin{proof}[Proof of Proposition~\ref{thm-sq-sub-critical}]
We only give the proof of Proposition~\ref{thm-sq-sub-critical} for $b \ge 3$.  The proof in the case $b =2$ is similar (where we shall use Lemma \ref{lem-b2} instead of using Lemma \ref{lemma-mean-zero-s}). 

By \eqref{eqn-alpha-equiv-s-s}, one has
\begin{align*}
 b^{\frac{n D_F}{2}} \widehat{\mu}_{\infty}(b^n)&=\frac{1}{(b\E[W^2])^{\frac{n}{2}}}\sum_{|u|=n}  \prod_{j=1}^{n} W(u|_j)  \cdot \widehat{\mu}_{\infty}^{(u)}(1)
=\sum_{|u|=n}\sqrt{\VU}\widehat{\mu}_{\infty}^{(u)}(1). 
\end{align*}
Hence, combining with \eqref{eqn-A-equal-0} and \eqref{eqn-A-mo-equal-0},  conditioned on $\mathscr{F}_n$,  we can decompose $b^{\frac{n D_F}{2}} \widehat{\mu}_{\infty}(b^n)$ as the sum of  centered independent random variables 
\[
 \big\{\sqrt{\VU}  \cdot \widehat{\mu}_{\infty}^{(u)}(1)\big\}_{|u|=n}
\]
with conditional covariance matrix given by \eqref{eqn-cov-conditional}. Apparently, 
\[
\mathrm{Cov}\Big[b^{\frac{n D_F}{2}}\cdot \widehat{\mu}_{\infty}(b^n)\big| \mathscr{F}_{n}\Big]\xrightarrow[n\to\infty]{a.s.}    \varrho \mathscr{M}_\infty(W^{(2)}) 
\begin{pmatrix}
1/2 &0\\
0&1/2 \\
\end{pmatrix}.
\]

We want to show that this sequence of conditional complex random variables converges in distribution to a  complex  Gaussian random variable  $\sqrt{\varrho\cdot \mathscr{M}_\infty(W^{(2)})}\mathcal{N}_{\C}(0,1)$.

We shall check the conditional  Lindeberg-Feller condition in Proposition~\ref{Prop-c-clt}. 
In other words, we need to show that,  for any $\varepsilon>0$, the following holds:
 \begin{align}\label{eqn-CLT-feller}
\lim_{n\rightarrow \infty} \sum_{|u|=n} \E[\VU |\widehat{\mu}_{\infty}^{(u)}(1)|^2\mathds{1}(\VU |\widehat{\mu}_{\infty}^{(u)}(1)|^2>\varepsilon)|\mathscr{F}_n]=0.
 \end{align}
Set 
 \begin{align}\label{eqn-CLT-sigma}
\sigma(x): =\E\big[|\widehat{\mu}_{\infty}(1)|^2 \mathds{1}(|\widehat{\mu}_{\infty}(1)|^2>x)\big], \quad \forall x \geqslant 0.
\end{align}
It is known from \eqref{eqn-b-power-n-rho}  that $\E[|\widehat{\mu}_{\infty}(1)|^2]=\varrho<+\infty$. Thus, by  Dominated Convergence Theorem,  
\begin{align}\label{sigma-to-zero}
\sigma(x)\downarrow 0 \quad \text{as $x\uparrow \infty $ }.
\end{align}
Hence  Lemma~\ref{lemma-Yu-zero-s} below  implies that, almost surely, 
\[
\lim_{n\rightarrow \infty} \sum_{|u|=n} \E\Big[\VU |\widehat{\mu}_{\infty}^{(u)}(1)|^2\mathds{1}(\VU |\widehat{\mu}_{\infty}^{(u)}(1)|^2>\varepsilon)\Big| \mathscr{F}_n\Big]=\lim_{n\rightarrow \infty} \sum_{|u|=n} \VU\sigma\Big(\frac{\varepsilon}{\VU}\Big)=0.
\]
By the conditional  Lindeberg-Feller's central limit theorem (Proposition~\ref{Prop-c-clt} in the Appendix),  we have
\[
b^{\frac{n D_F}{2}} \widehat{\mu}_{\infty}(b^n)\xrightarrow[n\to\infty]{d} \sqrt{\varrho\mathscr{M}_\infty(W^{(2)})}\mathcal{N}_\C(0,1). 
\]
This completes the proof.
\end{proof}

\begin{lemma}\label{lemma-Yu-zero-s}
 For any $\varepsilon>0$,  we have 
\[
\lim_{n\rightarrow \infty} \sum_{|u|=n} \VU\sigma\Big(\frac{\varepsilon}{Y_u}\Big)=0 \quad \text{a.s}.
\]
\end{lemma}
\begin{proof}
For $|u|=n$, write 
\begin{align}\label{def-Vu}
 - \log \VU=   \sum_{j=1}^{|u|}    \big( -2 \log W(u|_j) + \log \E[bW^2]\big). 
\end{align}
Then by \cite[Theorem~1.3]{shi-zhan-book}, almost surely, one has
\begin{align*}
\lim_{n\to\infty} \frac{1}{n} \inf_{|u|=n}  ( - \log \VU) & =   - \inf_{t>0} \frac{1}{t} \log \E \Big[ b \exp \Big(t (2 \log W - \log \E[bW^2])    \Big) \Big]  
\\
 & =  - \inf_{t>0}   \frac{1}{t} \log   \Big( \frac{\E[b W^{2t}]}{(\E[bW^2])^t}  \Big). 
\end{align*}
Hence, by recalling  the definition \eqref{def-phi-W} of the function 
$
\varphi_W(t)= \log \E[W^t] - (t-1) \log b
$,  we obtain 
\[
\lim_{n\to\infty}  \frac{1}{n} \sup_{|u|=n}   \log \VU  =  \inf_{t>0}   \frac{ \varphi_W(2t) -     t \varphi_W(2) }{t}  = 2   \inf_{s>0}  \Phi(s)   \quad a.s. ,
\]
where \[
\Phi(s)  = \frac{\varphi_W(s)}{s}- \frac{\varphi_W(2)}{2}.
\]
 Since $\Phi(2)= 0$ and  by the squared sub-critical assumption on the random weight $W$, 
\[
\Phi'(2)=   \frac{1}{4} \Big(   \E\Big[ \frac{W^2}{\E[W^2]} \log \frac{W^2}{\E[W^2]}\Big] -  \log b\Big) < 0,
\]
we have  
\[
\inf_{s>0}  \Phi(s)<0
\] 
and thus 
\[
\lim_{n\rightarrow \infty} \sup_{|u|=n} \VU=0 \quad a.s. 
\]
Therefore, combining with the almost sure convergence of the martingale  $(\MAB)_{n\ge 1}$,  for fixed $\varepsilon>0$,  we obtain that, almost surely,
\begin{align*}
\lim_{n\rightarrow \infty} \sum_{|u|=n} \VU\sigma\Big(\frac{\varepsilon}{\VU}\Big) 
\le \lim_{n\rightarrow \infty}\MAB\sigma\Big(\frac{\varepsilon}{\sup_{|u|=n}\VU}\Big) = 0.
\end{align*}
This completes the whole proof.
\end{proof}

\subsection{Squared critical regime: Proof of Proposition~\ref{thm-sq-critical}}
Now suppose that $W$ is in the squared critical regime.  We shall  use  again the conditional   Lindeberg-Feller central limit theorem.

\begin{proof}[Proof of Proposition~\ref{thm-sq-critical}]
We only give the proof of Proposition~\ref{thm-sq-critical} for $b \ge 3$. The proof in the case $b =2$ is similar (where we shall use Lemma \ref{lem-b2} instead of using Lemma \ref{lemma-mean-zero-s}). 

 Recall the definition $\VU$ in  \eqref{eqn-yu-uuu}. 
As in  the squared sub-critical regime,   conditioned on $\mathscr{F}_n$,  we  can decompose  $n^{\frac{1}{4}} \cdot b^{\frac{n D_F}{2}}  \cdot \widehat{\mu}_{\infty}(b^n)$ as the sum of $b^{n}$ centered independent random variables 
\[
\big\{n^{\frac{1}{4}}\sqrt{\VU}  \cdot \widehat{\mu}_{\infty}^{(u)}(1)\big\}_{|u|=n}, 
\]
and the conditional covariance matrix of each of them  satisfies  (in view of \eqref{root-n-M}) 
\begin{align*}
\mathrm{Cov}\big[n^{\frac{1}{4}}\sqrt{\VU}  \cdot \widehat{\mu}_{\infty}^{(u)}(1) \mid \mathscr{F}_{n}\big] 
\xrightarrow{\textrm{in probability}}   \varrho \sqrt{\frac{2}{ \pi\sigma^2}} \mathscr{D}_\infty(W^{(2)})   \begin{pmatrix}
1/2 &0\\
0&1/2\\
\end{pmatrix}.
\end{align*}
By Lemma~\ref{lemma-CLT-critical} below,  we get the   Lindeberg-Feller condition:
 \begin{align*}
\lim_{n\rightarrow \infty} \sum_{|u|=n} \E\Big[n^{1/2}\VU |\widehat{\mu}_{\infty}^{(u)}(1)|^2\mathds{1}(n^{1/2}\VU |\widehat{\mu}_{\infty}^{(u)}(1)|^2>\varepsilon)\Big|\mathscr{F}_n\Big]=0 \quad \text{in probability}.
 \end{align*}
We complete the proof by the conditional Lindeberg-Feller central limit theorem  (Proposition~\ref{Prop-c-clt}).
\end{proof}
\begin{lemma}\label{lemma-CLT-critical}
For any $\varepsilon> 0$, we have 
\[
\lim_{n\rightarrow\infty} \sum_{|u|=n} n^{1/2}\VU\sigma\Big(\frac{\varepsilon}{n^{1/2}\VU}\Big)=0 \quad \text{in probability},
\]
where $\sigma(x)$ is defined in \eqref{eqn-CLT-sigma}. 
\end{lemma}
\begin{proof}
As in  the proof of Lemma~\ref{lemma-Yu-zero-s}, for $|u|=n$, we have
\[
- \log \VU= \sum_{j=1}^n \big  (-2 \log W(u|_j) + \log \E[bW^2]\big). 
\]
By \cite[Theorem 5.12]{shi-zhan-book}, it is known that 
\[
\frac{1}{\log n} \inf_{|u|=n} ( - \log \VU) \xrightarrow[n\rightarrow\infty]{\text{in probability}}\frac{3}{2}.
\]
Thus 
\[
\lim_{n\rightarrow\infty} n^{1/2}\sup_{|u|=n} \VU= 0 \quad \text{in probability}.
\]
And therefore, by \eqref{root-n-M} and  \eqref{sigma-to-zero}, 
\[
\sum_{|u|=n} n^{1/2}\VU\sigma\Big(\frac{\varepsilon}{n^{1/2}\VU}\Big)\leq   n^{1/2}\MAB\sigma\Big(\frac{\varepsilon}{n^{1/2}\sup_{|u|=n}\VU}\Big)\xrightarrow[n\rightarrow\infty]{\text{in probability}} 0.
\]
This completes the proof. 
\end{proof}

\subsection{Squared super-critical regime: Proof of Proposition~\ref{thm-non-cri-super-critical}}

Suppose that $W$ is in the squared super-critical regime and $\log W$ is non-lattice. 

Recall the Biggins-Kyprianou transform \eqref{bdd-case} of $W$ and   the definition \eqref{def-psi} for $\psi = \psi_\xi$.  By Lemma~\ref{lem-psi-beta}, we have 
$
 D_F = \frac{2 \psi(\beta)}{\log b}
$
and hence 
\[
b^{\frac{D_F}{2}}=\exp(\psi(\beta)) =  \E[be^{-\beta \xi}]. 
\]
 By \eqref{eqn-function-equality} and the definition of the BRW $(V(u))_{u\in \Ab^*}$ in  \eqref{eqn-BRW-VU}, we have 
\[
 \widehat{\mu}_{\infty}\left(b^n\right)\stackrel{d}{=}  \sum_{|u|=n}  \prod_{j=1}^{n} \frac{e^{-\beta \xi(u|_j)}}{\E[b e^{-\beta\xi}]}  \cdot \widehat{\mu}_{\infty}^{(u)}(1)  = b^{- \frac{nD_F}{2}}  \sum_{|u| = n } e^{- \beta V(u)} \cdot \widehat{\mu}_{\infty}^{(u)}(1). 
\]
It follows that 
\begin{align}\label{eqn-psi-t-star}
  b^{\frac{n D_F}{2}}   n^{\frac{3}{2}\beta} \widehat{\mu}_{\infty}\left(b^n\right) = \sum_{|u|=n} e^{-\beta ( V(u)-\frac{3}{2}\log n )}\widehat{\mu}_{\infty}^{(u)}(1).
\end{align}

 \begin{lemma}\label{lemma-tightness}
There exists  $p\in (1/\beta,2]$ such that  for all $\gamma \in (0, \frac{1}{\beta p})$, we have
\[
0< \E[|\widehat{\mu}_\infty(1)|^p]<\infty \an 
\sup_{n\in\N^+}\E\Big[\big| b^{\frac{n D_F}{2}}   n^{\frac{3}{2}\beta} \widehat{\mu}_{\infty}\left(b^n\right)  \big|^{p\gamma}\Big]<\infty. 
\]
\end{lemma}

\begin{proof}
Recall the Kahane's $L^p$-criterion: If $1<p<\infty$, then  $\E[\mu_\infty([0,1])^p]<\infty$ if and only if $\E[W^p]<b^{p-1}$. Note that, in the boundary case,  
\[
\E[W^t]<b^{t-1} \Longleftrightarrow \frac{\psi(t \beta)}{t\beta} < \frac{\psi(\beta)}{\beta}. 
\]
Since $\psi(1)= 0, \psi(\beta)> 0$ and  $\psi$ is continuous,  one can find $p\in (\frac{1}{\beta}, 2]$ with
$
\frac{\psi(p \beta)}{p\beta} < \frac{\psi(\beta)}{\beta}
$
and hence 
$
\E[W^p]<b^{p-1}. 
$
Now we fix this exponent $p\in (1/\beta, 2]$.  First of all, we have 
\begin{align}\label{eqn-finite-p-moment}
\E[|\widehat{\mu}_\infty(1)|^p]  \le \E[\mu_\infty ([0,1])^p ]<\infty.
\end{align}
To prove $\E[|\widehat{\mu}_\infty(1)|^p]> 0$, we may apply Bourgain-Stein inequality as in the proof of Theorem~\ref{thm-precise}.   For avoiding the repetition of the same lines of arguments, here,  we derive $\E[|\widehat{\mu}_\infty(1)|^p]> 0$ from  Theorem~\ref{thm-precise}. More precisely,  let $\mathcal{V}: \mathfrak{B}^b \rightarrow \C$ be the scalar complex measure defined by 
\[
\mathcal{V}(I_u):=\int_{I_u} e^{-i 2\pi x}dx. 
\]
Recall the definition   \eqref{def-mart} of $Q_n\mathcal{V}([0,1])$: 
\[
Q_n\mathcal{V} ([0,1]) = \sum_{|u|=n}  \Big(\prod_{j=0}^n W(u|_j)\Big) \cdot \mathcal{V}(I_u)=\widehat{\mu}_n(1), \quad n\ge 1. 
\] 
Set
\begin{align*}
\chi(\mathcal{V}, W, p): =   \E\Big[ \Big| \sum_{|u|\ge 1}    \Big(  \prod_{j=0}^{|u|-1} W(u|_j)^2 \Big)   \cdot  \big| \mathcal{V}(I_u)  \big|^2       \Big|^{p/2} \Big]>0.
\end{align*}
Then Theorem~\ref{thm-precise} implies that 
\begin{align*}
\E[|\widehat{\mu}_\infty(1)|^p]=\lim_{n \rightarrow\infty}\E[|\widehat{\mu}_n(1)|^p] =\limsup_{n\in \N}\E[|\widehat{\mu}_n(1)|^p] \geq C_p   \|\mathring{W}\|_1^p\cdot    \chi(\mathcal{V}, W, p)>0.   
\end{align*}

We now proceed to the proof of the second part of the lemma. Using \eqref{eqn-psi-t-star}, we obtain
\[
\E\Big[\big| b^{\frac{n D_F}{2}}   n^{\frac{3}{2}\beta} \widehat{\mu}_{\infty}\left(b^n\right)\big|^{p\gamma}\Big] =   \E\Big[ \Big|     \sum_{|u|=n} e^{-\beta ( V(u)-\frac{3}{2}\log n )}\widehat{\mu}_{\infty}^{(u)}(1)  \Big|^{p\gamma}   \Big]. 
\]
Since $0< \gamma< \frac{1}{p\beta}< 1$,  by Jensen's inequality 
\[
\E\Big[ \Big|     \sum_{|u|=n} e^{-\beta ( V(u)-\frac{3}{2}\log n )}\widehat{\mu}_{\infty}^{(u)}(1)  \Big|^{p\gamma}   \Big|\mathscr{F}_n\Big]  \le   \Big( \E\Big[ \Big|     \sum_{|u|=n} e^{-\beta ( V(u)-\frac{3}{2}\log n )}\widehat{\mu}_{\infty}^{(u)}(1)  \Big|^{p}   \Big|\mathscr{F}_n\Big] \Big)^\gamma.
\]
Now since  $1< \frac{1}{\beta}< p\le 2$, we can apply Burkholder's martingale  inequality  (since, conditioned on $\mathscr{F}_n$, the sequence  $(\widehat{\mu}_{\infty}^{(u)}(1))_{|u|=n}$ are i.i.d. and mean zero) and obtain 
\begin{align*}
\E\Big[ \Big|     \sum_{|u|=n} e^{-\beta ( V(u)-\frac{3}{2}\log n )}\widehat{\mu}_{\infty}^{(u)}(1)  \Big|^{p}   \Big|\mathscr{F}_n\Big]   \lesssim_p   \E\Big[ \Big|     \sum_{|u|=n} e^{-2 \beta ( V(u)-\frac{3}{2}\log n )}|\widehat{\mu}_{\infty}^{(u)}(1)|^2  \Big|^{\frac{p}{2}}   \Big|\mathscr{F}_n\Big]. 
\end{align*}
Thus, since $p/2\le 1$,  by applying \eqref{2-ele-ineq}, we obtain 
\begin{align*}
\E\Big[ \Big|     \sum_{|u|=n} e^{-\beta ( V(u)-\frac{3}{2}\log n )}\widehat{\mu}_{\infty}^{(u)}(1)  \Big|^{p}   \Big|\mathscr{F}_n\Big]   & \lesssim_p   \E\Big[ \Big|     \sum_{|u|=n} e^{- p \beta ( V(u)-\frac{3}{2}\log n )}|\widehat{\mu}_{\infty}^{(u)}(1)|^{p}  \Big|\mathscr{F}_n\Big] 
\\
& =        \sum_{|u|=n} e^{- p \beta ( V(u)-\frac{3}{2}\log n )} \E[ |\widehat{\mu}_{\infty}^{(u)}(1)|^{p}].  
\end{align*}
Therefore, we arrive at 
\begin{align*}
\E\Big[\big| b^{\frac{n D_F}{2}}   n^{\frac{3}{2}\beta} \widehat{\mu}_{\infty}\left(b^n\right)\big|^{p\gamma}\Big]\lesssim_p  \E\Big[\Big(\sum_{|u|=n} e^{-\beta p ( V(u)-\frac{3}{2}\log n )}\Big)^{\gamma}\Big]\,\,\cdot(\E[|\widehat{\mu}_{\infty}(1)|^p])^\gamma.
\end{align*}
Since   $\beta p>1$ and $\gamma<\frac{1}{\beta p}$, by  \eqref{s-m-asymp} and  \eqref{eqn-finite-p-moment}, 
\[
\sup_{n\in\N^+}\E\Big[\big| b^{\frac{n D_F}{2}}   n^{\frac{3}{2}\beta} \widehat{\mu}_{\infty}\left(b^n\right) \big|^{p\gamma}\Big]<\infty.
\]
This completes the proof.
\end{proof}

\begin{proof}[Proof of Lemma \ref{lem-wel-def}]
Let $p\in (1/\beta, 2]$ be as in Lemma \ref{lemma-tightness}. Then $\E[ |\widehat{\mu}_\infty(1)|^p] <\infty$.   By Burkholder's inequality, for any finite subset $F\subset \N$, we have 
\[
  \E \Big[ \Big|   \sum_{i \in F}  e^{-\beta x_i}   \mathrm{Re} (\widehat{\mu}^{(i)}_\infty(1)) \Big|^p  \Big| \mathcal{Z}_\infty\Big]   \lesssim_p     
    \sum_{i  \in F}  e^{- p \beta x_i}   \cdot  \E[ |\widehat{\mu}_\infty(1)|^p] \le  \sum_{i  \in \N}  e^{- p \beta x_i}   \cdot  \E[ |\widehat{\mu}_\infty(1)|^p]. 
\]
By \eqref{Z-finite}, we complete the whole proof. 
\end{proof}

\begin{proof}[Proof of Proposition~\ref{thm-non-cri-super-critical}]
The proof of Proposition~\ref{thm-non-cri-super-critical} is divided into three steps.

{\flushleft \it Step 1: Tightness and limit.}
Lemma~\ref{lemma-tightness} implies  the tightness of the family 
\[
\Big\{ b^{\frac{n D_F}{2}}   n^{\frac{3}{2}\beta} \widehat{\mu}_{\infty}(b^{n})  : n\in \N\Big\}. 
\]
Recall \eqref{eqn-psi-t-star}, 
\begin{align*}
&b^{\frac{n D_F}{2}}   n^{\frac{3}{2}\beta} \widehat{\mu}_{\infty}(b^{n})  = \sum_{|u|=n} e^{-\beta (V(u) -\frac32\log n)} \widehat{\mu}_{\infty}^{(u)}(1) \\
= &  \sum_{|u|=n} e^{-\beta (V(u) -\frac32\log n)} f_K(V(u) - \frac32\log n)\widehat{\mu}_{\infty}^{(u)}(1) +  \sum_{|u|=n} e^{-\beta (V(u) -\frac32\log n)} \overline{f}_K(V(u)-\frac32\log n)\widehat{\mu}_{\infty}^{(u)}(1)  \\
=: & Y_{n,K} +R_{n,K},
\end{align*}
where $f_K(x) = 1_{(-\infty, K]}(x)+ (K+1-x)1_{(K, K+1]}(x)$ and $\overline{f}_K=1-f_K$. By using the similar arguments as \cite[Theorem 2.5]{Iksanov}, we could check that
\begin{itemize}
\item $R_{n,K}$ converges in probability to zero as $n\to\infty$ and then $K\to\infty$.
\item $Y_{n,K}$ converges in law to $Y_K:=  \sum_{i\in \N} e^{-\beta x_i}f_K(x_i) \widehat{\mu}^{(i)}_\infty(1) $ as $n\to\infty$.
\item $Y_K$ converges in law to $\Upsilon:=  \sum_{i\in \N} e^{-\beta x_i}\widehat{\mu}^{(i)}_\infty(1)$ as $K\uparrow\infty$.
\item $\Upsilon$ is well-defined r.v. as it is the limit of $L^p$-bounded martingale (conditionally on $\sum_i \delta_{x_i}$) $\{\sum_{i=1}^m e^{-\beta x_i}\widehat{\mu}^{(i)}_\infty(1) \}_{m\ge1}$. See Lemma \ref{lem-wel-def}.
\end{itemize}
We thus obtain that 
\begin{align}\label{eqn-limit-exists-s}
   b^{\frac{n D_F}{2}}   n^{\frac{3}{2}\beta} \widehat{\mu}_{\infty}(b^{n})     \xrightarrow[n\to\infty]{d}        \Upsilon = \sum_{i\in \N} e^{-\beta x_i}\widehat{\mu}^{(i)}_\infty(1),
\end{align}
where $(x_i)_{i\in\N}$ is the  point process  $\mathcal{Z}_\infty$ in \eqref{eqn-z-infinity-s},  $\{\widehat{\mu}^{(i)}_\infty(1)\}_{i\in \N}$ are i.i.d. copies of $\widehat{\mu}_\infty(1)$  independent of $\mathcal{Z}_\infty$. 

{\flushleft \it Step 2: Non-trivial limit.} We shall show that the limit in \eqref{eqn-limit-exists-s} is non-trivial (not identically zero).  Indeed,  let  $p\in (1/\beta,2]$ be as in Lemma~\ref{lemma-tightness}. Since the point process $\mathcal{Z}_\infty=(x_i)_{i\in\N}$ and the sequence of random variables $(\widehat{\mu}_{\infty}^{(i)} (1) )_{i\in\N}$ are independent,  Burkholder's  inequality implies that 
\begin{align*}
\E\Big[\big|\sum_{i\in \N} e^{-\beta x_i}\widehat{\mu}_{\infty}^{(i)} (1) \big|^p\mid \mathcal{Z}_\infty \Big]& \gtrsim_p \E\Big[\big(\sum_{i\in \N} e^{-2\beta x_i}|\widehat{\mu}_{\infty}^{(i)} (1) |^2\big)^{\frac{p}{2}}\mid \mathcal{Z}_\infty \Big]\geq e^{-p\beta x_1}\E[|\widehat{\mu}_\infty(1)|^p].
\end{align*}
Since  the point process   $\mathcal{Z}_\infty$ is non-trivial, 
we have 
\[
\E\Big[\big|\sum_{i\in \N} e^{-\beta x_i}\widehat{\mu}_{\infty}^{(i)} (1) \big|^p\Big]>0.
\]

{\flushleft \it Step 3: Stable law.}    
By \cite[Theorem~1.1]{Mad}, under $\PP^*$, the point process $\mathcal{Z}_\infty$ in \eqref{eqn-z-infinity-s} has the following description:  Take the $\PP^*$-a.s. limit of derivative martingale 
\begin{equation}\label{eqn-dmartingale-beta}
\mathscr{D}_\infty(W^{(\frac1\beta)})=\lim_{n\to\infty}\sum_{|u|=n} V(u) e^{-V(u)}
\end{equation}
with 
\[
 \PP^*(\mathscr{D}_\infty(W^{(\frac1\beta)})\in(0,\infty))=1.
\]
Let $\mathcal{P}$ be an independent Poisson point process on the whole line $\R$ whose intensity measure is given by  $\lambda\mathrm{e}^y \mathrm{~d} y$ for some $\lambda>0$. For each atom $y$ of $\mathcal{P}$, we attach a point process $\mathcal{Q}^{(y)}$, where $\mathcal{Q}^{(y)}$ are independent copies  (given $\mathcal{P}$) of a certain point process $\mathcal{Q}$.  Denote $\mathcal{P}=\{y_i: i\in\N\}$ and $\mathcal{Q}^{(y_i)}=\{z_j^{(i)}: i\in \N\}$, which are all assumed to be independent of $\mathscr{D}_\infty(W^{(\frac1\beta)})$.   Then 
\[
\mathcal{Z}_\infty \stackrel{d}{=}\{y_i+z_j^{(i)} - \log \mathscr{D}_\infty(W^{(\frac1\beta)}): i, j\in\N\}.
\]
It follows that 
\[
X:=       \sum_{i\in \N} e^{-\beta x_i}   \mathrm{Re}( \widehat{\mu}^{(i)}_\infty(1) )  \stackrel{d}{=}  \mathscr{D}_\infty(W^{(\frac1\beta)})^\beta  \underbrace{ \sum_{i, j\in\N}e^{-\beta (y_i+z_j^{(i)})}\mathrm{Re}(\widehat{\mu}^{(i, j)}_\infty(1))}_{\text{denoted as $\widetilde{X}$}}=\mathscr{D}_\infty(W^{(\frac1\beta)})^\beta \sum_{i\in \N}e^{-\beta y_i} R_i
\]
 with 
 \[
 R_i=\sum_{j\in\N} e^{-\beta z_j^{(i)}}\mathrm{Re}(\widehat{\mu}^{(i, j)}_\infty(1)). 
\]
Note that, given $\mathcal{Z}_\infty$,  both summations for $\widetilde{X}$ and   $R_i$ are  well-defined (see Lemma \ref{lem-partition}  below and are understood as in \eqref{def-Spk}). 

The transformed  point process  $\widehat{\mathcal{P}}=\{\widehat{y}_i=e^{-\beta y_i}: i\in \N\}$ of $\mathcal{P}$  is  a Poisson point process on $\R_+$ with new intensity measure given by 
\[
\rho(dy)=\frac{\lambda}{\beta}y^{-1-\frac{1}{\beta}}  \mathds{1}(y>0) dy.
\]
Thus we can write 
  \[
  \widetilde{X}= \sum_{\widehat{y}_i\in \widehat{\mathcal{P}}}\widehat{y}_i R_i  \an X \overset{d}{=} \mathscr{D}_\infty(W^{(\frac1\beta)})^\beta\widetilde{X},
  \]
   where $\mathscr{D}_\infty(W^{(\frac1\beta)})$ and $\widetilde{X}$ are independent, and $(R_i)_{i\in\N}$ are i.i.d. copies of a random variable $R$ and the sum is understood as a   martingale limit (given $\widehat{\mathcal{P}}$). 

For any $\varepsilon>0$, consider the truncated sum
\[
\widetilde{X}_\varepsilon=\sum_{\widehat{y}_i\in \widehat{\mathcal{P}}}\mathds{1}(\widehat{y}_i\ge \varepsilon)\widehat{y}_i R_i.
\]
Almost surely, $\widetilde{X}_\varepsilon$ is a  finite sum. 
 Under $\PP^*$,  the Fourier transform of $\widetilde{X}_\varepsilon$  is given by 
\begin{align*}
\E^*[\exp( it \widetilde{X}_\varepsilon)]&=\E^*\Big[\prod_{\widehat{y}_i \in \widehat{\mathcal{P}} \cap [\varepsilon, \infty)}   \exp (i t  \widehat{y}_i R_i)   \Big] =\E^*\Big [   \E^* \Big[ \prod_{\widehat{y}_i \in \widehat{\mathcal{P}} \cap [\varepsilon, \infty)}   \exp (i t \widehat{y}_i R_i)    \Big| \widehat{\mathcal{P}} \Big] \Big]. 
\end{align*}
Let  $f_R(t): = \E^*[\exp(itR)]$ be the Fourier transform of the random variable $R$. Then by the  independence between  $(R_i)_{i\in\N}$ and $\widehat{\mathcal{P}}$,
\[
 \E^* \Big[ \prod_{\widehat{y}_i \in \widehat{\mathcal{P}} \cap [\varepsilon, \infty)}   \exp (i t \widehat{y}_i R_i)    \Big| \widehat{\mathcal{P}} \Big] =  \prod_{\widehat{y}_i \in \widehat{\mathcal{P}} \cap [\varepsilon, \infty)}      f_R (t\widehat{y}_i). 
\]
Therefore, by the Bogoliubov functional (or the Laplace transform) of the Poisson point process (see \cite[Section 3.2]{Kingman}), we have 
\begin{align*}
\E^*[\exp( it\widetilde{X}_\varepsilon)] & =  \E^*\Big[  \prod_{\widehat{y}_i \in \widehat{\mathcal{P}} \cap [\varepsilon, \infty)}      f_R (t\widehat{y}_i)\Big]
\\
&  =\exp\Big(\int_{\varepsilon}^\infty \big(f_R(ty)-1\big)  \frac{\lambda dy}{ \beta y^{1+1/\beta}}\Big) 
\\
 &  = \exp\Big({|t|^{\frac{1}{\beta}}} \int_{ |t| \varepsilon}^\infty  (f_R( \sgn(t) \cdot y)-1) \frac{\lambda dy}{\beta y^{1+1/\beta}}\Big),
\end{align*}
where in the last step, we used the  change of variable $|t|y \to y$.  
By Dominated Convergence Theorem, 
\[
\lim_{\varepsilon\to 0}  \E^*[\exp( it \widetilde{X}_\varepsilon)] =  \E^*[\exp( it\widetilde{X})]. 
\]
It follows that the following limit exists 
\[
\lim_{\varepsilon\to 0} \int_{ |t| \varepsilon}^\infty  (f_R( \sgn(t)\cdot y)-1) \frac{\lambda dy}{\beta y^{1+1/\beta}} =  \lim_{\varepsilon\to 0} \int_{\varepsilon}^\infty  (f_R( \sgn(t) \cdot y)-1) \frac{\lambda dy}{\beta y^{1+1/\beta}}. 
\]
Consequently, 
\[
\E^*[\exp( it\widetilde{X})] = \exp\big( (c_1  + i \sgn(t) c_2)  |t|^{\frac{1}{\beta}}\big)
\]
and 
\[
 \E^*[\exp( it X)] =
\E^*\big[ \exp\big( (c_1  + i \sgn(t) c_2)  |t|^{\frac{1}{\beta}}\mathscr{D}_\infty(W^{(\frac1\beta)}) \big)\big],
\]
where 
\begin{align}\label{c-1}
c_1  =  \int_{0}^\infty  ( \E^*[\cos (Ry) ]  -1) \frac{\lambda dy}{\beta y^{1+1/\beta}} <0
\end{align}
and 
\begin{align}\label{c-2}
c_2  =   \lim_{\varepsilon\to 0}   \int_{\varepsilon}^\infty \E^*[\sin (Ry) ]   \frac{\lambda dy}{\beta y^{1+1/\beta}}. 
\end{align}
In other words, 
\[
\E^*[\exp( it\widetilde{X})] = \exp\big( -c  (1 + i \sgn(t) c')  |t|^{\frac{1}{\beta}}\big) \quad \text{with $c> 0$}.  
\]
Hence, under $\PP^*$, the random variable  $\widetilde{X}$ has a $\frac{1}{\beta}$-stable law.  The desired equality \eqref{zero-prob} follows since  any stable law is  continuous. 
 \end{proof}

\begin{remark*}
Since $c_1$  in \eqref{c-1} is a finite real number and $\cos(Ry) -1\le 0$ does not change signs, by Fubini Theorem and change of variable $|R|y \to y$, we obtain 
\[
c_1 =     -  \E^*[|R|^{1/\beta}]  \cdot   \int_{0}^\infty  (1- \cos y) \frac{\lambda dy}{\beta y^{1+1/\beta}}    =   - \frac{ \lambda \pi  \E^*[|R|^{1/\beta}] } {  2  \sin (\pi /2 \beta)   \Gamma(1/\beta)   } . 
\]
Hence, we must have 
$
\E^*[|R|^{1/\beta}]<\infty.
$  Similarly, since $\E^*(R)=0$, the quantity  $c_2$ in \eqref{c-2} has the following form:
\begin{align*}
c_2 = \E^*[\sgn(R) |R|^{1/\beta}]  \cdot  \int_{0}^\infty  (\sin y - y)   \frac{\lambda dy}{\beta y^{1+1/\beta}}
=          \frac{\lambda \pi \E^*[\sgn(R) |R|^{1/\beta}]  }{ 2 \cos(\pi /2\beta)   \Gamma(1/\beta)}. 
\end{align*}
\end{remark*}

The following Lemma \ref{lem-partition} follows immediately from   Lemma \ref{lem-wel-def} and we omit its routine proof. 

\begin{lemma}\label{lem-partition}
Let $p\in (1/\beta, 2]$ be taken as in Lemma \ref{lemma-tightness}. For an arbitrary partition of $\N$: 
\[
\N = \bigsqcup_{k=1}^\infty P_k \quad \text{with $P_k$ not necessarily finite}, 
\]
we may define the  martingale limits  (in the sense of $L^p (\PP(\cdot | \mathcal{Z}_\infty))$-convergence)
\begin{align}\label{def-Spk}
S(P_k): = \lim_{N\to\infty}\sum_{ i \in P_k \cap [1, N]} e^{-\beta x_i}   \mathrm{Re} (\widehat{\mu}^{(i)}_\infty(1))
\end{align}
and obtain the almost sure equality under $\PP(\cdot |\mathcal{Z}_\infty)$ 
\[
\lim_{N\to\infty}\sum_{i=1}^N e^{-\beta x_i}   \mathrm{Re} (\widehat{\mu}^{(i)}_\infty(1))   =  \lim_{L\to\infty} \sum_{k=1}^L  S(P_k). 
\]
\end{lemma}

 \subsection{Proof of Proposition \ref{prop-clt}}

\begin{lemma}\label{lemma-cov-law-con-two}
Let $(X_n)_{n\geq 1}$  and $X_\infty$ be  complex random variables on some probability space $(\Omega,\mathcal{F},\PP)$. Assume that  $E\in\mathcal{F} $ satisfies $\PP(E)=\PP(X_\infty\neq 0)\in (0,1)$ and $X_n=0$ on $E^c$ for any $n\geq 1$ . 
If 
\[
X_n\xrightarrow[n\rightarrow\infty]{d} X_\infty,
\]
then the following conditioned convergence in distribution holds:
\[
\mathcal{L}(X_n\mid E)\xrightarrow[]{n\rightarrow\infty} \mathcal{L}(X_\infty|X_\infty\neq 0).
\]
\end{lemma}
\begin{proof}
For any bounded continuous function $f$, we have
\begin{align*}
\E[f(X_n)\mid E]&=\frac{\E[\mathds{1}_Ef(X_n)]}{\PP(E)}=\frac{\E[f(X_n)]-\E[f(X_n)\mathds{1}_{E^c}]}{\PP(E)} =\frac{\E[f(X_n)]-f(0)\PP(E^c)}{\PP(E)}.
\end{align*}
Therefore, by the  assumption that  $X_n$ converges in distribution to $X_\infty$ and $\PP(E)= \PP(X_\infty \ne 0)$, 
\begin{align*}
\lim_{n\rightarrow\infty}\E[f(X_n)\mid E]=\frac{\E[f(X_\infty)]-f(0) \PP(X_\infty=0)}{\PP(X_\infty\neq 0)} =\E[f(X_\infty)\mid X_\infty\neq 0].
\end{align*}
This completes the proof.
\end{proof}

 \begin{proof}[Proof of Proposition \ref{prop-clt}]
Recall the notation in Proposition \ref{prop-clt}.  By \cite[\S 3.2]{shi-zhan-book},  one has 
\begin{itemize}
\item 
in the  squared sub-critical regime  (see Remark \ref{rk-w-w2} below)
\begin{align}\label{w-w2}
\PP(\mathscr{M}_\infty(W)\neq 0)=\PP(\mathscr{M}_\infty(W^{(2)})\neq 0)=\PP(\mu_\infty\neq 0); 
\end{align}
\item  in the  squared critical regime 
\[
\PP(\mathscr{D}_\infty(W^{(2)})\neq 0)=\PP(\mu_\infty\neq 0); 
\] 
\item in the 
 squared super-critical regime 
\[
\PP(\mathcal{Z}_\infty\neq 0)=\PP(\mu_\infty\neq 0). 
\]
\end{itemize}
Therefore, Proposition \ref{prop-clt} follows from Propositions~\ref{thm-sq-sub-critical}, \ref{thm-sq-critical} and \ref{thm-non-cri-super-critical} and Lemma \ref{lemma-cov-law-con-two}. 
\end{proof}

\begin{remark}\label{rk-w-w2}
In the squared sub-critical regime, using the self-similar structure,  we have (see for instance the proof of  \cite[Proposition 3.3]{Heu}) 
\[
\PP(\mathscr{M}_\infty(W) = 0)=\PP(\mu_\infty = 0)  = x_0 \in [0,1)
\]
with $x_0$ the  unique non-trivial fixed point (the point $x=1$ is a trivial fixed point)  of the map 
\[
[0,1]\ni x\mapsto \big(\PP(W=0)+\PP(W>0)x\big)^b \in [0,1].
\] Then using  $\PP(W=0) = \PP(W^{(2)}=0)$, we obtain  
\[
\PP(\mathscr{M}_\infty(W^{(2)}) =0)   = \PP(\mathscr{M}_\infty(W) = 0). 
\]
\end{remark}

\subsection{Proof of Proposition \ref{prop-to-infty}}

\begin{lemma}\label{lemma-cov-law-con-one}
Let $(X_n)_{n\geq 1}$ be a sequence of complex random variables  converging in distribution to a  random variable $X_\infty$ with $\PP(X_\infty\neq 0)=1$, then for any positive increasing sequence $(a_n)_{n\in\N}$ with $\lim_{n\rightarrow\infty} a_n=\infty$, we have
\[
a_n|X_n|\xrightarrow[n\rightarrow\infty ]{\text{in probability}}  \infty.
\]
That is, for any $M>0$, 
\[
\lim_{n\rightarrow\infty} \PP(a_n|X_n|>M)=1.
\]
\end{lemma}

\begin{proof}
For  any $M>0$  and $\varepsilon>0$, there exists $N>0$ such that $\frac{M}{a_n}\leq \varepsilon$ for any $n\geq N$. Then 
\[
\PP(a_n|X_n|>M)\geq \PP(|X_n|>\varepsilon) \quad \text{for $n\geq N$}.
\]
Hence, by the convergence in distribution assumption, 
\[
\liminf_{n\rightarrow\infty} \PP(a_n|X_n|>M)\geq \liminf_{n\rightarrow\infty} \PP(|X_n|>\varepsilon) \ge \PP(|X_\infty|>\varepsilon).
\]
Since $\varepsilon>0$ is arbitrary, we have
\[
\liminf_{n\rightarrow\infty} \PP(a_n|X_n|>M)\geq \PP(|X_\infty|>0)=1.
\]
This completes the whole proof. 
\end{proof}

\begin{proof}[Proof of Proposition \ref{prop-to-infty}]  
By Proposition \ref{prop-clt} and  Lemma \ref{lemma-cov-law-con-two},  under $\PP^* = \PP(\cdot |\mu_\infty\ne 0)$,   one has 
\[
r(n) \cdot  b^{\frac{n D_F}{2}}  \cdot \widehat{\mu}_\infty(b^{n})  \xrightarrow[n\to\infty]{d}   \mathcal{X}_\infty, 
\]
with  $\PP^* ( \mathcal{X}_\infty \ne 0)   = 1$  almost surely  and 
$
r(n) = n^\theta
$ with $\theta$ takes values in  $\{0, \frac{1}{4}, \frac{3\beta}{2}\}$  according to the squared sub-critical, squared critical and squared super-critical regimes respectively.   

Therefore,  for any $\varepsilon>0$,   by applying Lemma \ref{lemma-cov-law-con-one} and  $r(n) = o(b^{n\varepsilon})$, we have 
\[
b^{\frac{n (D_F+\varepsilon)}{2}} |\widehat{\mu}_\infty(b^{n})|   \xrightarrow[n\to\infty]{\text{in probability $\PP^*$}} \infty. 
\]
It follows that, along a subsequence $(n_k)$, the above convergence holds $\PP^*$-almost surely and  we complete the proof of Proposition \ref{prop-to-infty}. 
\end{proof}

\section{Proofs of Theorem \ref{thm-fourier} and Corollaries  \ref{cor-salem}, \ref{cor-holder}}\label{sec-con-1}

\subsection{Proof of Theorem \ref{thm-fourier}}

We start with the proof of Lemma \ref{lem-trans-coef} and Corollary \ref{cor-series-dim}.

\begin{proposition}\label{lem-Fdecay}
Let $\nu$ be a  Borel probability measure on $[0,1]$.  Assume that there exist $\kappa\in (0,1]$ and $c>0$ such that 
\begin{align}\label{ass-poly-decay}
|\widehat{\nu}(l)|\le  \frac{c}{|l|^{\kappa}} \quad \text{for all $l\in \Z^* = \Z\setminus \{0\}$}.
\end{align}
Then there exists $c'>0$ such that
\begin{align}\label{rest-fourier}
\sup_{\text{$I \subset[0,1]$: sub-intervals}}\Big|\int_{I}e^{-i 2\pi l t}\nu(dt)\Big| \le c' \cdot \frac{\log l}{l^{\kappa}} \quad \text{for all $l>1$}. 
\end{align}
\end{proposition}

\begin{remark*}
Proposition \ref{lem-Fdecay} can not be extended to $\kappa>1$ since the restriction of the Lebesgue measure onto a proper sub-interval has a Fourier coefficient decay as $l^{-1}$. Indeed, one can see that  the upper estimate \eqref{small-part}  in the proof below of Proposition~\ref{lem-Fdecay} is not valid for $\kappa >1$ since 
\[
\sum_{1\le l \le s/2} \frac{2}{s l^\kappa} \ge \frac{2}{s},
\]
which decays slower than $s^{-\kappa}$.
\end{remark*}

\begin{lemma}\label{lem-convol}
Assume that $\nu$ is a Borel probability measure on $[0,1]$ satisfying the Fourier coefficient estimate \eqref{ass-poly-decay}.  Then for any sub-interval $I\subset [0,1]$ and any $s\in \N$, 
\[
\int_I e^{-i 2 \pi s t} \nu(dt)  =  (\widehat{\indi_I}* \widehat{\nu})(s),
\]
where $\widehat{\indi_I}* \widehat{\nu}$ is understood as the discrete convolution of two sequences in $\ell^{q'}$ and $\ell^q$ respectively (here $q>\kappa^{-1}$ and $q'$ is the conjugate exponent of $q$): 
\[
(\widehat{\indi_I}* \widehat{\nu})(s) = \sum_{l\in \Z} \widehat{\indi_I} (s-l) \widehat{\nu}(l).
\]
\end{lemma}

\begin{proof}
First of all, by the polynomial decay assumption \eqref{ass-poly-decay} and H\"older's inequality,  we have 
\begin{align}\label{holder-ineq}
\sum_{l\in \Z} |\widehat{\indi_I} (s-l) \widehat{\nu}(l)|<\infty \quad \text{for all $s\in \N$}. 
\end{align}
Now for any $r\in (0,1)$, consider the Poisson convolutions $\nu_r$  of $\nu$: 
\[
\nu_r (t): =  \sum_{l\in \Z} r^{|l|} \widehat{\nu}(l) e^{i 2\pi l t}, \quad  t \in [0,1].
\]
It is a standard fact that  the sequence of the probability measures $\nu_r(dt)= \nu_r(t)dt$ converges weakly to $\nu$. Moreover, since $\nu$ satisfies the polynomial decay \eqref{ass-poly-decay}, it has no atoms. Therefore, by restricting onto any sub-interval $I$,  the sequence of measures $\nu_r|_I$ converges weakly to $\nu|_I$. Hence for any $s \in \Z$, we have 
\[
\int_I e^{-i 2 \pi  s t} \nu(dt) = \lim_{r\to 1^{-}} \int_I e^{-i 2 \pi s  t} \nu_r(dt).
\]
Note that for any $r\in (0,1)$, we have 
\[
\int_I e^{-i 2 \pi s  t} \nu_r(dt) = \widehat{\indi_I \cdot \nu_r} (s) = ( \widehat{\indi_I} * \widehat{\nu_r})(s) = \sum_{l\in \Z} \widehat{\indi_I} (s-l)   r^{|l|}\widehat{\nu}(l). 
\]
Therefore, by \eqref{holder-ineq}, we may apply the Dominated Convergence Theorem to get 
\[
\lim_{r\to 1^{-}}  \int_I e^{-i 2 \pi s  t} \nu_r(dt)  = \sum_{l\in \Z} \widehat{\indi_I} (s-l)  \widehat{\nu}(l). 
\]
This ends the proof. 
\end{proof}

\begin{proof}[Proof of Proposition \ref{lem-Fdecay}]
Note that 
\[
 \sup_{\text{$I \subset[0,1]$: sub-intervals}}|\widehat{\indi_I} (l)|  \le \pi^{-1} l^{-1}\quad \text{for all $l\in \Z^*$}. 
\]
Now 
\begin{align*}
(\widehat{\indi_I}* \widehat{\nu})(s) = \sum_{l\in \Z} \widehat{\indi_I} (s-l) \widehat{\nu}(l)  =       \widehat{\indi_I} (0) \widehat{\nu}(s) +    \widehat{\indi_I} (s) \widehat{\nu}(0)    +\sum_{l\in \Z\setminus \{s, 0\}} \widehat{\indi_I} (s-l) \widehat{\nu}(l).
\end{align*}
One has
\begin{align*}
\Big| \sum_{l\in \Z\setminus \{s,0\}} \widehat{\indi_I} (s-l) \widehat{\nu}(l)\Big|  & \le  \sum_{l\in \Z\setminus \{s,0\}} |\widehat{\indi_I} (s-l) \widehat{\nu}(l)|   \le  c\pi^{-1} \sum_{l\in \Z\setminus \{s,0\}}      \frac{1}{|s-l|} \frac{1}{|l|^\kappa}
\\
& \le c\pi^{-1} \frac{1}{2s} \cdot \frac{1}{s^\kappa} + c\pi^{-1} \sum_{l\in \Z\setminus \{\pm s,0\}}      \frac{1}{|s-l|} \frac{1}{|l|^\kappa}
\\
& \le c\pi^{-1} \frac{1}{2s} \cdot \frac{1}{s^\kappa} + c\pi^{-1} \sum_{l\in \Z\setminus \{\pm s,0\}}      \frac{1}{|s-|l||} \frac{1}{|l|^\kappa}
\\
& \le c\pi^{-1} \frac{1}{2s} \cdot \frac{1}{s^\kappa} + 2  c\pi^{-1} \sum_{l\in \N\setminus \{s\}}      \frac{1}{|s-l|} \frac{1}{l^\kappa}.  
\end{align*}
Write 
\[
\sum_{l\in \N\setminus \{s\}}      \frac{1}{|s-l|} \frac{1}{l^\kappa}= \sum_{l=1}^{s-1}   \frac{1}{(s-l) l^\kappa}  + \sum_{l = s+1}^\infty  \frac{1}{(l-s) l^\kappa}.
\]

First, we assume that $\kappa \in (0,1)$.  Then 
\begin{align}\label{small-part}
\begin{split}
\sum_{l=1}^{s-1}      \frac{1}{(s-l) l^\kappa}  &  \le \sum_{1\le l \le s/2} \frac{1}{(s-l) l^\kappa}  + \sum_{s/2 \le l \le s-1} \frac{1}{(s-l) l^\kappa}  
\\
&\le \sum_{1\le l \le s/2} \frac{2}{s l^\kappa} + \sum_{s/2 \le l \le s-1} \frac{1}{s-l} \cdot     \frac{2^\kappa}{s^{\kappa}} \lesssim s^{-\kappa} + s^{-\kappa} \log s \lesssim s^{-\kappa} \log s.
\end{split}
\end{align}
Similarly, 
\begin{align*}
\sum_{l = s+1}^\infty  \frac{1}{(l-s) l^\kappa}  & = \sum_{l = s+1}^{2s}  \frac{1}{(l-s) l^\kappa}  + \sum_{l = 2s+1}^\infty  \frac{1}{(l-s) l^\kappa}  
\\
& \le \sum_{l = s+1}^{2s}  \frac{1}{(l-s) s^\kappa} + \sum_{l = 2s+1}^{\infty}  \frac{1}{ \frac{l}{2} \cdot l^\kappa}  \lesssim s^{-\kappa} \log s + s^{-\kappa}  \lesssim s^{-\kappa} \log s. 
\end{align*}

For $\kappa = 1$, we have 
\begin{align*}
\sum_{l=1}^{s-1}      \frac{1}{(s-l) l}  &  \le \sum_{1\le l \le s/2} \frac{1}{(s-l) l}  + \sum_{s/2 \le l \le s-1} \frac{1}{(s-l) l}  
\\
&\le \sum_{1\le l \le s/2} \frac{2}{s l} + \sum_{s/2 \le l \le s-1} \frac{1}{s-l} \cdot     \frac{2}{s}  \lesssim  \frac{\log s}{s} + \frac{\log s}{s} \lesssim \frac{\log s}{s}.
\end{align*}
Similarly, 
\begin{align*}
\sum_{l = s+1}^\infty  \frac{1}{(l-s) l}  & = \sum_{l = s+1}^{2s}  \frac{1}{(l-s) l}  + \sum_{l = 2s+1}^\infty  \frac{1}{(l-s) l}  
\\
& \le \sum_{l = s+1}^{2s}  \frac{1}{(l-s) s} + \sum_{l = 2s+1}^{\infty}  \frac{1}{ \frac{l}{2} \cdot l}  \lesssim  \frac{ \log s}{s}+ \frac{1}{s} \lesssim  \frac{ \log s}{s}. 
\end{align*}
This ends the proof. 
\end{proof}

\begin{proof}[Proof of Lemma \ref{lem-trans-coef}]
By Proposition~\ref{lem-Fdecay},  if $|\widehat{\mu}(n)|  = O(n^{-\kappa})$ with $\kappa \in (0, 1]$ as $\N\ni n\to\infty$, then 
\[
|\widehat{\mu}_0(n)|  = O(n^{-\kappa} \log n) \an  |\widehat{\mu}_1(n)| = O(n^{-\kappa} \log n),
\]
where $\mu_0 = \mu|_{[0,1/2)}$ and $\mu_1= \mu|_{[1/2, 1]}$ are the restrictions of $\mu$ on two sub-intervals. Now since both $\supp(\mu_0)$ and $\supp(\mu_1)$ are contained in a sub-interval of length strictly less than $1$,  we can apply \cite[Chapter 17, Lemma 1]{Kahane-book} to $\mu_0$ and $\mu_1$ respectively and obtain 
\[
\text{$|\widehat{\mu}(\zeta)| \le    |\widehat{\mu}_0(\zeta)|  + |\widehat{\mu}_1(\zeta)|   = O  \big( |\zeta|^{-\kappa} \log (|\zeta|)\big)$ as $\R\ni \zeta \to \infty$.}
\]
The above argument shows that 
\[
 \sup\{D \in[0,1]:   |\widehat{\mu}_\infty(\zeta)|^2 = O(|\zeta|^{-D}) \}     \le    \sup \{ D \in [0,1):     |\widehat{\mu}(n)|^2  = O(n^{-D})   \}.
\]
The converse inequality is trivial. Hence we complete the proof of the lemma. 
\end{proof}

\begin{proof}[Proof of Corollary \ref{cor-series-dim}]
Let $\mu$ be a finite positive Borel  measure on $[0,1]$.  By Lemma \ref{lem-trans-coef}, we have 
\[
\dim_F(\mu) =  \sup \{ D \in [0,1):     |\widehat{\mu}(n)|^2  = O(n^{-D})   \}.
\]
Set
 \[
D'(\mu) := \sup \Big\{D \in [0,1):    \text{$  \sum_{n\ge 1 } | n^{\frac{D}{2}}  \cdot \widehat{\mu}(n)|^q <\infty$ for some $q>2$}    \Big\}. 
\]
If $0\le D<1$ is such that  
$
 |\widehat{\mu}(n)|^2  = O(n^{-D}). $ Then for any $\varepsilon\in (0,1)$ and $q>2/\varepsilon>2$, we have  
 \[
  \sum_{n\ge 1 } | n^{\frac{D-\varepsilon}{2}}  \cdot \widehat{\mu}(n)|^q    \le \sup_{n\in \N} (n^D |\widehat{\mu}(n)|^2)^{q/2}    \cdot \sum_{n\ge 1} n^{- \frac{\varepsilon q}{2}}<\infty. 
   \]
Hence  $D-\varepsilon\le D'(\mu)$. Since $D\le \dim_F(\mu)$ and $\varepsilon>0$ are arbitrary, we obtain $\dim_F(\mu)\le D'(\mu)$.  The converse inequality is proved similarly. 
\end{proof}

\begin{proof}[Proof of Theorem \ref{thm-fourier}]
Recall the definition \eqref{def-tau-c} of $\tau_c(\mu_\infty)$:
\[
\tau_c(\mu_\infty): =  \sup \Big\{\tau \in \R: \sum_{n\ge 1} | n^{\tau}  \cdot \widehat{\mu}_\infty(n)|^q<\infty \text{\, for some $q>2$} \Big\}. 
\]
Recall also the definition   \eqref{def-a-c}  $\alpha_c$ :
\[
\alpha_c =  \sup \Big\{\alpha \in \R:    \Big(\E\Big[  \Big\{ \sum_{n \ge 1}    \big| n^{\alpha} \cdot \widehat{\mu}_\infty(n)\big|^q  \Big\}^{\frac{p}{q}}\Big] \Big)^{\frac{1}{p}}  <\infty \text{\, for some $1<p<2<q<\infty$} \Big\}.
\]
Clearly, 
\[
\tau_c(\mu_\infty) \ge \alpha_c \quad a.s. 
\]
On the other hand, as a consequence of Corollary \ref{cor-series-dim},   almost surely on $\{\mu_\infty \ne 0\}$, we have  the equality 
\[
\dim_F(\mu_\infty) = \min( 2 \tau_c(\mu_\infty), 1).
\]
By Proposition \ref{prop-sq-3case}, under the assumptions of Theorem \ref{thm-fourier}, we have 
\[
\alpha_c\ge D_F/2. 
\]
Moreover, by Lemma \ref{lem-DF}, $D_F\in (0,1)$. Consequently,  almost surely on $\{\mu_\infty \ne 0\}$,
\[
\dim_F(\mu_\infty)   = \min(2 \tau_c(\mu_\infty), 1)  \ge    \min(2 \alpha_c, 1) \ge \min( D_F, 1) = D_F \quad a.s. 
\]

Finally, by Proposition \ref{prop-to-infty},  almost surely on $\{\mu_\infty \ne 0\}$, 
\[
\dim_F(\mu_\infty) \le D_F + \frac{1}{k}
\]
 for any $k\in \N$ and hence  
$
\dim_F(\mu_\infty) \le D_F. 
$
This completes the proof of the desired exact Fourier dimension formula $\dim_F(\mu_\infty) = D_F$. 
\end{proof}

\subsection{Proof of Corollary \ref{cor-salem}}
Recall that,  almost surely on $\{\mu_\infty \ne 0\}$, we have 
\begin{align}\label{def-Dh}
\dim_H(\mu_\infty)=D_H=1-\frac{\log \E[W\log W]}{\log b}.
\end{align}

 If $W$ is in the squared sub-critical or squared critical regime,  then by Theorem~\ref{thm-fourier}, on $\{\mu_\infty \ne 0\}$, almost surely, 
\[
\dim_F(\mu_\infty)=D_F=1-\frac{\log \E[W^2]}{\log b}.
\]
Since $\E[W]=1$, Jensen's inequality applied  to the probability measure $\Q:  = W\PP$ gives
\[
 \E(W\log W)= \E_\Q(\log W)\leq   \log \E_\Q(W) = \log \E[W^2 ],
\]
the equality holds if and only if $W$ is $\Q$-almost surely a constant, that is, there exists $x$ such that 
\[
1 = \Q(W= x) = \E[W \mathds{1}(W=x)] = x \PP(W=x). 
\]
It follows that $0< x \le 1$ and $\PP(W=x) =  x^{-1}$. On the other hand, since  $W\ge 0$ is non-constant  and $\E[W]=1$, we must have  $x<1$ and 
$\PP(W =0) =  1 - x^{-1}$. 
Therefore, in the squared sub-critical or squared critical  regime, on $\{\mu_\infty\ne 0\}$, almost surely, $\mu_\infty$ is  a Salem measure if and only if $W$ has a two-point distribution: 
$
\PP(W=x) = 1- \PP(W=0) = x^{-1}. 
$

Now suppose that $W$ is in the squared super-critical regime:
$
\E[W^{(2)}\log W^{(2)}]>\log b.
$
Then  by Lemma~\ref{lem-non-sub}, $W$ is in the boundary case.  Recall the Biggins-Kyprianou transform \eqref{bdd-case} of $W$: 
\[
W= \frac{e^{-\beta \xi}}{\E[e^{-\beta \xi}]}   \quad  \text{with} \quad    \frac{1}{2} < \beta< 1
\]
and $\xi$ satisfies the conditions \eqref{xi-cond} and \eqref{xi-cond-bis}. 
Recall also the definition \eqref{def-psi} for  $\psi(t) =\log \E[be^{-t\xi}]$.  By Lemma \ref{lem-psi-beta}, we have 
\[
 D_F=1 -   \inf_{1/2\le t\le 1} \frac{\log  \E[b^{1-t}W^{2t}]}{ t \log b} = \frac{ 2\psi(\beta)}{\log b}. 
\]
On the other hand,  by \eqref{def-Dh}, we have 
\begin{align*}
D_H=   1 - \frac{d}{dt}\Big|_{t = 1} \Big(  \frac{\E[W^t]}{\log b}\Big)  
 =  1- \frac{d}{dt}\Big|_{t=1}\Big( \frac{ \exp(\psi(\beta t) - t \psi(\beta) + (t-1) \log b)}{\log b}  \Big)  =   \frac{  \psi(\beta) - \beta \psi'(\beta)}{\log b}. 
 \end{align*}
 
\vspace{2mm}
{\flushleft \bf Claim A: } For $1/2 < \beta<1$, we have $\psi(\beta) < - \beta \psi'(\beta)$. 
\vspace{2mm}

Therefore, in the squared super-critical regime,  we always have $D_F< D_H$. Hence, almost surely on $\{\mu_\infty\ne 0\}$, we have  
$
\dim_F(\mu_\infty)< \dim_H(\mu_\infty)
$ and $\mu_\infty$ is not a Salem measure. 

It remains to prove Claim A.  Indeed, by Lemma \ref{lem-psi},  the function $\psi \in C^\infty$   is non-negative,  strictly convex,  strictly decreasing on $[0,1]$ and   $\psi(1) = 0$.    Since $1/2< \beta<1$, 
\begin{align*}
\frac{\psi(\beta)}{\beta} < \frac{\psi(\beta)}{1-\beta}   =  - \frac{\psi(\beta) - \psi(1)}{ \beta -1} = - \psi'(\beta_0) \quad \text{with $\beta_0\in (\beta, 1)$}.  
\end{align*}
The strict convexity of $\psi$ implies that $\psi'(\beta_0) > \psi'(\beta)$. This completes the proof of Claim A.

\subsection{Proof of Corollary \ref{cor-holder}}\label{sec-cor-holder}

Given any non-negative finite Borel measure $\nu$ on $[0,1]$,   we consider a  function   $F_\nu: [0,1]\rightarrow \R$ defined by  
\[
F_\nu(t): = \nu([0,t]) -   t \cdot \nu([0,1]). 
\]
Then, for any $n\in \Z$ with $n \ne 0$,  one has 
\begin{align}\label{F-coef}
\widehat{F}_\nu(n)=  \frac{\widehat{\nu}(n)}{ i2\pi n }. 
\end{align}

We shall use the following result  of  Boas \cite[Theorem 1]{Boa67} on H\"older continuity criterion for functions with non-negative  Fourier sine or cosine coefficients.  More precisely, consider any sequence $(\lambda_n)_{n\ge 1}$ of {\it non-negative} numbers and the formal Fourier sine and cosine series :
\[
S(\theta) = \sum_{n\ge 1} \lambda_n \sin(n 2 
\pi  \theta) \an C(\theta)= \sum_{n\ge 1} \lambda_n \cos(n 2 \pi \theta). 
\] 
Boas showed that, for any $0< \gamma< 1$, the functions $S(\theta)$ and   $C(\theta)$  belong to the class $Lip(\gamma)$ : 
\[
Lip(\gamma)= \Big\{  f: \R\rightarrow \R\Big |   \sup_{x\ne y} \frac{|f(x)-f(y)|}{|x-y|^\gamma}<\infty \Big\} 
\]
if and only if 
\[
\sum_{k\ge n}\lambda_k = O (n^{-\gamma}). 
\]
As an immediate consequence of  Boas criterion,  for a general  Fourier series with complex coefficients
\[
F(\theta) = \sum_{k \in \Z} a_k e^{i k2 \pi \theta}, 
\]
the following implication holds:  
\begin{align}\label{boas-cri}
\sum_{|k|\ge n} |a_k|= O(|n|^{-\gamma})   \Longrightarrow  F\in Lip(\gamma).  
\end{align}

Recall the definition \eqref{def-FA} of upper Frostman regularity. 
\begin{lemma}\label{lem-boas}
Let $0<\gamma<1$. 
Let $\nu$ be a non-negative finite Borel measure on $[0,1]$ whose Fourier coefficients of $\nu$ satisfies 
\begin{align}\label{nu-decay}
| \widehat{\nu}(n)| = O(n^{-\gamma}) \quad \text{as $n\to\infty$}.
\end{align}
Then $\nu$ is $\gamma$-upper Frostman regular. 
\end{lemma}

\begin{proof}
By \eqref{F-coef} and  \eqref{nu-decay},  the Fourier coefficients of   the function $F_\nu(t)= \nu([0,t]) -   t \cdot \nu([0,1])$ satisfies 
\[
|\widehat{F}_\nu (k)| =  O  (|k|^{-1 - \gamma}) \quad \text{as $|k|\to\infty$}.
\]
It follows that 
\[
\sum_{|k|\ge n} |\widehat{F}_\nu (k)|   = O(n^{-\gamma})  \quad \text{as $n\to\infty$}.
\]
Therefore, by \eqref{boas-cri}, we have $F_\nu \in Lip(\gamma)$ and hence  ($t\to t\cdot \nu([0,1])$ clearly belongs to $Lip(\gamma)$) 
\[
\nu([0,t]) = F_\nu(t) + t \cdot \nu([0,1])
\]
 belongs to $Lip(\gamma)$.  By the definition \eqref{def-FA}, the measure $\nu$ is $\gamma$-upper Frostman regular. 
\end{proof}

\begin{proof}[Proof of Corollary \ref{cor-holder}]
By Theorem \ref{thm-fourier},  for any $0< \gamma< D_F/2 < 1/2$, almost surely on $\{\mu_\infty \ne 0\}$, 
\[
\widehat{\mu}_\infty(n) = O (n^{- \gamma}) \quad \text{as $n\to\infty$}. 
\]
Since $\gamma  \in (0, 1/2)$, we may apply Lemma \ref{lem-boas} and conclude that $\mu_\infty$ is $\gamma$-upper Frostman regular. 
\end{proof}


\section{Proof of Theorem \ref{thm-pRaj}}\label{sec-con-2}

Theorem \ref{thm-pRaj} follows  from  the following general 
\begin{theorem}\label{thm-preFdim}
Let $\nu$ be a  Borel probability measure on $[0,1]$ with 
$
\dim_F(\nu)>0.
$
Assume that $W\in L^{1+\delta}(\PP)$ for a positive $\delta>0$ and 
\begin{align}\label{W-ass-Fdim}
 \frac{\E[W\log W]}{\log b} < \lim_{p\to 1^{+}} \ldim_{p}(\nu). 
 \end{align}
 Then the random  measure  $Q\nu$ is non-degenerate and   almost surely on $\{Q\nu\ne 0\}$, we have 
 \[
 \dim_F(Q\nu) > 0.
 \]
\end{theorem}

\begin{remark*}
Since the function $p\mapsto \ldim_p(\nu)$ is non-increasing, the limit $ \lim_{p\to 1^{+}} \ldim_{p}(\nu)$ indeed exists. 
\end{remark*}

Indeed, in what follows,  we obtain a lower estimate of $\dim_F(Q\nu)$ which can be described as follows. 
Given any Borel probability measure $\nu$ on $[0,1]$ and random weight $W\in L^{1+\delta}(\PP)$ for some positive $\delta>0$, let $\kappa=\dim_F(\nu)>0$. We define $\eta_W(\nu)$ to be the supremum of all non-negative real numbers $\alpha\ge 0$ such that the following system  of inequalities has a solution: 
\[
\begin{cases}
1<p< \min(2, 1 + \delta)\le     2\le q<\infty,
\vspace{2mm}
\\
\displaystyle \tau(p, q, \alpha) = p \frac{(\kappa-\alpha)q - 1}{q \kappa}>1,
\vspace{2mm}
\\
\displaystyle\frac{\log \E[W^p]}{\log b}  < (\tau(p, q, \alpha)- 1) \ldim_{\tau(p, q, \alpha)}(\nu).
\end{cases}
\]
Then we shall show that, under the assumption \eqref{W-ass-Fdim},   almost surely on $\{Q\nu\ne 0\}$, 
\begin{align}\label{low-F-nu}
\dim_F(Q\nu)\ge   2 \eta_W(\nu) > 0. 
\end{align}

For the log-normal weights, by a direct computation of the constant $\eta_W$, the lower bound \eqref{low-F-nu} implies the following 

\begin{corollary}\label{cor-nu}
Assume that $\nu$ is a Borel probability measure on $[0,1]$ with $\dim_F(\nu)>0$. Assume moreover that  $p\mapsto \ldim_{p}(\nu)$ is a constant function  on $p\in (1, 2)$ with 
\[
\ldim_{p}(\nu) =  D \quad  \text{for all $p \in (1,2)$.}
\]   Then for the   log-normal random weights $W=e^{\sigma N-\sigma^2/2}$ with $0< \sigma < \sqrt{2  D \log b}$,  the corresponding multiplicative chaos measure $Q\nu$ is non-degenerate and almost surely on $\{Q\nu\ne 0\}$, we have 
 \begin{align}\label{F-Qnu-nu}
 \dim_F(Q\nu)  \ge   \left\{
 \begin{array}{ll}
  \displaystyle  \Big(1-   \frac{\sigma^2}{D\log b}\Big) \cdot \dim_F(\nu)  & \text{if $ \displaystyle \frac{\sigma^2}{\log b} \le  \frac{D}{2}$}
  \vspace{2mm}
  \\  
    \displaystyle    2 \Big(  1  - \frac{\sigma}{\sqrt{2D \log b}} \Big)^2 \cdot \dim_F(\nu) & \text{if $ \displaystyle \frac{D}{2} <   \frac{\sigma^2}{\log b} <  2D$}
 \end{array}
 \right.. 
 \end{align}
\end{corollary}

\begin{remark*}
Comparing \eqref{log-normal-F} with \eqref{F-Qnu-nu}, by taking $\nu$ to be the Lebesgue measure,  we see that our method in this section for obtaining the lower bound of  $\dim_F(Q\nu)$  is also  sharp.  
\end{remark*}

\subsection{The initial step of the proof}\label{sec-initial-one}
We briefly explain the strategy for proving Theorem~\ref{thm-preFdim}.   Similar to the definition \eqref{eq-npq},   given any  $\nu$ satisfying the assumption of Theorem \ref{thm-preFdim}, we set 
\[
\NQ(Q\nu)  = \Big(\E\Big[\Big\{\sum_{s =  1}^\infty \big|s^{\alpha} \widehat{Q\nu}(s)\big|^q\Big\}^{p/q}\Big]\Big)^{1/p}.
\]
 Our main goal is to find a positive number $0< \alpha<\dim_F(\nu)$ such that  there exists a pair $(p, q)$ of exponents with  $1<p<2<q<\infty$ and  
\begin{align}\label{Lp-lq-finite}
\NQ(Q\nu)  <\infty. 
\end{align}
Then it follows immediately that almost surely on $\{Q\nu\ne 0\}$,  we have 
$
\dim_F(Q\nu) \ge \alpha. 
$

By Corollary \ref{cor-Qnu-fourier} below and the standard fact about uniformly integrable vector-valued martingales in a finite-dimensional Banach space, for any finite  $N\ge 1$ and any $\alpha\ge 0$, 
\[
\E\Big[\Big\{\sum_{s =  1}^N \big|s^{\alpha} \widehat{Q\nu}(s)\big|^q\Big\}^{p/q}\Big]  = \sup_{n\ge 1}  \E\Big[\Big\{\sum_{s =  1}^N \big|s^{\alpha} \widehat{Q_n\nu}(s)\big|^q\Big\}^{p/q}\Big]. 
\]
By monotone convergence theorem, one can easily see that
\begin{align}\label{sup=infty}
\big[\NQ(Q\nu)\big]^p =  \E\Big[\Big\{\sum_{s =  1}^\infty \big|s^{\alpha} \widehat{Q\nu}(s)\big|^q\Big\}^{p/q}\Big]  = \sup_{n\ge 1}  \E\Big[\Big\{\sum_{s =  1}^\infty \big|s^{\alpha} \widehat{Q_n\nu}(s)\big|^q\Big\}^{p/q}\Big] \in [0,+\infty]. 
\end{align}
Therefore,  we  can  interpret \eqref{Lp-lq-finite} as 
\begin{align}\label{goal-ineq}
 \E\Big[  \Big\| \Big( s^\alpha \widehat{Q\nu}(s)\Big)_{s\ge 1} \Big\|_{\ell^q}^p \Big]  =   \sup_{n\ge 1}\E\Big[  \big\| Q_n \mathcal{V}([0,1])\big\|_{\ell^q}^p \Big]<\infty,
\end{align}
where $\mathcal{V} = \mathcal{V}_{\alpha, q}: \mathfrak{B}^b \rightarrow \ell^q$ is the vector measure defined by 
 \begin{align}\label{def-Valpha}
\mathcal{V}(A): =   \Big( s^\alpha \int_{A}e^{-i 2\pi st}\nu(dt) \Big)_{s\ge 1} \in \ell^q \quad \text{for all $A\in \mathfrak{B}^b$}.
\end{align}
By Proposition \ref{lem-Fdecay}, 
the above vector measure $\mathcal{V}$ is well-defined for any sufficiently large $q$ with 
\[
q> \max\Big( \frac{1}{\dim_F(\nu)- \alpha}, 2\Big).
\]

Recall the definition \eqref{def-rsigma} for the growth rate function 
$H_\mathcal{V}$.   Since  for $q\ge 2$,  the Banach space $\ell^q$ has martingale type 2,  by  Proposition~\ref{prop-mart-type} and Corollary~\ref{cor-HM-comp}, the desired inequality \eqref{goal-ineq} will follow from 
\[
\varphi_W(p)+ H_\mathcal{V}(p)<0. 
\]
So  a suitable  upper estimate of $H_\mathcal{V}(p)$ will be needed and  it will be given  in Lemma \ref{lem-gen-measure} below.

\subsection{Full action by the multiplicative cascades}

The following proposition is due to Fan \cite{Fan-JMPA} (a sharper result is also established by  Barral and Jin \cite[Theorem 2.3 (a)]{JInxiong-Barral-IMRN}).  Here we give a self-contained proof using the method developped in \S \ref{sec-Mtype}. 
\begin{proposition}\label{prop-full}
Let $\nu$ be a  Borel probability measure on $[0,1]$ satisfying the condition \eqref{W-ass-Fdim}.  Assume that $W\in L^{1+\delta}(\PP)$ for a positive $\delta>0$.   Then  $Q$ acts fully on $\nu$, namely, 
\[
Q_n\nu([0,1])=  \E_n[Q\nu([0,1])] \quad a.s. \quad \text{for all $n\in \N$}.
\]
 \end{proposition}
 
 \begin{corollary}\label{cor-Qnu-fourier}
 Under the assumption of Proposition \ref{prop-full}, we have  
 \[
\widehat{ Q_n\nu}(s) =  \E_n[\widehat{Q\nu}(s)] \quad a.s. \quad \text{for all $n, s \in \N$}.
 \]
 \end{corollary}

 \begin{proof}[Proof of Proposition \ref{prop-full}]
 Take the vector measure $\mathcal{V}$ in Corollary \ref{cor-HM-comp} to be the probability measure $\nu$ on $[0,1]$ and by the equality \eqref{H=pdim}, we get 
 \begin{align*}
H_\mathcal{V}(p) +  \varphi_W(p)  =  -(p-1) \ldim_p(\nu)  + \frac{\log \E[W^p]}{\log b} \quad \text{for all $1< p< 1 + \delta$}.
 \end{align*}
 It follows that (where we use the fact $\E[W]=1$) 
 \[
\lim_{p\to1^{+}} \frac{H_\mathcal{V}(p) +  \varphi_W(p)}{p-1} = - \lim_{p\to 1^{+}} \ldim_p(\nu) +  \frac{\E[W\log W]}{\log b}<0. 
 \]
 Therefore, there exists $\varepsilon\in(0, \delta)$ such that for $p = 1 + \varepsilon$, we have  $H_\mathcal{V}(p) +  \varphi_W(p)<0$. Hence by Corollary \ref{cor-HM-comp}, the  scalar martingale  $(Q_n\nu([0,1]))_{n\ge 1}$ is $L^p$-uniformly bounded. This ends the proof of the proposition. 
 \end{proof}
 
 \begin{proof}[Proof of Corollary \ref{cor-Qnu-fourier}]
 It follows immediately from Proposition \ref{prop-full}, since for any $s\in\N$, the scalar martingale $(\widehat{Q_n\nu}(s))_{n\ge 1}$ is dominated by the non-negative martingale $
(Q_n\nu([0,1]))_{n\ge 1}$.
 \end{proof}

\subsection{An upper estimate of growth rate of vector measures}\label{sec-grv}

\begin{lemma}\label{lem-gen-measure}
Let $\nu$ be a  Borel probability measure on $[0,1]$.  Assume that there exist $\kappa\in (0,1]$ and $c>0$ such that 
\begin{align}\label{nu-poly-decay}
|\widehat{\nu}(l)|\le  \frac{c}{|l|^{\kappa}} \quad \text{for all $l\in \Z^* = \Z\setminus \{0\}$}.
\end{align}
Then for any $0< \alpha < \kappa$ and  $q>(\kappa- \alpha)^{-1}$, the growth rate function $H_{\mathcal{V}}$ of the vector measure $\mathcal{V} = \mathcal{V}_{\alpha, q}$ defined in \eqref{def-Valpha} has the upper estimate: 
\begin{align}\label{HV-upper}
H_{\mathcal{V}}(p) \le         p-1 -  (\tau(\alpha, p, q) - 1) \ldim_{\tau(\alpha, p, q)}(\nu)  \quad \text{for all $p> \frac{q\kappa}{q(\kappa - \alpha) -1}$},
\end{align}
where
\[
\tau(\alpha,p, q) = p \frac{q(\kappa-\alpha) - 1}{q \kappa}>1.
\]
\end{lemma}

\begin{remark*}
For  the Lebesgue measure on the interval $[0,1]$,  the inequality \eqref{HV-upper} becomes an equality. See  Proposition \ref{lem-grf} below in the Appendix of this paper. 
\end{remark*}

 Lemma \ref{lem-lq-dim} below follows immediately  from the definition  \eqref{def-D-p} of $\ldim_p(\nu)$. 
\begin{lemma}\label{lem-lq-dim}
Let $\nu$ be a  Borel probability measure on $[0,1]$ and $p>1$. Then for any $\delta> 0$,   there exists $n_\delta$ such that 
\[
\sum_{|u|=n}  \nu(I_u)^p \le   b^{-n (p-1) (\ldim_p(\nu) - \delta)}  \,\,  \text{for all $n\ge n_\delta$.}
\]
\end{lemma}

\begin{proof}[Proof of Lemma \ref{lem-gen-measure}]
Define 
\[
S_0 : = \Big\{  (\alpha, p,q) \Big|  0< \alpha < \kappa, q>\frac{1}{\kappa-\alpha} \an p> \frac{q\kappa}{q(\kappa- \alpha) -1} \Big\}.
\]
We want to show that for all triplets $(\alpha, p, q)\in S_0$, 
\begin{align}\label{goal-HV}
H_\mathcal{V}(p)  \le     p-1 -  (\tau - 1) \ldim_{\tau}(\nu) \, \,
\text{with $\tau = \tau(\alpha, p, q) = p\frac{ q (\kappa - \alpha)  -1}{q \kappa}>1$.}
\end{align}

For any $\varepsilon\in (0, \kappa)$, define $\kappa_\varepsilon= \kappa- \varepsilon$ and set 
\[
S_\varepsilon: = \Big\{  (\alpha, p,q) \Big|  0< \alpha < \kappa_\varepsilon, q>\frac{1}{\kappa_\varepsilon-\alpha} \an p> \frac{q\kappa_\varepsilon}{q(\kappa_\varepsilon- \alpha) -1} \Big\}.
\]
It can be directly shown that if $\kappa> \varepsilon_1 >\varepsilon_2>0$, then 
\begin{align}\label{S-monotone}
S_{\varepsilon_1} \subset S_{\varepsilon_2} \an 
\bigcup_{0< \varepsilon< \kappa} S_\varepsilon = S_0.
\end{align}

Now fix any triplet $(\alpha, p, q)\in S_0$.  By the monotonicity \eqref{S-monotone},  there exists $\varepsilon_0\in (0, \kappa)$ such that for all $0< \varepsilon<\varepsilon_0<\kappa$,  this fixed triplet $(\alpha, p, q)$ belongs to $S_\varepsilon$.   Therefore, by  the continuity of  the function $\tau\mapsto (\tau-1)\ldim_\tau(\nu)$ on $(1, \infty)$, to prove the inequality \eqref{goal-HV} for this triplet $(\alpha, p, q)$, it suffices to prove that,   for any $\varepsilon \in (0, \varepsilon_0)$, 
\begin{align}\label{goal-HV-bis}
H_\mathcal{V}(p)  \le     p-1 -  (\tau_\varepsilon - 1) \ldim_{\tau_\varepsilon}(\nu) \, \,
\text{with $\tau_\varepsilon = p\frac{ q (\kappa_\varepsilon - \alpha)  -1}{q \kappa_\varepsilon}>1$.}
\end{align}

So we  take any  $\varepsilon\in (0, \varepsilon_0)$. For simplifying notation, from now on, we  write 
\[
 \kappa' = \kappa_\varepsilon= \kappa - \varepsilon>0 \an \tau'= \tau_\varepsilon = p\frac{ q (\kappa' - \alpha)  -1}{q \kappa'}> 1. 
\] 
By Proposition~\ref{lem-Fdecay},   we may  assume  that there exists  $c>0$ such that 
\begin{align*}
\sup_{u\in \Ab^*} \Big|\int_{I_u}e^{-i 2\pi st}\nu(dt)\Big| \le \frac{c}{s^{\kappa'}} \quad \text{for all $s>1$}. 
\end{align*}
In particular,  for any $u\in \Ab^*$,  taking $C = \max\{c, 1\}$, we have 
\begin{align}\label{b-adic-ass}
\Big|\int_{I_u}e^{-i 2\pi st}\nu(dt)\Big| \le  C \min\big\{ s^{-\kappa'},\, \nu(I_u) \big\}.
\end{align}
Recall the definition \eqref{def-H-mart} of $T_n^\mathcal{V}$ and \eqref{def-rsigma} of $H_\mathcal{V}(p)$.  Here we have 
\begin{align*}
 T_n^{\mathcal{V}} =   b^n \sum_{|u|=n}      \Big(    s^\alpha \int_{I_u}  e^{-i 2 \pi s t}\nu(dt)\Big)_{s \ge 1}  \indi_{I_u}.
\end{align*}
Therefore,  by \eqref{b-adic-ass}, 
\begin{align*}
\|T_n^{\mathcal{V}}\|_{L^p(\ell^q)}^p & \le  b^{n(p-1)}  C^p  \sum_{|u|=n}     \Big( \sum_{s=1}^\infty   s^{\alpha q}    \min\big\{ s^{-\kappa' q},\, \nu(I_u)^q \big\} \Big)^{p/q}. 
\end{align*}
Then, by using the following decomposition 
\begin{align*}
\sum_{s=1}^\infty   =\sum_{1\le s\le a(u)-1}  + \sum_{ a(u)-1<s<a(u)+1} + \sum_{s\ge a(u)+1}\quad \text{with $a(u)= \nu(I_u)^{-1/\kappa'}$}
\end{align*}
and the elementary inequality $(x+y +z)^t\le  3^t (x^t + y^t +z^t)$ for all $x, y, z, t \ge 0$, we obtain 
\begin{align*}
\|T_n^{\mathcal{V}}\|_{L^p(\ell^q)}^p & \le  b^{n(p-1)}  C^p  3^{p/q}  \sum_{|u|=n}     \Big[ \Big(   \sum_{1\le s\le a(u) }    s^{\alpha q}    \cdot  \nu(I_u)^q  \Big)^{p/q} +   \Big(\sum_{s> a(u)}    s^{-(\kappa'-\alpha)q}\Big)^{p/q} 
\\
& \qquad\qquad  \qquad \qquad   + \Big(\sum_{ a(u)-1<s<a(u)+1}  s^{\alpha q}    \min\big\{ s^{-\kappa' q},\, \nu(I_u)^q \big\} \Big)^{p/q}\Big]. 
\end{align*}
Clearly, we have 
\begin{align*}
\sum_{1\le s\le a(u) -1 }    s^{\alpha q} \cdot  \nu(I_u)^q  \le  \nu(I_u)^q  \cdot \int_1^{a(u)}  t^{\alpha q} dt  \le \nu(I_u)^{q-\frac{\alpha q +1}{\kappa'}}
\end{align*}
and 
\begin{align*}
\sum_{s\ge  a(u)+1}    s^{-(\kappa'-\alpha)q} \le \int_{a(u)}^\infty \frac{dt}{t^{(\kappa' - \alpha) q}} \le  \frac{1}{(\kappa' - \alpha) q -1}  \nu(I_u)^{q-\frac{\alpha q +1}{\kappa'}}.
\end{align*}
Moreover,  we have   (in each of the summations $A$ and $B$ below,   there is at most one term)
\begin{align*}
A: =& \sum_{ a(u)-1<s\le a(u)}   s^{\alpha q}    \min\big\{ s^{-\kappa' q},\, \nu(I_u)^q \big\}  \le   \nu(I_u)^q   \sum_{ a(u)-1<s\le a(u)}   s^{\alpha q}   \le  \nu(I_u)^q  \cdot a(u)^{\alpha q}
\end{align*}
and  (here we use the assumption $\alpha <\kappa'$)
\[
B: = \sum_{ a(u)\le s<a(u)+1}  s^{\alpha q}    \min\big\{ s^{-\kappa' q},\, \nu(I_u)^q \big\}  \le \sum_{ a(u)\le s<a(u)+1}  s^{(\alpha-\kappa') q} \le a(u)^{(\alpha-\kappa')q}. 
\]
Hence by noting that $a(u)= \nu(I_u)^{-1/\kappa'}$ and $\nu(I_u)\le 1$, we obtain 
\begin{align*}
 & \sum_{ a(u)-1<s<a(u)+1}  s^{\alpha q}    \min\big\{ s^{-\kappa' q},\, \nu(I_u)^q \big\}  
   =    A + B
  \\
  &\le   a(u)^{\alpha q} \nu(I_u)^q + a(u)^{(\alpha-\kappa')q}   = \nu(I_u)^{- \frac{\alpha q}{\kappa'} +q} + \nu(I_u)^{-\frac{(\alpha-\kappa')q}{\kappa'}}
\\
&  =  2\nu(I_u)^{q- \frac{\alpha q}{\kappa'}} \le  2\nu(I_u)^{q- \frac{\alpha q +1}{\kappa'}}. 
\end{align*}
Consequently,  there exists a constant $c(\kappa', \alpha, p, q)>0$ such that 
\begin{align*}
\|T_n^{\mathcal{V}}\|_{L^p(\ell^q)}^p & \le  c(\kappa', \alpha, p, q)   \cdot  b^{n(p-1)}  \sum_{|u|=n}  \nu(I_u)^{p \frac{ q(\kappa'-\alpha) -1}{q \kappa'}}
\\ 
& =  c(\kappa', \alpha, p, q)   \cdot  b^{n(p-1)}  \sum_{|u|=n}  \nu(I_u)^{\tau'}. 
\end{align*}
By Lemma \ref{lem-lq-dim}, for any small enough $\delta>0$, there exists $n(\delta, \tau') \in \N$ such that 
\[
 \sum_{|u|=n}  \nu(I_u)^{\tau'}   \le  b^{-n (\tau'-1) (\ldim_{\tau'}(\nu) - \delta)}  \,\,  \text{for all $n\ge n(\delta, \tau')$.}
\]
Therefore, for all $n\ge n(\delta, \tau')$,
\begin{align*}
\|T_n^{\mathcal{V}}\|_{L^p(\ell^q)}^p & \le  c(\kappa', \alpha, p, q)   \cdot  b^{n(p-1)}  \sum_{|u|=n}  \nu(I_u)^{\tau'}
\\
&\le   c(\kappa', \alpha, p, q)   \cdot  b^{n(p-1)} \cdot  b^{-n (\tau'-1) (\ldim_{\tau'}(\nu) - \delta)}. 
\end{align*}
It follows that 
\[
H_\mathcal{V}(p) = \limsup_{n\to\infty}  \frac{ \log \|T_n^{\mathcal{V}}\|_{L^p(\ell^q)}^p }{n\log b}   \le     p-1 - (\tau'-1) ( \ldim_{\tau'}(\nu) - \delta). 
\]
By letting $\delta \to 0^{+}$, we obtain  the desired inequality \eqref{goal-HV-bis} and  end the  whole proof. 
\end{proof}

\subsection{The proof of Theorem \ref{thm-preFdim}}
By assumption, let $\kappa = \dim_F(\nu)>0$. Recall also that $W\in L^{1+\delta}(\PP)$ for some $\delta>0$.    By Corollary~\ref{cor-HM-comp}, it suffices to find  a triplet $(\alpha, p, q)$ with  the following properties: 
\begin{itemize}
\item 
$
0< \alpha< \dim_F(\nu) = \kappa;
$
\item $1< p<1 + \delta$; 
\item $q>\max(\frac{1}{\kappa - \alpha}, 2)$; 
\item The vector measure $\mathcal{V}: \mathfrak{B}^b([0,1]) \rightarrow \ell^q$ defined by
\[
\mathcal{V}(A): =   \Big( s^\alpha \int_{A}e^{-i 2\pi st}\nu(dt) \Big)_{s\ge 1} \in \ell^q \quad \text{for all $A\in \mathfrak{B}^b([0,1])$}
\]
has a growth rate function $H_\mathcal{V}$ with 
\[
\varphi_W(p) + H_\mathcal{V}(p)<0. 
\]
\end{itemize}

We now show the existence of  a triplet $(\alpha, p, q)$ satisfying  all the above  required properties.  First, fix any two constants $0< \alpha_0<\kappa \le 1$ (sufficiently small) and $q_0\ge \frac{2}{\kappa- \alpha_0}$ (sufficiently large)  such that 
\[
\frac{\kappa}{\kappa-\alpha_0} < 1 +\frac{\delta}{2} \an \frac{q_0\kappa}{q_0(\kappa-\alpha_0) -1}< 1+\frac{\delta}{2}. 
\]
Then for all $0< \alpha< \alpha_0$ and $q\ge q_0$, 
\[
\frac{\kappa}{\kappa-\alpha}  < \frac{\kappa}{\kappa-\alpha_0} < 1 + \frac{\delta}{2}
\]
and 
\[
   \frac{q\kappa}{q(\kappa-\alpha) -1}  = \frac{\kappa}{\kappa - \alpha - \frac{1}{q}} \le \frac{\kappa}{\kappa - \alpha_0 - \frac{1}{q_0}}   =   \frac{q_0\kappa}{q_0(\kappa-\alpha_0) -1} < 1 + \frac{\delta}{2}.
\]
In particular, if $0< \alpha< \alpha_0$ and $q\ge q_0$,  then 
\[
q(\kappa - \alpha) - 1 \ge q_0(\kappa- \alpha_0)-1 \ge 1 \an 
\frac{q\kappa}{q(\kappa-\alpha)-1}>1.
\]

Now define a set 
\[
S : = \Big\{(\alpha, p, q)\Big|  0< \alpha< \alpha_0,  q_0 < q < \infty \an  \frac{q\kappa}{q (\kappa-\alpha)-1}<p < 1+\frac{\delta}{2}\Big\}. 
\]
For any triplet $(\alpha, p, q)\in S$,  set 
\[
\tau(\alpha, p, q): = p \frac{q(\kappa- \alpha)-1}{q\kappa}>1
\]
and
\begin{align*}
  g(\alpha, p, q): =   \frac{\log \E[W^p]}{\log b} - (\tau(\alpha, p, q)- 1) \ldim_{\tau(\alpha, p, q)}(\nu).
\end{align*}

Note that, for any $1<p < 1 + \frac{\delta}{2}$, the triplet $(\alpha, p, q)$ belongs to the set $S$  if and only if 
\[
\alpha + \frac{1}{q} < \kappa \Big( 1 - \frac{1}{p}\Big), \quad 0< \alpha< \alpha_0 \an q  > q_0. 
\]
Thus,  by defining $\alpha_p$ and $q_p$ as 
\[
\alpha_p: = \min \Big\{\frac{\kappa}{3} \Big( 1 - \frac{1}{p}\Big), \alpha_0\Big\} \an \frac{1}{q_p} : = \min \Big\{\frac{\kappa}{3} \Big( 1 - \frac{1}{p}\Big), \frac{1}{q_0}\Big\},
\]
we have 
\[
S_p : = \Big\{ (\alpha, p, q) \Big| 0<\alpha <\alpha_p \an   q_p<q <\infty  \Big\} \subset S. 
\]
Therefore, by Lemma \ref{lem-gen-measure}, if $1< p < 1+ \frac{\delta}{2}$, then for all $0< \alpha< \alpha_p$ and $q_p < q<\infty$, we have 
\[
H_{\mathcal{V}}(p) \le         p-1 -  (\tau(\alpha, p, q)- 1) \ldim_{\tau(\alpha, p, q)}(\nu)
\]
and hence 
\begin{align}\label{H+phi}
H_\mathcal{V}(p)  + \varphi_W(p) \le     g(\alpha, p, q).
\end{align}

Using the continuity  of $q\mapsto \tau(p, q, \alpha)$ and $\tau\mapsto (\tau- 1) \ldim_{\tau}(\nu)$, we have 
\[
\lim_{\alpha\to 0^{+}}\Big( \lim_{q\to +\infty} g(\alpha, p, q) \Big) =  h(p) : = \frac{\log \E[W^p]}{\log b} - (p- 1)  \ldim_{p}(\nu). 
\]
Since $\ldim_p(\nu)$ is non-increasing and $W\in L^{1+\delta}(\PP)$ for some positive $\delta>0$,  by the assumption \eqref{W-ass-Fdim}, we have 
\[
\lim_{p\to 1^{+}} \frac{h(p)}{p-1} = \frac{\E[W\log W]}{\log b} - \lim_{p\to 1^{+}} \ldim_{p}(\nu)<0. 
\]
It follows that there exists $1< p^*< 1+\frac{\delta}{2}$ such that 
\[
h(p^*)= \lim_{\alpha\to 0^{+}}\Big( \lim_{q\to +\infty} g(\alpha, p^*, q) \Big)<0. 
\]
Consequently, there exists a pair  $(\alpha^*, q^*)$ with $0< \alpha^*< \alpha_p$ and $q^*>q_p$, such that 
\[
g(\alpha^*, p^*, q^*)<0. 
\]
By \eqref{H+phi},  the triplet $(\alpha^*, p^*, q^*)$ satisfies all the conditions required.

\appendix

\section{}

\subsection{Second moment of the Fourier coefficients } \label{S-second moment}

\begin{lemma}
If the initial random weight $W$ satisfies $\E[W^2]<b$, then 
\begin{align}\label{eqn-rela-sd}
\E[|\widehat{\mu}_\infty(k)|^2]\asymp k^{-\big(1-\frac{\log \E[W^2]}{\log b}\big)} \quad \text{as $k\to\infty$,}
\end{align}
where the constants in the asymptotic relation depend only on $b$ and $\E[W^2]$.
In particular, if $W$ is in the squared sub-critical or squared critical regime, then  $\E[|\widehat{\mu}_\infty(k)|^2]\asymp k^{-D_F}$ as $k\to\infty$.
\end{lemma}
\begin{proof}
Note that  from the proof of  Lemma~\ref{lemma-jihi}, if $\E[W^2]<b$, then 
for any $k \in\N$, one has
\begin{align}\label{eqn-esdt}
\E[|\widehat{\mu}_\infty(k)|^2]=\frac{\mathrm{Var}(W)}{\E[W^2]}\sum_{n=1}^\infty \big(b\E[W^2]\big)^{n}\Big|\frac{e^{i2\pi k b^{-n}}-1}{2\pi k}\Big|^2.
\end{align}
Denote by  $T(k,n):=\frac{\big|e^{i2\pi  k b^{-n}}-1\big|^2}{4\pi k^2}$ and use $|e^{i2\pi x}-1|\asymp |x|$ for $|x|\leq 1$,
one obtains   
\begin{align*}
\sum_{n=1}^\infty(b\E[W^2])^n\,T(k,n)
=& \sum_{1\leq n\leq \lceil \log_b k \rceil}(b\E[W^2])^n\,T(k,n)+\sum_{n\geq \lceil \log_b k \rceil +1} (b\E[W^2])^n\,T(k,n)\\
\leq &\frac{1}{2\pi k ^2}\sum_{1\leq n\leq \lceil \log_b k \rceil}(b\E[W^2])^n+C_1 \sum_{n\geq \lceil \log_b k \rceil +1} \big(\frac{\E[W^2]}{b}\big)^n\\
\leq& C_2 k^{-2}(b\E[W^2])^{\lceil \log_b k \rceil}+C_3 \big(\frac{\E[W^2]}{b}\big)^{\lceil \log_b k \rceil +1}\\
\leq&  C_2 k ^{-2}(b\E[W^2])^{\log_b k }+C_3 \big(\frac{\E[W^2]}{b}\big)^{\log_b k }\\
=&  C \big(\frac{\E[W^2]}{b}\big)^{\log_b k } = C k^{-\big(1-\frac{\log \E[W^2]}{\log b}\big)},
\end{align*}
where   $C_1$ is an  absolute constant and $C_2,C_3, C$ are constants depending only on $b$ and $\E[W^2]$.  

On the other hand, we have
\begin{align*}
&\sum_{n=1}^\infty(b\E[W^2])^n\,T(k,n)
\geq  \sum_{n\geq \lceil \log_b k \rceil } \big(\frac{\E[W^2]}{b}\big)^n\geq C_4 \big(\frac{\E[W^2]}{b}\big)^{\log_b k }=C_4 k^{-\big(1-\frac{\log \E[W^2]}{\log b}\big)},
\end{align*}
where $C_4$ is a constant depending only on $\E[W^2]$ and $b$. Therefore we get the desired relation \eqref{eqn-rela-sd}. 
\end{proof}
\begin{remark*}
Indeed, by \eqref{eqn-esdt}, we have
\[
\E[|\widehat{\mu}_\infty(b^n)|^2]=\big(\frac{\E[W^2]}{b}\big)^n\E[|\widehat{\mu}_\infty(1)|^2],
\]
which also implies the asymptotic relation \eqref{eqn-rela-sd} along the subsequence of the $b$-adic integers $b^n$.
\end{remark*}

\subsection{A precise growth rate function}

\begin{proposition}\label{lem-grf}
Let  $\alpha \in [0,1)$ and  
$
q>   (1-\alpha)^{-1}.
$  Let  $\mathcal{V}: \mathfrak{B}^b \rightarrow \ell^q$ be the vector measure:
\[
\mathcal{V}(A) =   \Big( s^\alpha \int_{A}e^{-i 2\pi st} dt \Big)_{s\ge 1} \in \ell^q \quad \text{for all $A\in \mathfrak{B}^b([0,1])$}.
\]
Then the growth rate function $H_\mathcal{V}$ defined in \eqref{def-rsigma} has the following exact form: 
\[
H_{\mathcal{V}}(p) =    p (  \alpha + q^{-1}). 
\]
\end{proposition}

\begin{proof}
Recall the definition \eqref{def-H-mart} of $T_n^\mathcal{V}$. Here we have 
\[
T_n : =   b^n \sum_{|u|=n}      \Big(  s^\alpha   \int_{I_u}  e^{-i 2 \pi s t}dt\Big)_{s \ge 1}  \indi_{I_u}.
\]
By \eqref{V-exp}, we have 
\[
\|T_n\|^p_{L^p(\ell^q)} = \frac{b^{np}}{2^p \pi^p} \Big( \underbrace{ \sum_{s=1}^\infty     s^{-(1-\alpha)q}  | e^{i 2 \pi sb^{-n}} - 1 |^q}_{\text{denoted $\Sigma_n$}}\Big)^{p/q}. 
\]
By regrouping all $s$ according to their congruent classes in $\Z/b^n\Z$, we have 
\begin{align*}
\Sigma_n &= \sum_{k=1}^{b^n}     \sum_{l=0}^\infty (k+b^n l )^{-(1-\alpha)q} \cdot | e^{i 2 \pi (k+b^nl)b^{-n}} - 1 |^q 
\\
& = \sum_{k=1}^{b^n} | e^{i 2 \pi k b^{-n}} - 1 |^q  \sum_{l=0}^\infty (k+b^n l)^{-(1-\alpha)q}
\\
&= \underbrace{\sum_{k=1}^{b^n} | e^{i 2 \pi k b^{-n}} - 1 |^q  k^{-(1-\alpha)q} }_{\text{denoted $\Sigma_n'$}}  + \underbrace{ \sum_{k=1}^{b^n} | e^{i 2 \pi k b^{-n}} - 1 |^q  \sum_{l=1}^\infty (k+b^n l)^{-(1-\alpha)q}}_{\text{denoted $\Sigma_n''$}}.
\end{align*}
By using the limit of Riemann sum of a continuous function (here the assumption $\alpha\in [0,1)$ is used), we have 
\begin{align*}
\lim_{n\to\infty} b^{n [ -1+ (1-\alpha)q]}   \Sigma_n' 
= \int_0^1   \Big|\frac{ e^{i 2 \pi t } - 1}{t^{1-\alpha}} \Big|^q dt  = :c_1(\alpha, q)\in (0,\infty).
\end{align*}
For the term $\Sigma_n''$, note that by assumption,  we have $q(1-\alpha)>1$ and 
\begin{align*}
  \Sigma_n''& \le   \sum_{k=1}^{b^n} | e^{i 2 \pi k b^{-n}} - 1 |^q  \sum_{l=1}^\infty (b^n l)^{-(1-\alpha)q} =  \sum_{k=1}^{b^n} | e^{i 2 \pi k b^{-n}} - 1 |^q    b^{-n(1-\alpha)q} \sum_{l=1}^\infty l^{-(1-\alpha)q}.
\end{align*}
Hence, 
\begin{align*}
\limsup_{n\to\infty} b^{n [ -1+ (1-\alpha)q]}  \Sigma_n'' \le     \int_0^1 | e^{i 2 \pi t} - 1 |^q    dt \cdot  \sum_{l=1}^\infty l^{-(1-\alpha)q} =: c_2(\alpha, q)<\infty. 
\end{align*}
It follows that there exist two constants $c_1'(\alpha,q), c_2'(\alpha,q)>0$ such that   
\[
 b^{n [ 1 -(1-\alpha)q]}   c_1'(\alpha, q)\le \Sigma_n \le  b^{n [ 1 -(1-\alpha)q]}   c_2'(\alpha, q) \quad \text{for all $n\ge 1$}.
\]
Hence by the definition \eqref{def-rsigma}, we obtain  the desired equality.
\end{proof}

\subsection{The uniqueness of Biggins-Kyprianou transform}\label{sec-xi}

Suppose that  two pairs $(\beta, \xi)$ and $(\beta', \xi')$ satisfy the conditions \eqref{xi-cond} and  \eqref{xi-cond-bis} such that 
\[
\frac{e^{-\beta \xi}}{\E[e^{-\beta \xi}]}  = \frac{e^{-\beta' \xi'}}{\E[e^{-\beta' \xi'}]}. 
\]
Then there exist $x, y\in \R$ with $x>0$ such that 
\[
\xi' = x\xi + y.
\]  
By switching $\xi$ and $\xi'$ if needed, we may assume that $x\in (0,1]$. If $x=1$, then $\beta  = \beta'$ and  it is easy to see that $y = 0$, which implies $\xi= \xi'$. Now assume by contradiction that $x\in (0,1)$. Using the argument in the proof  of Lemma~\ref{lem-psi} in \S \ref{sec-BK-transformation} (in particular, here we shall use the condition \eqref{xi-cond-bis} for $\xi$),  we know that the function $\psi_\xi$ defined in \eqref{def-psi} is $C^\infty$ and strictly convex on $(0,1)$. Hence, for all $t\in (0,1)$, 
\[
\psi_\xi'(t) = -\frac{\E[\xi e^{-t\xi}]}{\E[e^{-t\xi}]}. 
\] 
Thus, by using the \eqref{xi-cond}, we have 
\begin{align*}
0= \E[\xi' e^{-\xi'}]  =  \E[(x\xi + y) e^{-x \xi -y}]  = - x \psi_\xi'(x)   \E[e^{-x\xi}] e^{-y} + \frac{y}{b}
 =&  - x \frac{\psi_\xi'(x)}{b} + \frac{y}{b} . 
\end{align*}
Hence $y = - x \psi_\xi'(x)$. By
$
1 = \E[be^{-x\xi -y}]  = \exp(\psi_\xi(x) -y), 
$
we have  $y= \psi_\xi(x)$. Therefore $\psi_\xi(x)/x= \psi_\xi'(x)$. But this is impossible since by Lemma \ref{lem-psi}, for $x\in (0,1)$, we have $\psi_\xi(x)/x >0 > \psi_\xi'(x)$.

\subsection{Proof of Lemma \ref{lem-psi}}\label{sec-psi-fn}

The first and second assertions are easy.    The third assertion follows immediately from the fourth one. So let us only verify the last two assertions (the reader is also referred to \cite[ArXiv version, Lemma A.1]{Jaffuel}).  

A direct computation yields
\[
\psi''(t)=\E[\xi^2  w_t(\xi)]-(\E[\xi w_t(\xi)])^2 \quad \text{with $w_t(\xi) = \frac{e^{-t\xi}}{\E[e^{-t\xi}]}$}.
\]
Hence by Cauchy-Schwarz inequality, we have $\psi''(t)\ge 0$ and the equality $\psi''(t)=0$ holds if and only if $\E[(\xi-   \E[\xi w_t(\xi)])^2 w_t(\xi)]  =0$. In other words, $\psi''(t) =0$ if and only if  the random variable  $\xi$ takes values in $\{\E[\xi w_t(\xi)], +\infty\}$.  However, if $\xi$ takes values in a set $\{x, +\infty\}$ with $x\in \R$, then by \eqref{xi-cond}, we must have $x= 0$ and therefore $\E[be^{-\xi}] =  b \PP(\xi = 0)  = 1$.  It follows that $\xi$ and $W$ are given in the Example~\ref{ex-bad-boy}, which violates the assumption that $W$ is in the boundary case.   Hence $\psi''(t)>0$ for all $t\in (0, 1+\delta)$. That is, $\psi$ is strictly convex. 

\begin{figure}[H]
\begin{tikzpicture}[>=stealth, scale=1.5]
 \draw[->](0,0)--(2.5,0)node[below]{$t$};
 \draw[->](0,0)--(0,2)node[left]{$y$};
 \node[scale=0.8] at(-0.2,-0.2){$0$};
 \draw [domain=0:1.6] plot(\x,{1.61*(\x-1)^2});
  \draw[domain=0:1.6,blue] plot(\x,{0.5796/1.6*\x});
  \draw[domain=2.5116/1.932:1.6,red] plot(\x,{1.932*\x-2.5116});
  \fill (1,0) circle (0.5pt);
  \node[scale=0.8] at (1,-0.2) {$1$};
    \fill (1.6,0) circle (0.5pt);
  \node[scale=0.8] at (1.6,-0.2) {$\beta$};
  \draw[dashed](1.6,0)--(1.6,0.5769);
    \fill (1.6,0.5769) circle (0.5pt);
     \draw[dashed](0,0.5769)--(1.6,0.5769);
      \node[scale=0.8] at (-0.3,0.5769) {$\psi(\beta)$};
      \node[scale=0.8] at (0.85,0.9){$ y=\psi(t)$};
       \fill[] (0,1.61) circle (0.5pt);
         \node[scale=0.8] at (-0.3,1.61) {$\log b$};
\end{tikzpicture}\caption{Illustration of $\frac{\psi(\beta)}{\beta}>\psi'(\beta)$.}\label{psi-graph}
\end{figure}
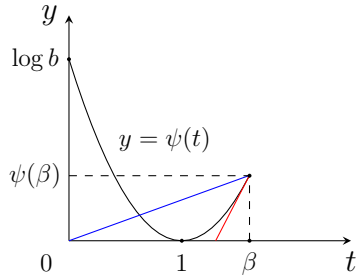

Finally,  in the boundary case, the condition $\E[W\log W]<\log b$ is equivalent to 
\begin{align}\label{beta-strictless-1}
\frac{\psi(\beta)}{\beta}> \psi'(\beta).
\end{align}
  By  the strict convexity of $\psi$,  if $\beta\in (0, 1)$, then $\psi'(\beta)< 0$, hence \eqref{beta-strictless-1} holds. For $\beta=1$, we have $\psi(1)= \psi'(1) = 0$.  So for proving the first equivalence in  \eqref{beta-less-1}, it suffices to show the inequality $\psi(\beta)/\beta \le \psi'(\beta)$ for any $\beta > 1$. But this inequality follows immediately from the convexity of $\psi$ (see Figure \ref{psi-graph} for the illustration).  The other two equivalences can be similarly obtained.

\subsection{Conditional Linderberg-Feller central limit theorem}
\begin{proposition}\label{Prop-c-clt}
Suppose that $(\Omega,\mathcal{F}, (\mathcal{F}_n)_{n\in\N};\mathbb{P})$ is  a filtered probability space. For any fixed   $n\in \N$, there are two  sequences of random variables   $\{X_{n,k}: 1\leq k \leq t_n\}$  and  $\{Y_{n,k}: 1\leq k \leq t_n\}$ satisfying the following conditions:
\begin{itemize}
\item[(1)] $\lim_{n\rightarrow\infty} t_n=\infty$;
\item[(2)]  The family of real-valued random variables $\{X_{n,k}: 1\leq k\leq t_n\}$ are $\mathcal{F}_n$-measurable;
\item[(3)] The family of complex-valued random variables $\{Y_{n,k}: 1\leq k\leq t_n\}$ are independent of  $\mathcal{F}_n$; 
\item[(4)]  $\{Y_{n,k}: 1\leq k\leq t_n\}$ are i.i.d. copies of a complex random variable $Y$ with 
\[
\mathbb{E}[Y]=0,\quad \mathbb{E}[Y^2]=0,\quad \text{and}\quad \mathbb{E}[|Y|^2]=c\in (0,\infty);
\] 
\item[(5)]
\[
\mathbb{E}\Big[\Big|\sum_{k=1}^{t_n} X_{n,k}Y_{n,k}\Big|^2\Big]=c>0
\]
\item[(6)] There exists a  sequence of non-negative random variables $(M_n)_{n\in\N}$ such that 
\begin{align}\label{eqn-convergence}
\mathbb{E}\Big[\Big|\sum_{i=1}^{t_n} X_{n,k}Y_{n,k}\Big|^2\big|\mathcal{F}_n\Big]=c\sum_{i=1}^{t_n} X_{n,k}^2=cM_n
\end{align}
and
\[
\mathbb{E}\Big[\Big(\sum_{i=1}^{t_n} X_{n,k}Y_{n,k}\Big)^2\big| \mathcal{F}_n\Big]=0.
\]
In other words, $\mathbb{E}[M_n]=1$.
\item[(7)] $\lim_{n\rightarrow \infty}M_n=M_\infty$ almost surely (in probability), where $M_\infty\in \mathcal{F}_\infty$.
\item[(8)] For any $\varepsilon>0$, 
\begin{align}\label{eqn-hjhi}
\lim_{n\rightarrow\infty} \mathbb{E}\Big[\sum_{k=1}^{t_n} |X_{n,k}|^2|Y_{n,k}|^2\mathds{1}(|X_{n,k}Y_{n,k}|>\varepsilon) \big|  \mathcal{F}_n\Big]=0 \quad  \text{a.s. (in probability)}
\end{align}
\end{itemize}
Then for any bounded continuous function $\phi$,
\[
\lim_{n\rightarrow\infty}\mathbb{E}\Big[\phi\Big(\sum_{k=1}^{t_n} X_{n,k}Y_{n,k}\Big) \big|\mathcal{F}_n\Big]=\mathbb{E}\left[\phi(\mathcal{N}_\C(0, M_{\infty}))| \mathcal{F}_\infty\right]\quad \text{a.s. (in probability)}
\]
where $\mathcal{N}_\C(0, M_{\infty})$ is a complex normal random variable with mean zero and covariance matrix given by 
\[
M_\infty\begin{pmatrix}
1/2&0\\
0&1/2\\
\end{pmatrix}.
\]
\end{proposition}

\begin{proof}
We shall use the same strategy as the proof of  \cite[Theorem 3.4.10]{Rick-Durrett}. Here we give the proof of the case of almost sure convergence in $(7)$ and $(8)$. The case of   convergence in probability is proved similarly. 

We only need to prove that the  conditioned distribution of $Z_n:=\sum_{k=1}^{t_n} X_{n,k}Y_{n,k}$ given $\mathcal{F}_n$ converges to the conditioned distribution of $\mathcal{N}_\C(0, M_\infty)$ given $\mathcal{F}_\infty$ almost surely. In other words, let $Z_n=\xi_n+i \zeta_n$, we only need to show for any $x,y, t\in\mathbb{R}$: 
\begin{align}\label{eqn-desired}
\lim_{n\rightarrow\infty}\mathbb{E}\big[\exp\{it(x\xi_n+y\zeta_n)\}\big| \mathcal{F}_n\big]=\exp\Big(-\frac{1}{2}\frac{M_\infty}{2}t^2(x^2+y^2)\Big) \quad a.s.
\end{align}
Now 
\begin{align*}
\mathbb{E}\big[\exp\{it(x\xi_n+y\zeta_n)\}\big| \mathcal{F}_n\big]=&\mathbb{E}\Big[\exp\Big\{it\sum_{k=1}^{t_n}X_{n,k}\big(x\mathrm{Re}(Y_{n,k})+y\mathrm{Im}(Y_{n,k})\big)\Big\}\big| \mathcal{F}_n\Big]\\
=&\prod_{k=1}^{t_n}\mathbb{E} \Big[\exp \Big\{itX_{n,k}(x\mathrm{Re}(Y_{n,k})+y\mathrm{Im}(Y_{n,k})) \Big\} \big| \mathcal{F}_n\Big]\\
=&\prod_{k=1}^{t_n}\mathbb{E}[\exp( itA_{n,k}) | \mathcal{F}_n],
\end{align*}
where  $A_{n,k}: =X_{n,k} (x\mathrm{Re}(Y_{n,k})+y\mathrm{Im}(Y_{n,k}) )$ for $1\leq k\leq t_n$. 


{\flushleft \bf Claim B: } One has 
\begin{align}\label{eqn-last-f}
\lim_{n\rightarrow\infty }\Big|\prod_{k=1}^{t_n}\mathbb{E}[\exp(itA_{n,k})|\mathcal{F}_n]-\prod_{k=1}^{t_n}\Big(1-\frac{c t^2}{4} X_{n,k}^2 (x^2+y^2)\Big)\Big|=0 \quad a.s.
\end{align}
 
 By \cite[Exercise 3.1.1]{Rick-Durrett},
we have 
\begin{align}\label{eqn-last-e}
\lim_{n\rightarrow\infty}\prod_{k=1}^{t_n}\Big(1-\frac{c t^2}{4} X_{n,k}^2(x^2+y^2)\Big)=e^{-\frac{M_\infty}{4}t^2(x^2+y^2)}\,\,a.s.\,\,\text{
where 
$
M_\infty=\lim_{n\rightarrow\infty}\sum_{i=1}^{t_n} X_{n,k}^2.
$
}
\end{align}
Then \eqref{eqn-last-f} and \eqref{eqn-last-e} imply that 
\[
\lim_{n\rightarrow\infty }\prod_{k=1}^{t_n}\mathbb{E}[\exp(itA_{n,k})|\mathcal{F}_n]=e^{-\frac{M_\infty}{4}t^2(x^2+y^2)}\quad a.s.
\]
This is the desired equality \eqref{eqn-desired}.

Now we prove \textbf{Claim B}. By \cite[Lemma 3.3.19]{Rick-Durrett}, one has
\begin{align*}
&\Big|\mathbb{E}[\exp(itA_{n,k})|  \mathcal{F}_n]-1+\frac{ct^2}{4}(x^2+y^2)X_{n,k}^2\Big|
\\
\leq& \mathbb{E}\Big[ \min\left(|tA_{n,k}|^{3},2|tA_{n,k}|^2\right) \big|\mathcal{F}_n\Big]\\
\leq &\mathbb{E}\Big[ |tA_{n,k}|^{3}\mathds{1}(|A_{n,k}|\leq\varepsilon)\big|\mathcal{F}_n\Big]+\mathbb{E}\Big[ 2|tA_{n,k}|^{2}\mathds{1}(|A_{n,k}|>\varepsilon)\big|\mathcal{F}_n\Big]\\
\leq&\varepsilon t^3\mathbb{E}\Big[ |A_{n,k}|^{2}\mathds{1}(|A_{n,k}|\leq\varepsilon)\big|\mathcal{F}_n\Big]+2t^2\mathbb{E}\Big[ |A_{n,k}|^{2}\mathds{1}(|A_{n,k}|>\varepsilon)\big|\mathcal{F}_n\Big].
\end{align*}
Hence it follows that 
\begin{align}\label{eqn-jkjkh}
&\sum_{k=1}^{t_n}\big|\mathbb{E} [\exp(itA_{n,k})|\mathcal{F}_n]-1+\frac{ct^2}{4} X_{n,k}^2 (x^2+y^2)\big|\\\nonumber
\leq &\varepsilon t^3\sum_{k=1}^{t_n}\mathbb{E}\Big[ |A_{n,k}|^{2}\mathds{1}(|A_{n,k}|\leq\varepsilon)\mid\mathcal{F}_n\Big]+2t^2\sum_{k=1}^{t_n}\mathbb{E}\Big[ |A_{n,k}|^{2}\mathds{1}(|A_{n,k}|>\varepsilon)\big|\mathcal{F}_n\Big].
\end{align}
Now for any fixed $x,y\in\mathbb{R}$ with $x^2+y^2>0$, one has
\begin{align}\label{eqn-AN-23}
A_{n,k}^2
\leq X_{n,k}^2((\mathrm{Re}(Y_{n,k}))^2+(\mathrm{Im}(Y_{n,k}))^2)(x^2+y^2)
=X_{n,k}^2|Y_{n,k}|^2(x^2+y^2).
\end{align}
This implies that 
\begin{align}\label{eqn-A_n-F_n}
\mathbb{E}[A_{n,k}^2 | \mathcal{F}_n]\leq X_{n,k}^2(x^2+y^2)\mathbb{E}[|Y|^2]=c(x^2+y^2)\cdot X_{n,k}^2.
\end{align}
By \eqref{eqn-AN-23}, we have
\begin{align*}
\mathbb{E}[ |A_{n,k}|^{2}\mathds{1}(|A_{n,k}|>\varepsilon)|\mathcal{F}_n]
\leq (x^2+y^2)\mathbb{E}\Big[|X_{n,k} Y_{n,k}|^2\mathds{1}\Big(|X_{n,k}Y_{n,k}|>\frac{\varepsilon}{\sqrt{x^2+y^2}}\Big) \Big| \mathcal{F}_n\Big].
\end{align*}
Hence \eqref{eqn-hjhi} implies that 
\[
\lim_{n\rightarrow\infty} \sum_{k=1}^{t_n} \mathbb{E}\left[A_{n,k}^2\mathds{1}(|A_{n,k}|>\varepsilon) | \mathcal{F}_n\right]=0.
\]
Combing with \eqref{eqn-A_n-F_n} and \eqref{eqn-jkjkh}, we have
\begin{align*}
&\limsup_{n\rightarrow\infty}\sum_{k=1}^{t_n}\big|\mathbb{E}[\exp(itA_{n,k}) | \mathcal{F}_n]-1+\frac{ct^2}{4} X_{n,k}^2 (x^2+y^2)\big|\\
\leq &\limsup_{n\rightarrow\infty}\varepsilon t^3 c(x^2+y^2)\sum_{k=1}^{t_n} X_{n,k}^2
=\varepsilon t^3 c(x^2+y^2)M_\infty \quad a.s.
\end{align*}
Since $\varepsilon>0$ is arbitrary, we have 
\begin{align}\label{eqn-important}
\limsup_{n\rightarrow\infty}\sum_{k=1}^{t_n}\big|\mathbb{E}[\exp(itA_{n,k}) | \mathcal{F}_n]-1+\frac{ct^2}{4} X_{n,k}^2 (x^2+y^2)\big|=0 \quad a.s.
\end{align}
Notice that by \eqref{eqn-AN-23}, one has
\[
\mathbb{E}[A_{n,k}^2 |\mathcal{F}_n]=\frac{c}{2}(x^2+y^2)X_{n,k}^2.
\]
Hence
\[
\frac{c}{2}(x^2+y^2)X_{n,k}^2=\mathbb{E}[A_{n,k}^2| \mathcal{F}_n]\leq \varepsilon^2+\mathbb{E}\left[A_{n,k}^2\mathds{1}(A_{n,k}>\varepsilon)| \mathcal{F}_n\right]
\]
which implies that 
\[
\lim_{n\rightarrow\infty}\frac{c}{2}(x^2+y^2)X_{n,k}^2\leq \varepsilon^2+\lim_{n\rightarrow\infty}\mathbb{E}\left[A_{n,k}^2\mathds{1}(A_{n,k}>\varepsilon) | \mathcal{F}_n\right]=\varepsilon.
\]
Since $\varepsilon>0$ is arbitrary,   for all $1\leq k\leq t_n$, we have
\begin{align}\label{eqn-jijiji}
\lim_{n\rightarrow\infty}\frac{c}{2}(x^2+y^2)X_{n,k}^2=0 \quad  a.s.
\end{align}
Hence almost surely, for $n$ large enough,
\[
\Big|1-\frac{ct^2}{4} X_{n,k}^2 (x^2+y^2)\Big|\leq 1. 
\] 
By    \cite[Lemma 3.4.3]{Rick-Durrett}, almost surely, for $n$ large enough, we have 
\begin{align*}
& \Big|\prod_{k=1}^{t_n}\mathbb{E}[\exp(itA_{n,k})|\mathcal{F}_n]-\prod_{k=1}^{t_n}\Big(1-\frac{ct^2}{4} X_{n,k}^2 (x^2+y^2)\Big)\Big|
\\
\leq & \sum_{k=1}^{t_n}\Big|\mathbb{E}[\exp(itA_{n,k}) | \mathcal{F}_n]-1+\frac{ct^2}{4} X_{n,k}^2 (x^2+y^2)\Big|.
\end{align*}
Combing with  \eqref{eqn-important}, we have
\[
\lim_{n\rightarrow\infty }\Big|\prod_{k=1}^{t_n}\mathbb{E}[\exp( itA_{n,k}) |\mathcal{F}_n]-\prod_{k=1}^{t_n}\Big(1-\frac{ct^2}{4} X_{n,k}^2 (x^2+y^2)\Big)\Big|=0 \quad a.s. 
\]
This is the desired equality \eqref{eqn-last-e}.
\end{proof}


\end{document}